\newif\ifshowflag
\showflagtrue   
\showflagfalse   
\newif\ifshowqstn
\showqstntrue   
\showqstnfalse   
\newif\ifshowinfo
\showinfotrue   
\showinfofalse   

\newcommand{\flag}[1]{\ifshowflag
  {\noindent\color{red}{$\clubsuit\clubsuit\clubsuit$\; {\sffamily #1}\; $\clubsuit\clubsuit\clubsuit$}}\fi}

\documentclass[12pt]{amsart}
\usepackage{amssymb,amscd}

\usepackage{graphicx}
\usepackage{xcolor}         

\numberwithin{equation}{section}        

\newcommand{\Va}{V\"ais\"al\"a}     

\newcommand{\HappyFace}{($\ddot{\smile}$)}

\catcode`\@=11       

\def\rf#1{\@rf{#1}#1:;;}
\def\rfs#1{\@rfs{#1}#1:;;}
\def\rfm#1{\@rfF#1<>;;}

\def\@C{C}\def\@CC{CC}\def\@E{E}\def\@F{F}\def\@L{L}\def\@P{P}\def\@PP{PP}\def\@Q{Q}
\def\@R{R}\def\@S{S}\def\@T{T}\def\@TT{TT}\def\@X{X}\def\@XX{XX}\def\@Ex{Ex}
\def\@s{s}\def\@ss{ss}\def\@f{f}

\def\@rf#1#2:#3;;{\def\@b{#2}
  \ifx\@b\@C Corollary~\ref{#1}\else%
  \ifx\@b\@CC Corollary~\ref{#1}\else%
  \ifx\@b\@E (\ref{#1})\else
  \ifx\@b\@Ex Exercise~\ref{#1}\else%
  \ifx\@b\@F Fact~\ref{#1}\else%
  \ifx\@b\@L Lemma~\ref{#1}\else%
  \ifx\@b\@P Proposition~\ref{#1}\else%
  \ifx\@b\@PP Proposition~\ref{#1}\else%
  \ifx\@b\@Q Question~\ref{#1}\else%
  \ifx\@b\@R Remark~\ref{#1}\else%
  \ifx\@b\@S Section~\ref{#1}\else%
  \ifx\@b\@T Theorem~\ref{#1}\else%
  \ifx\@b\@TT Theorem~\ref{#1}\else%
  \ifx\@b\@X Example~\ref{#1}\else%
  \ifx\@b\@XX Example~\ref{#1}\else%
  \ifx\@b\@s \S\ref{#1}\else
  \ifx\@b\@ss \S\ref{#1}\else
  \ifx\@b\@f Figure~\ref{#1}\else%
  \ref{#1}\fi\fi\fi\fi\fi\fi\fi\fi\fi\fi\fi\fi\fi\fi\fi\fi\fi\fi}

\def\@rfs#1#2:#3;;{\def\@b{#2}
  \ifx\@b\@C Corollaries~\ref{#1}\else%
  \ifx\@b\@CC Corollaries~\ref{#1}\else%
  \ifx\@b\@Ex Exercises~\ref{#1}\else%
  \ifx\@b\@F Facts~\ref{#1}\else%
  \ifx\@b\@L Lemmas~\ref{#1}\else%
  \ifx\@b\@P Propositions~\ref{#1}\else%
  \ifx\@b\@PP Propositions~\ref{#1}\else%
  \ifx\@b\@Q Questions~\ref{#1}\else%
  \ifx\@b\@R Remarks~\ref{#1}\else%
  \ifx\@b\@S Sections~\ref{#1}\else%
  \ifx\@b\@T Theorems~\ref{#1}\else%
  \ifx\@b\@TT Theorems~\ref{#1}\else%
  \ifx\@b\@X Examples~\ref{#1}\else%
  \ifx\@b\@XX Example~\ref{#1}\else%
  \ifx\@b\@s \S\S\ref{#1}\else
  \ifx\@b\@ss \S\S\ref{#1}\else
  \ifx\@b\@f Figures~\ref{#1}\else%
  \ref{#1}\fi\fi\fi\fi\fi\fi\fi\fi\fi\fi\fi\fi\fi\fi\fi\fi\fi}

\def\@rfF<#1>#2;;{\def\@c{#2}
  \@rfs{#1}#1:;;\ifx\@c\empty\else\@rfL:#2;;\fi}

\def\@rfL:#1<#2>#3;;{\def\@b{#2}\def\@c{#3}
  #1\ifx\@b\empty\else\ref{#2}\ifx\@c\empty\else\@rfL:#3;;\fi\fi}

\catcode`\@=12       

\definecolor{darkblue}{rgb}{0,0,0.6}
\definecolor{darkgreen}{rgb}{0,0.4,0}
\definecolor{darkred}{rgb}{0.6,0,0}
\definecolor{lightblue}{rgb}{0.8,0.8,1}
\definecolor{lightgreen}{rgb}{0.25,1,0.25}
\definecolor{lightred}{rgb}{1,0.5,0.5}
\definecolor{lightpurple}{rgb}{1,0.4,0.6}
\definecolor{darkpurple}{rgb}{0.5,0,0.5}

\newcommand{\ds}{\displaystyle} \newcommand{\half}{\frac{1}{2}}       \newcommand{\qtr}{\frac{1}{4}}
                \newcommand{\xra}{\xrightarrow}

\newcommand{\comp}{\circ}           	        
\newcommand{\sm}{\setminus}                     

\newcommand{\lp}{\left(}    \newcommand{\rp}{\right)}   
\newcommand{\lex}{\lesssim}   
\newcommand{\eqx}{\simeq}

\def\vint_#1{\mathchoice
          {\mathop{\vrule width 6pt height 3 pt depth -2.5pt
                  \kern -8pt \intop}\nolimits_{\kern -4pt#1}}%
          {\mathop{\vrule width 5pt height 3 pt depth -2.6pt
                  \kern -6pt \intop}\nolimits_{#1}}%
          {\mathop{\vrule width 5pt height 3 pt depth -2.6pt
                  \kern -6pt \intop}\nolimits_{#1}}%
          {\mathop{\vrule width 5pt height 3 pt depth -2.6pt
                   \kern -6pt \intop}\nolimits_{#1}}}

\newcommand{\ifff}{if and only if }  \newcommand{\wrt}{with respect to }  \newcommand{\param}{parametrization}

        \newcommand{\holo}{holomorphic}

\newcommand{\homeo}{homeomorphism}

\newcommand{\mob}{M\"obius}        \newcommand{\MT}{\mob\ transformation}

\newcommand{\alf}{\alpha}       \newcommand{\del}{\delta}   \newcommand{\Del}{\Delta}
\newcommand{\veps}{\varepsilon}  \newcommand{\gam}{\gamma}
\newcommand{\Gam}{\Gamma}       \newcommand{\kap}{\kappa}   \newcommand{\lam}{\lambda}
\newcommand{\Lam}{\Lambda}      \newcommand{\om}{\omega}    \newcommand{\Om}{\Omega}
\newcommand{\sig}{\sigma}       \newcommand{\Sig}{\Sigma}   \newcommand{\tha}{\theta}
\newcommand{\vth}{\vartheta}    \newcommand{\Th}{\Theta}
\newcommand{\ups}{\upsilon}     

\newcommand{\mcA}{{\mathcal A}}

\DeclareMathOperator{\id}{{\mathsf{id}}}        

\providecommand{\abs}[1]{\lvert#1\rvert}        

\DeclareMathOperator*\Star{\text{\Huge$\star$}}

\DeclareMathOperator{\md}{\mathsf{mod}}

\DeclareMathOperator{\diam}{\mathsf{diam}}
\DeclareMathOperator{\dist}{\mathsf{dist}}

\newcommand{\B}{\mathsf{B}}     
\newcommand{\D}{\mathsf{D}}     
\newcommand{\A}{\mathsf{A}}     

\DeclareMathOperator{\Arg}{Arg}
\DeclareMathOperator{\Log}{Log}

\newcommand{\mathfont}{\mathsf} 
\newcommand{\mfA}{{\mathfont A}}      %
\newcommand{\mfC}{{\mathfont C}}      
\newcommand{\mfD}{{\mathfont D}}      
\newcommand{\mfN}{{\mathfont N}}      
\newcommand{\mfR}{{\mathfont R}}      
\newcommand{\mfS}{{\mathfont S}}      
\newcommand{\mfZ}{{\mathfont Z}}      

\newcommand{\cmfD}{\bar{\mathfont D}}      
\newcommand{\RS}{\hat{\mfC}}     
\newcommand{\bd}{\partial}      
\newcommand{\bA}{{\partial A}}  \newcommand{\bB}{{\partial B}}  
\newcommand{\bD}{{\partial D}}    

\newcommand{\cA}{{\bar A}}      
\newcommand{\cB}{{\bar B}}      %
	\newcommand{\bOm}{{\partial\Omega}} 

\newtheorem{Thm}{Theorem}
\newtheorem{Cor}[Thm]{Corollary}        
\theoremstyle{remark}
\theoremstyle{plain}
\newtheorem*{thm*}{Theorem}         
\newtheorem*{lma*}{Lemma}           
\newtheorem*{cor*}{Corollary}
\newtheorem*{conj*}{Conjecture}
\newtheorem*{prop*}{Proposition}

\theoremstyle{remark}
\newtheorem*{claim*}{Claim}
\newtheorem*{xx*}{Example}
\newtheorem*{xxs*}{Examples}
\newtheorem*{fact*}{Fact}
\newtheorem*{qstn*}{Question}
\newtheorem*{rmk*}{Remark}
\newtheorem*{rmks*}{Remarks}

\swapnumbers                
\theoremstyle{plain}
\newtheorem{thm}[equation]{Theorem}
\newtheorem{lma}[equation]{Lemma}
\newtheorem{cor}[equation]{Corollary}
\newtheorem{prop}[equation]{Proposition}

\theoremstyle{remark}

\newtheorem{xx}[equation]{Example}

\newtheorem{fact}[equation]{Fact}

\newtheorem{qstn}[equation]{Question}

\newtheorem{rmk}[equation]{Remark}
\newtheorem{rmks}[equation]{Remarks}

  {\par\smallskip\noindent{\em #1}}{\par\smallskip}

\newenvironment{noname}[1]
  {\par\smallskip\noindent%
   \leftskip=\nnlen\rightskip=\nnlen\addtolength{\leftmargini}{\nnlen}%
   \em #1}%
  {\par\smallskip\addtolength{\leftmargini}{-\nnlen}}
\newlength{\nnlen}

  {\par\smallskip\noindent\refstepcounter{equation}\theequation.{ \em #1.}}%
  {\qed\smallskip}
\newenvironment{pf*}[1]{\subsubsection*{#1}}{\qed\smallskip}	

\newcounter{aenumctr} 
\newenvironment{aenum}
  {\begin{list}%
    {\rm(\alph{aenumctr})}
    {\usecounter{aenumctr}}
    \setlength{\rightmargin}{\leftmargin}}
  {\end{list}} 

\newcommand{\subann}{\mathrel{\ooalign{$\subset$\cr\hidewidth\raisebox{0.22ex}{\hspace{1mm}$\scriptstyle\rm a$}\hidewidth}}}

\newcommand{\csubann}{\mathrel{\ooalign{$\subset$\cr\hidewidth\raisebox{0.22ex}{\hspace{1mm}$\scriptstyle\rm c$}\hidewidth}}}

\newcommand{\BL}{bi-Lipschitz}

\newcommand{\bp}{\mathsf{bp}}
\newcommand{\BP}{\mathsf{BP}}
\newcommand{\BPt}{\textsf{BP}}
\newcommand{\BPEP}{\textsf{BPEP}}

\newcommand{\core}{\mathsf{core}}
\newcommand{\band}{\mathsf{band}}

\newcommand{\bdi}{\bd_{\rm in}}     
\newcommand{\bdo}{\bd_{\rm out}}    

\newcommand{\Coo}{{\mfC}_{01}}		
\newcommand{\Cab}{{\mfC}_{ab}}		
\newcommand{\Co}{{\mfC}_*}			
\newcommand{\Do}{{\mfD}_*}			

\newcommand{\kk}{\mathsf{k}}        
\newcommand{\SSS}{\mathsf{S}}       
\newcommand{\MM}{\mathsf{M}}        
\newcommand{\LL}{\mathsf{L}}        
\newcommand{\XL}{\mathsf{X}}        

\begin{document} 
\title[Hyperbolic \& Quasi-Hyperbolic Quasi-Geodesics]{Quasi-Hyperbolic Geodesics are\\Hyperbolic Quasi-Geodesics}
\date{\today}


\author{David A. Herron}
\address{Department of Mathematical Sciences, University of Cincinnati, OH 45221-0025, USA}
\email{David.Herron@UC.edu}

\author{Stephen M. Buckley}
\address{Department of Mathematics, National University of Ireland, Maynooth, Co.~Kildare, Ireland}
\email{sbuckley@maths.nuim.ie}

\thanks{ The first author was supported in part by the Charles Phelps Taft Memorial Fund.  The second author was supported in part by Enterprise Ireland and Science Foundation Ireland.  Both authors were supported by the US NSF grant DMS-1500454}




\begin{abstract}
This is a tale describing the large scale geometry of Euclidean plane domains with their hyperbolic or quasihyperbolic distances.
We prove that in \emph{any} hyperbolic plane domain, hyperbolic and quasihyperbolic quasi-geodesics are the same curves.
We also demonstrate the simultaneous Gromov hyperbolicity of such domains with their hyperbolic or quasihyperbolic distances.
\end{abstract}

\dedicatory{Dedicated to Alan Beardon, for his deep insights into hyperbolic geometry.}

\newcommand{\PV}{{{\small \em  preliminary version---please do not circulate}}}
\maketitle
%

\section{Introduction}  \label{S:Intro} 
Throughout this article $\Om$ denotes a hyperbolic plane domain: $\Om\subset\mfC$ is open and connected and $\Om^c:=\mfC\sm\Om$ contains at least two points.  Each such $\Om$ carries a unique maximal constant curvature -1 conformal metric $\lam\,ds=\lam_\Om ds$ usually referred to as the \emph{Poincar\'e hyperbolic metric} on $\Om$.  The length distance $h=h_\Om$ induced by $\lam\,ds$ is called \emph{hyperbolic distance} in $\Om$.  There is also a \emph{quasihyperbolic metric} $\del^{-1}ds=\del_\Om^{-1}ds$ on $\Om$, whose length distance $k=k_\Om$ is called \emph{quasihyperbolic distance} in $\Om$; here $\del(z)=\del_\Om(z):=\dist(z,\bOm)$ is the Euclidean distance from $z$ to the boundary of $\Om$.  See \rf{ss:h&k} and \rf{ss:hk-d&g} for more details.

A straightforward, albeit non-trivial, argument reveals that the metric spaces $(\Om,h)$ and $(\Om,k)$ are isometric \ifff $\Om$ is an open half-plane and the isometry is the restriction of a M\"obius transformation.  Furthermore, these metric spaces are bi-Lipschitz equivalent \ifff the identity map is bi-Lipschitz. \flag{should i give proofs?}

It is well-known that the identity map $(\Om,k)\xra{\id}(\Om,h)$ enjoys the following properties:
\begin{itemize}
  \item  The map $\id$ is a 2-Lipschitz \homeo.\footnote{In fact, the map $\id$ is always quasiconformal, but need not be quasisymmetric.}
  \item  For any simply connected $\Om$, $\id$ is 2-bi-Lipschitz.\footnote{That $\id^{-1}$ is 2-Lipschitz is a consequence of Koebe's One Quarter Theorem.}
  \item  In general, $\id$ is bi-Lipschitz \ifff $\Om^c$ is uniformly perfect.\footnote{This is quantitative: the bi-Lipschitz and uniformly perfectness constants depend only on each other.}
\end{itemize}
See, for instance, \cite{BP-beta} and \cite{Pomm-unifly-perfect1}.

However, for general hyperbolic plane domains $\Om$, there is no simple metric control on $\id^{-1}$.  For example, given \emph{any} sequences $(h_n)_1^\infty$ and $(k_n)_1^\infty$ of positive numbers with say $1\ge h_n\to0$ and $2\le k_n\to\infty$, there are sequences $(a_n)_1^\infty, (b_n)_1^\infty$ of points in the punctured unit disk $\mfD_*:=\mfD\sm\{0\}$ with hyperbolic and quasihyperbolic distance $h_*(a_n,b_n)=h_n$ and $k_*(a_n,b_n)=k_n$.  Also, here the identity map $(\D_*,k_*)\to(\D_*,h_*)$ fails to be quasisymmetric.  See \rf{X:D*} in \rf{ss:punxd D}.

Nonetheless, the hyperbolic and quasihyperbolic geodesics\footnote{See \rf{ss:hk-d&g} and \rf{s:P&Gs} for definitions and terminology.} in $\D_*$ appear quite similar.  Near the unit circle the two metrics are bi-Lipschitz equivalent\footnote{Indeed,
$\forall\;\half\le|z|<1, \; 1/\del_*(z)\le\lam_*(z)\le2/\del_*(z)$.},
so each geodesic in one space (with endpoints near the unit circle) is a quasi-geodesic in the other space.  Near the puncture: we can pull back both metrics via the exponential map; a hyperbolic geodesic in $\D_*$ pulls back to a circular subarc in the left-half-plane which has euclidean length at most $\pi/2$ times the euclidean distance between its endpoints; and, this translates into saying that the quasihyperbolic length of the original hyperbolic geodesic is at most $\pi/2$ times the quasihyperbolic distance between its endpoints.  In summary, we find that for any points $a,b$ in $\mfD_*$ and any hyperbolic geodesic $[a,b]_h$ in $\D_*$, we have
\[
  \ell_k\lp [a,b]_h \rp \le 11\, k(a,b)\,, \quad\text{so $[a,b]_h$ is a quasihyperbolic $11$-quasi-geodesic in $\D_*$}\,;
\]
and similarly quasihyperbolic geodesics in $\D_*$ are hyperbolic quasi-geodesics.  

Here we prove that this phenomenon holds for \emph{every} hyperbolic plane domain.  Thus, these two sometimes similar but often quite different metric spaces actually have quite similar geometry.  Based on the first author's numerous discussions with many non-believers at various times during the past decade, this is apparently a surprising fact!

\begin{Thm}  \label{TT:geos}
There are absolute constants $H_o$ and $K_o$ such that for any hyperbolic plane domain $\Om$, any pair of points $a,b\in\Om$, any hyperbolic geodesic $[a,b]_h$, and any quasihyperbolic geodesic $[a,b]_k$,
\[
  \ell_k\lp [a,b]_h \rp \le K_o\,k(a,b) \quad\text{and}\quad  \ell_h\lp [a,b]_k \rp \le H_o\, h(a,b)\,.
\]
Moreover, for each $\Lam\ge1$ there are explicit constants $H$ and $K$ that depend only on $\Lam$ such that for any hyperbolic $\Lam$-quasi-geodesic $\gam^h$ and any quasihyperbolic $\Lam$-quasi-geodesic $\gam^h$ both with endpoints $a$ and $b$,
\[
  \ell_k\lp \gam^h \rp \le K\,k(a,b) \quad\text{and}\quad  \ell_h\lp \gam^k \rp \le H\, h(a,b)\,.
\]
\end{Thm}                   

An immediate corollary of \rf{TT:geos} is that in any hyperbolic plane domain, the hyperbolic and quasihyperbolic quasi-geodesics are exactly the same curves, and this is quantitative.  Another easy consequence (see \rf{C:F&KPY}) is that the same sort of phenomenon holds for hyperbolic domains in the Riemann sphere, but here we replace the quasihyperbolic metric with either the Ferrand or Kulkarni-Pinkhall-Thurston metrics; again, the (hyperbolic, Ferrand, or Kulkarni-Pinkhall-Thurston) lengths of any of the (hyperbolic, Ferrand, or Kulkarni-Pinkhall-Thurston) quasi-geodesics are all comparable.

In addition, we establish the following marvelous fact concerning the simultaneous Gromov hyperbolicity of plane domains with their hyperbolic or quasihyperbolic distance geometry.  This further supports our assertion that these two geometries really are quite similar.

\begin{Thm}  \label{TT:GrHyp}
For any hyperbolic plane domain $\Om$, $(\Om,h)$ and $(\Om,k)$ are either both Gromov hyperbolic or both not Gromov hyperbolic.
\end{Thm}                   
\noindent
An obvious corollary is that the Gromov hyperbolicity of $(\Om,k)$ is a conformal invariant; this also follows from the fact (see \cite{GO-unif}) that quasiconformal maps of plane domains are quasihyperbolic rough quasiisometries.  Thanks to \cite[Theorem~1.12]{BHK-unif}, we find that $(\Om,h)$ is Gromov hyperbolic \ifff $\Om$ is conformally equivalent to an inner uniform slit domain.\footnote{The careful reader notices that spherical quasihyperbolic distance is used in \cite{BHK-unif}.}

In the special case where $\bOm$ lies in a Euclidean straight line, the conclusion of \rf{TT:GrHyp} was established in \cite{HPRT-GrHypDenjoy}.

\bigskip

We mention, briefly and vaguely, the main concepts in our arguments.  These ideas should prove useful in other circumstances.  First, and foremost, is a result of Beardon and Pommerenke \cite{BP-beta} who introduced the domain function $\bp=\bp_\Om$ (see \rf{ss:BP} and \eqref{E:BP est}) that gives quantitative estimates for the metric ratio $\lam\,\del=\lam\,ds/\del^{-1}ds$.  In the regions where $\bp$ is not large, $\lam\,ds$ and $\del^{-1}ds$ are \BL\ equivalent, so these are ``good'' points.  At a ``bad'' point, where $\bp$ is large, there is an annulus in $\Om$ that has large conformal modulus and which separates $\bOm\cup\{\infty\}$.  The role of these large fat separating annuli is the second crucial ingredient in our proofs.

Both hyperbolic and quasihyperbolic geodesics possess what we call the \emph{ABC} property: these geodesics only cross moderate size annuli at most once, so if they enter deep into an annulus, they either stay there or cross and never return.  See \rf{s:ABC} and especially \rf{P:ABC} and \rfs{R:ABC}.

Now let $[a,b]_h$ and $[a,b]_k$ be a hyperbolic geodesic and a quasihyperbolic geodesic with the same endpoints.  We may assume that there are some ``bad'' points.  These give us large fat separating annuli that \emph{both} geodesics must cross.  In the inner cores of these annuli, the geodesics are roughly close (both hyperbolically and quasihyperbolically).  Each large fat annulus gives a ``bad'' subarc of each geodesic---this is the subarc that crosses the middle inner core; the leftover subarcs are ``good''.

There are useful estimates for $\bp$ in large annuli, and from these we find that the ``bad'' subarcs have comparable lengths; e.g., the quasihyperbolic length of a ``bad'' hyperbolic subarc is comparable to the quasihyperbolic length of the corresponding ``bad'' quasihyperbolic subarc (which equals the quasihyperbolic distance between its endpoints).  See \rf{TT:bad??}.

Finally, the hyperbolic and quasihyperbolic lengths of each pair of associated ``good'' subarcs are comparable.  But, what must be demonstrated is that all four ``good'' lengths (i.e., the hyperbolic and quasihyperbolic lengths of both the ``good'' hyperbolic and ``good'' quasihyperbolic subarcs) are comparable.

\bigskip

It turns out that the ``bad'' regions---where $\bp$ is large---are actually geometrically quite simple.  Each point $z\in\Om$ with $\bp(z)$ large has a naturally associated maximal large fat separating annulus $\A(z)$ (see \rf{s:BP&A}) and $\bp$ is large everywhere in the middle core of this annulus.  Geometrically, this core is a long (pinched) cylinder in $(\Om,k)$ (in $(\Om,h)$, respectively).  It follows that any two (hyperbolic or quasihyperbolic) geodesics that join the boundary circles of some concentric subannulus that lies inside this core have comparable (hyperbolic and quasihyperbolic) lengths; the endpoints of the two geodesics need not coincide.


This phenomenon also holds for quasi-geodesics and plays a major role in our proofs of both \rfs{TT:geos} and \ref{TT:GrHyp}.  The following technical result says that any two (hyperbolic or quasihyperbolic) quasi-geodesics that join the boundary circles of some subannulus that lies deep inside a large fat separating annulus have comparable (hyperbolic and quasihyperbolic) lengths.  Note that here the endpoints of the two quasi-geodesics need not be the same.

\begin{Thm}  \label{TT:bad??}
For each $\Lambda\ge1$, there are explicit constants $Q, H, K$ that depend only on $\Lambda$ such that for any hyperbolic plane domain $\Om$, any annulus $A\in\mcA_\Om(2Q)$\footnote{See \rf{s:anns} and especially \rf{ss:cores} for information about annuli and their moduli and cores.}, and any subannulus $\Sig\csubann\core_Q(A)$ with $\md(\Sig)\ge1$: for any pair of $\Lambda$-quasi-geodesics $\gam_1$ and $\gam_2$ (these can both be hyperbolic quasi-geodesics, or both quasihyperbolic quasi-geodesics, or one of each) that join the boundary circles of $\Sig$,
\[                           
  K^{-1} \le \frac{\ell_k(\gam_1)}{\ell_k(\gam_2)} \le K \quad\text{and}\quad
    H^{-1} \le \frac{\ell_h(\gam_1)}{\ell_h(\gam_2)} \le H\,.
  \]                         
\end{Thm}                    
\noindent
The real significance of the above inequalities is for the case when $\gam_1$ and $\gam_2$ are not of the same type.  For example, if $a,b$ are points on separate boundary circles of $\Sig$, then \eqref{E:k* ests} and the proof of \rf{L:k ests in A} reveal that
\begin{gather*}
  m:=\md(\Sig)\le k(a,b) \le 2\pi+2m \le 2(1+\pi)m \le 9m
  \intertext{and thus, e.g., when $\gam_1,\gam_2$ are both quasihyperbolic $\Lam$-chordarc paths 
  we obtain}
  \frac19\,\Lam \le \frac{\ell_k(\gam_1)}{\ell_k(\gam_2)} \le 9\,\Lam\,.
\end{gather*}
\flag{There are/shud be similar estimates for hyperbolic lengths of hyperbolic quasi-geodesics, but these are more difficult to establish..!!}

From a geometric viewpoint, \rf{TT:bad??} is quite plausible.  Indeed, the endpoints of the quasi-geodesics are roughly close together, each quasi-geodesic crosses the same ``long cylinder'' in $\Om$, and inside this ``long cylinder'' we have good estimates for the metrics.

\bigskip

\rf{S:Prelims} contains the usual definitions and terminology; especially, see \rf{s:anns} for information about annuli and \rf{s:cfml metrics} for details about the hyperbolic and quasihyperbolic metrics.  Some technical issues that make our arguments work are discussed in \rf{S:Tools}.  We prove \rfs{TT:geos}, \ref{TT:GrHyp}, \ref{TT:bad??} in \rfs{S:geos pf}, \ref{S:Gromov bs}, \ref{S:lengths} respectively.

\bigskip

The first author is indebted to David Minda for years of interesting discussions. 
\tableofcontents
%
\section{Preliminaries}  \label{S:Prelims} 
%
\subsection{General Information}  \label{s:general} %
Our notation is relatively standard. 
We work in the Euclidean plane which we identify with the complex number field $\mfC$.  Everywhere $\Om$ is a domain (i.e., an open connected set) and $\Om^c:=\mfC\sm\Om$ and $\bOm$ denote the complement and boundary (respectively) of $\Om$.  We also always assume that $\Om^c\cup\{\infty\}$ contains at least three points; such an $\Om$ is usually dubbed a \emph{hyperbolic domain}.

We write $C=C(D,\dots)$ to indicate a constant $C$ that depends only on the data $D,\dots$.  Virtually all our constants are either absolute constants, or, they depend only on a quasi-geodesic constant $\Lam$.  In most cases, we provide explicit constants, although these are typically far from optimal.  In some cases we write $K_1\lex K_2$ to indicate that $K_1\le C\,K_2$ for some computable constant $C$ that depends only on the relevant data, and $K_1\eqx K_2$ means $K_1\lex K_2\lex K_1$.

The Euclidean line segment joining two points $a,b$ is $[a,b]$, and $(a,b)=[a,b]\sm\{a,b\}$.  The open and closed Euclidean disks, centered at the point $a\in\mfC$ and of radius $r>0$, are denoted by $\D(a;r)$ and $\D[a;r]$ respectively, and $\mfD:=\D(0;1)$ is the unit disk.  We also define
\[
  \Co:=\mfC\sm\{0\} \,,\quad \Cab:=\mfC\sm\{a,b\} \,,\quad \Do:=\mfD\sm\{0\} \,, \quad \Do(a;r):=\D(a;r)\sm\{a\}\,;
\]
the definition of $\Cab$ is for distinct points $a$, $b$ in $\mfC$.

The quantity $\del(z)=\del_\Om(z):=\dist(z,\bOm)$ is the Euclidean distance from $z\in\mfC$ to the boundary of $\Om$, and $1/\del$ is the scaling factor (aka, metric-density) for the so-called \emph{quasihyperbolic} metric $\del^{-1}ds$ on $\Om\subset\mfC$; see \rf{ss:h&k}.  We make frequent use of the notation
\begin{gather*}
  \D(z)=\D_{\Om}(z):=\D(z;\del(z))=\D(z;\del_\Om(z))
  \intertext{for the maximal Euclidean disk in $\Om$ centered at a point $z\in\Om$, and then}
  \B(z)=\B_\Om(z):=\bd\D(z)\cap\bOm = \mfS^1(z;\del(z))\cap\bOm
\end{gather*}
is the set of all nearest boundary points for $z$.


\medskip

We require the following numerical fact.

\begin{lma} \label{L:NumLma} %
Let $M>0$.  Suppose $r,s,t\in\mfR$ satisfy $r\ge M+\bigl(|t|\vee|s|\bigr)$ and $|t-s|\le M$.  Then
\[
  \half\, \bigl(r-|s|\bigr) \le  r-|t|  \le \,2 \bigl( r-|s| \bigr)\,.
\]
\end{lma} 
\begin{proof}%
Evidently,
\begin{gather*}
  r-|t|\le2(r-|s|) \iff |s|-|t|\le r-|s|
  \intertext{and since}
  r-|s|\ge M \ge |s|-|t|
\end{gather*}
the asserted upper right-hand inequality holds.  Interchanging $s$ and $t$ gives the lower left-hand inequality.
\end{proof}%

\subsection{Paths \& Geodesics}      \label{s:P&Gs} %
A \emph{path in $\Om$} is a continuous map $\mfR\supset I\xra{\gam}\Om$ where $I$ is an interval which may be closed or open or neither and finite or infinite.  The \emph{trajectory} of such a path $\gam$ is $|\gam|:=\gam(I)$ which we call a \emph{curve}; when $I$ is closed and $I\ne\mfR$, $\bd\gam:=\gam(\bd I)$ denotes the set of \emph{endpoints of $\gam$}, so $\bd\gam$ consists of one or two points depending on whether or not $I$ is compact.  We call $\gam$ a \emph{compact path} if its parameter interval $I$ is compact

An \emph{arc} $\alf$ is an injective compact path.  Given points $a,b\in|\alf|$, there are unique $s,t\in I$ with $\alf(s)=a$, $\alf(t)=b$ and we write $\alf[a,b]:=\alf\vert_{[s,t]}$; this notation is also meant to imply an orientation---in general, $a$ precedes $b$ on $\alf$, and then $s<t$.

For the most part, we are more interested in $|\gam|$ than $\gam$; that is, most of our concerns with paths do not depend on the actual parametrization of the underlying curve.  For example, see the discussion immediately below concerning quasi-geodesics and chordarc paths.

A map $I\xra{\gam}X$ into a metric space $(X,d)$ is: a \emph{geodesic} if it is an isometry (for all $s,t\in I$, $d(\gam(s),\gam(t))=|s-t|$), a \emph{$K$-quasi-geodesic} if it is $K$-\BL\ (aka, a $K$-quasi-isometry),
\begin{gather*}
  \text{for all $s,t\in I$}\,, \quad K^{-1}|s-t| \le d(\gam(s),\gam(t)) \le K|s-t|\,,
  \intertext{and a \emph{$(K,C)$-rough-quasi-geodesic} if}
  \text{for all $s,t\in I$}\,, \quad K^{-1}|s-t|-C \le d(\gam(s),\gam(t)) \le K|s-t| + C\,;
\end{gather*}
here $I\subset\mfR$ is an interval, although for rough-quasi-geodesics it is common to allow $I$ to be the intersection of $\mfZ$ with an interval.  Every quasi-geodesic (so every geodesic too) is an injective path (so, an arc if compact), but rough-quasi-geodesics need not be continuous.  Nonetheless, every rough-quasi-geodesic can be ``tamed'' (see \cite[p.\ 403]{Brid-Haef}) meaning that it can be replaced by a nearby continuous rough-quasi-geodesic.

A characteristic property of geodesics is that the length of each subpath equals the distance between its endpoints.  There is a corresponding description for quasi-geodesics:
a path $I\xra{\gam}X$ in a metric space $(X,d)$ is an \emph{$L$-chordarc path} if it is rectifiable and
\[
  \text{for all $s,t\in I$}\,, \quad \ell_d(\gam\vert_{[s,t]}) \le L\,d(\gam(s),\gam(t))\,.
\]
(One can also introduce the notion of an \emph{$(L,C)$-rough-chordarc path}.)

If we ignore parameterizations, then the class of all quasi-geodesics (in some metric space) is exactly the same as the class of all chordarc paths.  More precisely, a $K$-quasi-geodesic is a $K^2$-chordarc path, and if we parameterize an $L$-chordarc path \wrt arclength, then we get an $L$-quasi-geodesic.  (Similarly, each ``tamed'' rough-quasi-geodesic is a rough-chordarc path, and every rough-chordarc path is a rough-quasi-geodesic.)

\smallskip

In this paper, the metric space $(X,d)$ will always be either $(\Om,h)$ or $(\Om,k)$ where $\Om$ is a hyperbolic plane domain and $h$ and $k$ are the hyperbolic and quasihyperbolic distances in $\Om$.  The geodesics and quasi-geodesics in $(\Om,h)$ are called \emph{hyperbolic} geodesics and \emph{hyperbolic} quasi-geodesics, and similarly in $(\Om,k)$ we attach the adjective \emph{quasihyperbolic}.  See \rf{ss:h&k} and \rf{ss:hk-d&g}.

\smallskip

We employ the following construction in several circumstances.  Given two points $a,b$ on some circle $C$, we write $\kap_{ab}$ for the \emph{shorter circular arc in $C$ from $a$ to $b$}; when $a,b$ are diametrically opposite, $\kap_{ab}$ can be either semi-circle.  Next, suppose $a,b$ are points that each lie on one of the boundary circles of an annulus $A$; say $a$ lies in $C_a$ and $b$ lies in $C_b$ where $\bA=C_a\cup C_b$.  Let $c$ and $d$ be the radial projections (\wrt the center of $A$) of $b$ onto $C_a$ and $a$ onto $C_b$,  respectively.  Then
\begin{subequations}\label{E:chi}
\begin{gather}
   \chi_{ab}:=\kap_{ac}\star[c,b] \quad\text{is the \emph{circular-arc-segment path} from $a$ through $c$ to $b$.}  \label{E:chi-ab}
   \intertext{and}
   \chi_{ba}:=\kap_{bd}\star[d,a] \quad\text{is the \emph{circular-arc-segment path} from $b$ through $d$ to $a$.}  \label{E:chi-ba}
\end{gather}
\end{subequations}
Note that we also have the reverse paths, $\chi^{-1}_{ab}=[b,c]\star\kap_{ca}$ and $\chi^{-1}_{ba}=[a,d]\star\kap_{db}$ which are \emph{segment-circular-arc} paths from $b$ through $c$ to $a$ and from $a$ through $d$ to $b$ respectively.  See the proofs of \rf{L:k ests in A}, \rfs{P:D*:ell-bad-arcs} and \ref{P:A:ell-bad-arcs}, and especially \rfs{TT:geos}, \ref{TT:GrHyp} and \rf{L:CA surgery} where we make extensive use of such circular-arc-segment paths.

The next estimate is useful in a number of contexts.  In \rf{L:k ests in A} we show that, under appropriate hypotheses on the location of the points $a$ and $b$, $\chi_{ab}$ is a quasihyperbolic quasi-geodesic (and so, according to \rf{TT:geos}, also a hyperbolic quasi-geodesic).

\begin{lma} \label{L:ell chi} %
Let $a,b\in\Co$ with $0<|a|\le|b|$.  Let $\chi:=\chi_{ab}$ be the circular-arc-segment path from $a$ to $c:=\bigl(|a|/|b|\bigr)b$ to $b$ as described above.  Then
\[
  |a-b| \le \ell(\chi) \le 3|a-b| \,.
\]
If $|b|\le e|a|$, then also $\ell(\chi)\le 2e|a|$.
\end{lma} 
\begin{proof}%
Let $\kap:=\kap_{ac}$ so that $\chi=\kap\star[c,b]$.  Put $\tha:=|\Arg(b/a)|=|\Arg(c/a)|$.  According to the Law of Cosines,
\begin{gather*}
  |a-c|^2=2|a|^2(1-\cos\tha)\,, \quad\text{so}\quad \cos\tha=1-\frac{|a-c|^2}{2|a|^2}
  \intertext{and therefore}
  |a-b|^2=|a|^2+|b|^2-2|a||b|\cos\tha=|a|^2+|b|^2-2|a||b|\biggl(1-\frac{|a-c|^2}{2|a|^2} \biggr) \\
    = |a|^2+|b|^2-2|a||b|+|b|\frac{|a-c|^2}{2|a|} = \bigl(|a|-|b|\bigr)^2 + \frac{|b|}{|a|}\,|a-c|^2 \ge |a-c|^2\,.
  \intertext{Thus}
  \ell(\kap)\le\frac{\pi}2\,|a-c|\le\frac{\pi}2\,|a-b|
  \intertext{and hence}
  \ell(\chi)=\ell(\kap)+\bigl(|b|-|a|\bigr)\le\bigl(\frac{\pi}2\,+1\bigr)|a-b|\,. 
\end{gather*}

If $|b|\le e|a|$, then as $\ell(\kap)\le\pi|a|$ we have $\ell(\chi)\le (\pi+e-1)|a|\le 2e|a|$.
\end{proof}%

\subsection{Annuli, Cores, \& Collars}      \label{s:anns} %
The Euclidean annulus $A:=\{z\in\mfC :  r<|z-o|<R\}$ has \emph{center} $c(A):=o\in\mfC$ and \emph{conformal modulus} $\md(A):=\log(R/r)$; if $r=0$ or $R=\infty$, we say that $A$ is a \emph{degenerate annulus} and set $\md(A):=\infty$.\footnote{We also define $\md(\cA)=\md(A)$}   We call $\mfS^1(A):=\mfS^1(o;\sqrt{rR})$ the \emph{conformal center circle} of $A$; $A$ is symmetric about this circle.  The \emph{inner} and \emph{outer boundary circles} of $A$ are, respectively,
\[
  \bdi A:=\mfS^1(o;r) \quad\text{and}\quad  \bdo A:=\mfS^1(o;R)\,.
\]
A point $z$ is \emph{inside} (resp., \emph{outside}) $A$ \ifff $z\in\B[o;r]$ (resp., $z\in\mfC\sm\B(o;R)$); that is, $z$ is inside (or outside) $A$ \ifff $z$ is inside $\bdi A$ (or outside $\bdo A$).

It is convenient to introduce the notation
\begin{gather*}
  A=\A(o;d,m):=\{z\in\mfC : d\,e^{-m}<|z-o|<d\,e^m\}
  \intertext{and}
  \A[o;d,m]:=\{z\in\mfC : d\,e^{-m}\le |z-o|<d\le e^m\}\,.
  \intertext{Then}
  \mfS^1(A)=\mfS^1(o;d)\,,\quad \bdi A=\mfS^1(o;d\,e^{-m})\,, \quad  \bdo A=\mfS^1(o;d\,e^m)\,.
\end{gather*}
In the above, $o=c(A)$, $d>0$, $m>0$\footnote{However, it is useful to define $\A(o;d,0):=\mfS^1(o;d)$ and to say that this `thin annulus' has \emph{zero} modulus.}, and $\md(A)=2m$.

We call $A':=\{z\in\mfC :  r'<|z-o'|<R'\}$ a \emph{subannulus of $A$} provided
\[
  A'\subset A \quad\text{and}\quad  o'\in\B[0;r]\subset\B(o';r')\,;
\]
when this holds, we write $A'\subann A$.  We say that two annuli are \emph{concentric} if they have a common center.  We say that $A'$ is a \emph{concentric subannulus of $A$}, denoted by $A'\csubann A$, provided $A'$ and $A$ are concentric and $A'\subann A$.

An annulus $A$ \emph{separates} $E$ if each of the components of $A^c$ contains points of $E$; thus when $A$ separates $\{a,b\}$, one of $a$ or $b$ lies inside $A$ and the other lies outside $A$, and if $A$ does not meet nor separate $\{a,b\}$, then $a$ and $b$ are \emph{on the same side} of $A$.  Evidently, if $A' \subann A$, then $A'$ separates the boundary circles of $A$.

We define
\begin{align*}
  \mcA_\Om     &:=\{A  :  \text{$A$ is an Euclidean annulus in $\Om$ with $c(A)\in\Om^c$}\}\,, \\
  \mcA(m)      &:=\{A  :  \text{$A$ is an Euclidean annulus with $\md(A)>m$}\}\,, \\
  \mcA(a,b)    &:=\{A  :  \text{$A$ is an Euclidean annulus that separates $\{a,b\}$}\}\,,
  \intertext{and then}
  \mcA^1_\Om   &:=\{A\in\mcA_\Om  :  \bA\cap\bOm\ne\emptyset\}\,,  \\
  \mcA^2_\Om   &:=\{A\in\mcA_\Om  :  \bdi A\cap\bOm\ne\emptyset\ne\bdo A\cap\bOm\}\,,
  \intertext{and}
  \mcA_\Om(m)     &:=\mcA_\Om\cap\mcA(m) \quad\text{and}\quad \mcA_\Om(a,b)=\mcA_\Om\cap\mcA(a,b)\,\\
  \mcA^2_\Om(a,b) &:=\mcA^2_\Om\cap\mcA(a,b) \quad\text{and}\quad \mcA^2_\Om(a,b;m):=\mcA^2_\Om(a,b)\cap\mcA(m)\,.
\end{align*}
The requirement $c(A)\in\Om^c$ means, essentially, that $A$ separates $\Om^c\cup\{\infty\}$.

\subsubsection{Crossing an Annulus}  \label{ss:cross}%
Let $A$ be an annulus and $\gam$ be a path.  Then $\gam$ \emph{meets} or \emph{misses $A$} depending on whether $|\gam|\cap A\ne\emptyset$ or $|\gam|\cap A=\emptyset$, and similarly, $\gam$ \emph{crosses $A$} if
\[
  |\gam|\cap\bdi A\ne\emptyset\ne|\gam|\cap\bdo A
\]
and $\gam$ \emph{does not cross $A$} if one of the above intersections is empty.  Next, $\gam$ \emph{crosses $A$ $n$ times} if there are non-overlapping subpaths $\gam_1,\dots,\gam_n$\footnote{In this case we can arrange that the terminal point of $\gam_{i-1}$ and the initial point of $\gam_i$ both lie on the same side of $A$.} of $\gam$ such that each $\gam_i$ crosses $A$, and \emph{$\gam$ crosses $A$ exactly $n$ times} if it crosses $n$ times but not $n+1$ times.  Evidently, if $A$ separates the endpoints of $\gam$, then $\gam$ crosses $A$.  Also, we say that $\gam$ \emph{bounces off $A$} if
\[
  |\gam|\cap A \ne\emptyset = |\gam|\cap\mfS^1(A)= \bd\gam\cap A \,.
\]
\flag{We could introduce a parameter $q$ and ask that $|\gam|$ not meet the middle $q$-core of $A$.  I originally used this terminology as part of the ABC property, but now?  I guess that we do not use this any
where.}

\subsubsection{Annulus Cores \& Collars}  \label{ss:cores}%
For a non-degenerate annulus $A$ and $m\in(0,\half\md(A))$, there is a unique concentric subannulus of $A$ that has modulus $2m$ and center circle $\mfS^1(A)$.  Such a subannulus is an example of an \emph{inner core} of $A$, and the complement (in $A$) of its closure is the union of two concentric subannuli of $A$, each of modulus $\half\md(A)-m$, that we call its (\emph{inner} and \emph{outer}) \emph{collars} relative to $A$.

Controlling the size of these collars is crucial for our purposes.  Given $q\in(0,\half\md(A))$, we define $\core_q(A)$---which we call the \emph{core of $A$ with collar parameter $q$}---to be the concentric subannulus of $A$ with the same center circle as $A$ and such that $\md(\core_q(A))=\md(A)-2q$.  Thus $A\sm\overline\core_q(A)$ is the union of two concentric subannuli of $A$ each of modulus $q$; these are the (\emph{inner} and \emph{outer}) \emph{collars of $\core_q(A)$ relative to $A$}.

We often describe $\core_q(A)$ as the \emph{middle $q$-core} of $A$, but one must remember that $q$ is the collar parameter; that is, $q$ gives the size of the collars of the core, relative to $A$.  For example, the $\log2$-core of $\{r<|z|<R\}$ is $\{2r<|z|<R/2\}$, and its $q$-core is $\{re^{q}<|z|<Re^{-q}\}$.  When $A=\A(o;d,m)$ and $0<q<m$,
\[
  \core_q(A)=\A(o;d,m-q)  \quad\text{and}\quad  \overline\core_q(A)=\A[o;d,m-q]\,.
\]
\flag{how about some pix!}

It is useful to have a similar notion for degenerate annuli.  To this end, we define
\[
  \core_q(A):=
    \begin{cases}
      \D_*(o;e^{-q}R)       &\text{when $A=\D_*(o;R)$}\,,       \\
      \mfC\sm\D[o;e^{q}R]   &\text{when $A=\mfC\sm\D[o;R]$}\,.  \\
    \end{cases}
\]
Here in both cases $\core_q(A)\csubann A$ but now $A\sm\overline\core_q(A)$ is a single annulus of modulus $q$ (and $q>0$ can be arbitrary).

We adopt the convention that whenever we write $\core_q(A)$ we have $0<q<\half\md(A)$.  With this in mind,
note that whenever $A\csubann B$, $\core_q(A)\csubann\core_q(B)$.  Also, $q<r$ implies $\overline\core_r(A)\csubann\core_q(A)$ and for $q>0$ and $r>0$, $\core_q\core_r(A)=\core_{q+r}(A)=\core_r\core_q(A)$.

It is convenient to have the notion of the core of a closed annulus, so we define $\core_q(\cA):=\overline\core_q(A)$.  Thus the core of an open (or closed) annulus is a concentric open (or closed) subannulus.

Occasionally it is useful to have a dual (or inverse) notion of a core.  The \emph{$r$-banding} of an annulus $A$ is the unique concentric superannulus $\band_r(A)\supset A$ with the property that $\core_r\band_r(A)=A$.  Again, $\band_r(A)$ is an open or closed annulus depending upon whether $A$ is open or closed.  We also define $\band_r\bigl(\mfS^1(o;d)\bigr):=\A(o;d,r)$ (and we recognize that here the $r$-banding of a ``closed zero modulus annulus'' is an open annulus).  Evidently, for all $r>0$,
\begin{align*}
  \band_r\A(o;d,m) &= \A(o;d,m+r)\,, \\
  \md\band_r(A) &= \md(A)+2r\,,      \\
  \core_r\band_r(A)&=A=\band_r\core_r(A)\,,
\end{align*}
and in fact $\core_q\band_r(A)=\band_r\core_q(A)=\core_t\band_s(A)$; here the final equality requires $r-q=s-t$ (amongst other conditions too), and we leave it to the reader to recognize the various restrictions on the core parameters $q,r,s,t$.

\medskip

We tacitly make use of the following elementary fact.  See especially \rf{L:k ests in A}.

\begin{lma} \label{L:del deep in A} %
Suppose $A:=\A(o;d,m)\in\mcA_\Om(\log4)$ and $z\in\core_{\log2}(A)$; so $m>\log2$ and $z\in\A(o;d,m-\log2)$.  Then
\[
  \B(z)\subset\D[o;de^{-m}] \quad\biggl(\text{i.e., for all $\zeta\in\bOm$ with $\del(z)=|z-\zeta|$, $|\zeta-o|\le d e^{-m}$}\biggr)
\]
and moreover $\ds\half\,|z-o|\le\del(z)\le|z-o|$.
\end{lma} 
\begin{proof}%
Assume $o=0$ and $d=1$.  Then $\del(z)\le|z|<e^{m-\log2}=e^m-e^{m-\log2}$, so any $\zeta\in\B(z)$ must lie outside $\A(o;d,m)$ and hence $|\zeta|\le e^{-m}=de^{-m}$.  Therefore
\[
  \del(z)=|z-\zeta|\ge|z|-e^{-m}\ge\half\,|z|\,.  \qedhere
\]
\end{proof}%

Below is a useful consequence of \rf{L:del deep in A}; in particular, it says that each of the centers of the two annuli involved lie inside \emph{both} of the annuli. 

\begin{lma} \label{L:both centers inside} %
Assume $r\wedge s>q\ge\log2$.  Let $\A(\eta;d,r),\A(\vth;c,s)\in\mcA_\Om$.  Suppose that $\A(\eta;d,r-q)\cap\A(\vth;c,s-q)\ne\emptyset$.\footnote{These two annuli are the $q$-cores of $\A(\eta;d,r)$ and $\A(\vth;c,s)$, respectively.}  Then
\begin{gather*}
  |\eta-\vth|\le(de^{-r})\wedge(ce^{-s})\,.
  \intertext{If in addition $\A(\eta;d,r),\A(\vth;c,s)\in\mcA^2_\Om$, then also}
  \half\,d \le c \le 2\,d  \quad\text{and}\quad |r-s| \le \log4\,.
\end{gather*}
\end{lma} 
\begin{proof}%
Let $z\in\A(\eta;d,r-q)\cap\A(\vth;c,s-q)$ and pick any $\zeta\in\B(z)$.  Then by \rf{L:del deep in A},
\[
  |\zeta-\eta|\le de^{-r} \quad\text{and}\quad  |\zeta-\vth|\le ce^{-s}\,.
\]

Suppose $|\eta-\vth|>de^{-r}$.  Then, since $\vth\in\bOm$ and $\A(\eta;d,r)\subset\Om$, $|\eta-\vth|\ge de^{r}$; but now $\vth$ far from $\eta$ means it is also far from $\zeta$ which will cause trouble!  We have
\begin{gather*}
  ce^{-s} \ge |\zeta-\vth| \ge |\eta-\vth|-|\eta-\zeta| \ge de^r-de^{-r} \ge \half\, d e^r
  \intertext{so $de^r\le2ce^{-s}$.  Then, as $|z-\zeta|=\del(z)\le|z-\eta| < de^{r-q}$,}
  ce^{-s+q} < |z-\vth| \le |z-\zeta|+|\zeta-\vth| \le d e^{r-q}+ce^{-s} \le (1+2e^{-q})ce^{-s}
  \intertext{which implies that $e^q<1+2e^{-q}$.  However, as $q\ge\log2$, this in turn implies the contradiction}
  2\le e^q<1+2e^{-q}\le2\,.
\end{gather*}

Thus $|\eta-\vth|\le de^{-r}$.  Interchanging the roles of $\eta,d,r$ and $\vth,c,s$ we likewise deduce that $|\eta-\vth|\le ce^{-s}$.

\medskip

Now suppose also that $\A(\eta;d,r),\A(\vth;c,s)\in\mcA^2_\Om$.  Select $\xi_1,\xi_2\in\bOm\cap\bd\A(\eta;d,r)$ with
\addtocounter{equation}{1}  
\begin{gather*}
  |\xi_1-\eta|=de^{-r}  \quad\text{and}\quad  |\xi_2-\eta|=de^{r}\,.
  \intertext{Below we verify that $\xi_1$ and $\xi_2$ are, respectively, inside and outside $\A(\vth;c,s)$; i.e.,}
  |\xi_1-\vth|\le ce^{-s}  \quad\text{and}\quad  |\xi_2-\vth|\ge ce^{s}\,.   \tag{\theequation}\label{E:in n out}
\end{gather*}
Assuming \eqref{E:in n out} (for now), we proceed as follows.

Since $\xi_1,\eta\in\D[\vth;ce^{-s}]$, $de^{-r}=|\xi_1-\eta|\le2ce^{-s}$.  Also,
\begin{gather*}
  ce^s \le |\xi_2-\vth| \le |\xi_2-\eta|+|\eta-\vth| \le de^r+de^{-r} \le 2de^r\,.
  \intertext{Thus}
  ce^s \le 2de^r \quad\text{and}\quad de^{-r} \le 2ce^{-s}\,.
  \intertext{Interchanging the roles of $\eta,d,r$ and $\vth,c,s$ we likewise deduce}
  de^r \le 2ce^s \quad\text{and}\quad ce^{-s} \le 2de^{-r}\,.
  \intertext{Therefore,}
  d^2=(de^r)(de^{-r})\le(2ce^s)(2ce^{-s})=4c^2\,, \quad\text{so $d\le2c$}.
  \intertext{Similarly, $c\le2d$.  Also,}
  e^{s-r}\le 2(c/d) \le 4 \quad\text{and}\quad e^{r-s}\le4\,.
\end{gather*}

\medskip
It remains to verify \eqref{E:in n out}.  Since $\xi_i\in\bOm$ while $\A(\vth;c,s)\subset\Om$, either $|\xi_i-\vth|\le ce^{-s}$ or $|\xi_i-\vth|\ge ce^{s}$.
Recall that $z\in\A(\eta;d,r-q)\cap\A(\vth;c,s-q)$.

Suppose $|\xi_2-\vth|\le ce^{-s}$.  Then $\xi_2,\eta\in\D[\vth;ce^{-s}]$, so $de^r=|\xi_2-\eta|\le2ce^{-s}$.  Thus
\begin{gather*}
  |z-\eta| < de^{r-q} \le 2ce^{-s-q}
  \intertext{whence}
  c e^{q-s} < |z-\vth| \le |z-\eta|+|\eta-\vth| < 2ce^{-s-q} + ce^{-s}
\end{gather*}
which implies the contradiction $e^q<1+2e^{-q}$ (as described above).  Thus, $|\xi_2-\vth|\ge ce^{s}$.

Now suppose $|\xi_1-\vth|\ge ce^{s}$.  Then
\begin{gather*}
  de^{-r} = |\xi_1-\eta| \ge |\xi_1-\vth|-|\vth-\eta| \ge ce^{s} - ce^{-s} \ge \half\,ce^s\,.
  \intertext{Therefore}
  \half\, c e^{s+q} \le de^{q-r} < |z-\eta| \le |z-\vth|+|\vth-\eta| < ce^{s-q} + ce^{-s}
\end{gather*}
which again implies the contradiction $e^q<1+2e^{-q}$.  So, $|\xi_1-\vth|\le ce^{-s}$.
\end{proof}%

\begin{cor} \label{C:cores disjoint} %
Assume $\sig>0$ and $r\wedge s\ge\tau\ge\log3\vee(\sig+\frac13\log3)$.  Let $A:=\A(\eta;d,r), B:=\A(\vth;c,s)\in\mcA_\Om$.  Suppose there exist points $a,b$ with
\begin{center}
  $a$ inside \& $b$ outside $\core_\sig(A)$ \hspace{1em}and\hspace{1em} $b$ inside \& $a$ outside $\core_\sig(B)$\,.
\end{center}
Then $\core_{\log2}(A)\cap\core_{\log2}(B)=\emptyset$\,.
\end{cor} 
\begin{proof}%
Suppose that
\begin{gather*}
  \core_{\log2}(A)\cap\core_{\log2}(B)\ne\emptyset \quad\text{(i.e., $\A(\eta;d,r-\log2)\cap\A(\vth;c,s-\log2)\ne\emptyset$)}\,.
  \intertext{Appealing to \rf{L:both centers inside} we can assert that}
  |\eta-\vth|\le(d e^{-r})\wedge(c e^{-s}) \le (c\wedge d) e^{-\tau} \le \frac{c}3\,.
  \intertext{Then, as $b$ is outside $\core_\sig(A)$ and inside $\core_\sig(B)$,}
  d e^{\tau-\sig} \le d e^{r-\sig} \le |b-\eta| \le |b-\vth|+|\eta-\vth| \le c e^{\sig-s} + |\eta-\vth|\,,
  \intertext{so}
  d\le c e^{2\sig-\tau-s}+e^{\sig-\tau}|\eta-\vth| \le c e^{2(\sig-\tau)} + e^{\sig-\tau}|\eta-\vth|\,.
\end{gather*}
Therefore, since $a$ is inside $\core_\sig(A)$,
\begin{align*}
  |a-\vth| &\le |a-\eta|+|\eta-\vth| \le d e^{\sig-r} + |\eta-\vth|  \\
           &\le e^{\sig-\tau} \lp c e^{2(\sig-\tau)} + e^{\sig-\tau}|\eta-\vth| \rp + |\eta-\vth|   \\
           &= c e^{3(\sig-\tau)} + e^{2(\sig-\tau)}|\eta-\vth| + |\eta-\vth| < c\,.
\end{align*}
However, as $a$ lies outside $\core_\sig(B)$, we now have $c>|a-\vth|\ge c e^{s-\sig}$, and this implies the contradiction $\tau>\sig>s\ge\tau$.
\end{proof}%

\subsection{Conformal Metrics}      \label{s:cfml metrics} %
We call $\rho\,ds=\rho(z)|dz|$ a \emph{conformal metric} on $\Om$ when $\rho$ is some positive continuous function defined on $\Om$.  Below we consider the hyperbolic and quasihyperbolic metrics.  Each conformal metric $\rho\,ds$ induces a length distance $d_\rho$ on $\Om$ that is defined by
\[
  d_\rho(a,b):=\inf_\gam \ell_\rho(\gam) \quad\text{where}\quad \ell_\rho(\gam):= \int_\gam \rho\, ds
\]
and where the infimum is taken over all rectifiable paths $\gam$ in $\Om$ that join the points $a,b$.  We call $\gam$ a \emph{$\rho$-geodesic} if $d_\rho(a,b)=\ell_\rho(\gam)$\footnote{Note that this agrees with the definition given in \rf{s:P&Gs}.}; these need not be unique.  We often write $[a,b]_\rho$ to indicate a $\rho$-geodesic with endpoints $a,b$, but one must be careful with this notation since these geodesics need not be unique.  When we have a point $z$ on a given fixed geodesic $[a,b]_\rho$, we write $[a,z]_\rho$ to mean the subarc of the given geodesic from $a$ to $z$.

We note that the ratio $\rho\,ds/\sig\,ds$ of two conformal metrics is a well-defined positive function on $\Om$.  We write $\rho\le C\,\sig$ to indicate that this metric ratio is bounded above by $C$.

\subsubsection{Hyperbolic and QuasiHyperbolic Metrics}    \label{ss:h&k} %
Every hyperbolic plane domain carries a unique metric $\lam\,ds=\lam_\Om ds$ which enjoys the property that its pullback $p^*[\lam\,ds]$, \wrt any holomorphic universal covering projection $p:\D\to\Om$, is the hyperbolic metric $\lam_\D(\zeta)|d\zeta|=2(1-|\zeta|^2)^{-1}|d\zeta|$ on $\D$.  In terms of such a covering $p$, the metric-density $\lam=\lam_\Om$ of the {\em Poincar\'e hyperbolic metric\/} $\lam_\Om ds$ can be determined from
$$
  \lam(z)=\lam_\Om(z)=\lam_\Om(p(\zeta))=2(1-|\zeta|^2)^{-1}|p'(\zeta)|^{-1}\,.
$$
Yet another description is that $\lam\,ds$ is the unique maximal (or unique complete) metric on $\Om$ that has constant Gaussian curvature $-1$.

For example, the hyperbolic metric $\lam_*ds$ on the punctured unit disk $\D_*:=\D\sm\{0\}$ can be obtained by using the universal covering $z=\exp(w)$ from the left-half-plane onto $\D_*$ and we find that
\[
  \lam_*(z)|dz|=\frac{|dz|}{|z|\bigl|\log|z|\bigr|}\,.
\]

Except for a short list of special cases, the actual calculation of any given hyperbolic metric is notoriously difficult; computing hyperbolic distances and determining hyperbolic geodesics is even harder.  Indeed, one can find a number of papers analyzing the behavior of the hyperbolic metric in a twice punctured plane.  Typically one is left with estimates obtained by using domain monotonicity and considering `nice' sub-domains and super-domains in which one can calculate, or at least estimate, the metric.

In many cases (but certainly not all), the hyperbolic metric is \BL\ equivalent to the \emph{quasihyperbolic metric} $\del^{-1}ds=|dz|/\del(z)$, where $\del=\del_\Om$, which is defined for any proper subdomain $\Om\subsetneq\mfC$.  This metric has proven useful in many areas of geometric analysis.  The quasihyperbolic metric in the punctured plane $\Co$ is simply $|dz|/|z|$, which classically was called the logarithmic metric; indeed, the \holo\ covering $\mfC\xra{\exp}\Co$ pulls back the quasihyperbolic metric $\del_*^{-1}ds$ on $\Co$ to the Euclidean metric on $\mfC$ and thus $|\Log(a/b)|$ is the quasihyperbolic distance between points $a,b\in\Co$.

\subsubsection{Hyperbolic \& QuasiHyperbolic Distances \& Geodesics} 	\label{ss:hk-d&g} %
The \emph{hyperbolic distance} $h=h_\Om$ and \emph{quasihyperbolic distance} $k=k_\Om$ in $\Om$ are the length distances $h_\Om:=d_{\lam}$ and $k_\Om:=d_{\del^{-1}}$ that are induced by the hyperbolic and quasihyperbolic metrics $\lam\,ds, \del^{-1}ds$ on $\Om$.  These are geodesic distances: for any points $a,b$ in $\Om$, there are always an $h$-geodesic $[a,b]_h$ and a $k$-geodesic $[a,b]_k$ joining $a,b$ in $\Om$.  These geodesics need not be unique, but they enjoy the properties that
\[
  h(a,b)=\ell_h([a,b]_h) \quad\text{and}\quad k(a,b)=\ell_k([a,b]_k)\,.
\]
Here we are writing $\ell_h$ and $\ell_k$ in lieu of $\ell_\lam$ and $\ell_{\del^{-1}}$.

We adopt the following \textbf{\emph{conventions}} throughout the remainder of this paper: the term \emph{geodesic} shall mean either a hyperbolic geodesic or a quasihyperbolic geodesic (we will explicitly say which when it is necessary to do so) and we will say that $[a,b]_g$ (for $g=h$ or $g=k$) is a geodesic to mean that it is a hyperbolic geodesic when $g=h$ and a quasihyperbolic geodesic when $g=k$.

We adopt similar conventions when discussing quasi-geodesics and rough-quasi-geodesics; see \rf{s:P&Gs} for their definitions.

\subsubsection{The Punctured Disk $\D_*$} 	\label{ss:punxd D} %
Here we catalog some striking differences between the metric spaces $(\D_*,h_*)$ and $(\D_*,k_*)$, where $\D_*:=\D\sm\{0\}$ is the punctured unit disk and $h_*$ and $k_*$ denote, respectively, hyperbolic and quasihyperbolic distance in $\D_*$.\footnote{Caution! Elsewhere in this article $k_*$ is usually quasihyperbolic distance in $\mfC_*=\mfC\sm\{0\}$.}   Similar facts hold for $\mfC\sm\cmfD$, as well as plenty of other hyperbolic domains.

Recall that the hyperbolic metric $\lam_*ds$ on $\D_*$ is given by $\ds\lam_*(z)|dz|=\bigl(|z|\bigl|\log|z|\bigr|\bigr)^{-1}|dz|$.  From this it is straightforward to check that when $a,b\in\D_*$ and $\Arg(a/b)=0$,
\[
  h_*(a,b)=\bigl|\log\frac{\log(1/|a|)}{\log(1/|b|)}\bigr|
    \quad\text{whereas}\quad
  k_*(a,b)=\bigl|\log\frac{|a|}{|b|}\bigr| \;\text{(provided $|a|\vee|b|\le1/2$)}\,.
\]

\begin{xx}  \label{X:D*} %
For the punctured unit disk $\D_*$, the following hold.
\begin{itemize}
  \item[(a)]  Given \emph{any} sequences $(h_n)_1^\infty$ and $(k_n)_1^\infty$ of positive numbers with $1\ge h_n\to0$ and $2\le k_n\to\infty$, there are sequences $(a_n)_1^\infty, (b_n)_1^\infty$ in $\mfD_*$ with $h_*(a_n,b_n)=h_n$ and $k_*(a_n,b_n)=k_n$.
  \item[(b)]  The identity map $(\D_*,k_*)\xra{\id}(\D_*,h_*)$ is not quasisymmetric.
  \item[(c)]  Hyperbolic rough-geodesics in $\D_*$ may \emph{not} satisfy the ABC property.\footnote{See \rf{s:ABC}, and contrast (c) with \rfs{R:ABC}.}  
  \item[(d)]  Hyperbolic rough-geodesics in $\D_*$ may \emph{not} be quasihyperbolic rough-quasi-geodesics.
\end{itemize}
\end{xx} %
\begin{proof}%
For (a), simply take $a_n:=b_n e^{-k_n}, b_n:=\exp(-k_n/[e^{h_n}-1])$ and compute.  To establish (b), consider $\veps\in(0,1/2)$ and $a:=\veps, b:=1-\veps$.  It is easy to see that
\begin{gather*}
  k_*(a,1/2)=\log\frac1{2\veps}=k_*(b,1/2)
  \intertext{and we claim that}
  \lim_{\veps\to0^+}\frac{h_*(a,1/2)}{h_*(b,1/2)}=0\,.
  \intertext{To confirm this claim, notice that for $1/2\le|z|\le b$, $1/|z|\le\log2$, so}
  \lam_*(z)\ge\frac1{\log2}\,\frac1{|z|} \quad\text{whence}\quad h_*(b,1/2)\ge \frac1{\log2}\,k_*(b,1/2)
  \intertext{and therefore}
  \frac{h_*(a,1/2)}{h_*(b,1/2)} \le \frac1{\log2}\,\frac{h_*(a,1/2)}{k_*(a,1/2)} \longrightarrow 0 \;\text{as $\veps\to0^+$}\,.
\end{gather*}

For item (c), we demonstrate that for each $\mu>0$ and $\veps>0$ (here $\mu$ is large while $\veps$ is small:-) there exists a hyperbolic $\veps$-rough geodesic $\gam$ in $\D_*$ such that $\gam$ fails to have the $(\mu,\mu)$-ABC property.  To this end, let $(h_n)_1^\infty$ and $(k_n)_1^\infty$ be sequences of positive numbers with $1\ge h_n\to0$ and $2\le k_n\to\infty$, and put $a_n:=b_n e^{-k_n}, b_n:=\exp(-k_n/[e^{h_n}-1])$ so that $h_*(a_n,b_n)=h_n$ and $k_*(a_n,b_n)=k_n$.

Let $\mu>0$ and $\veps>0$ be given.  Choose $N\in\mfN$ so that for all $n\ge N$, $k_n>2\mu$ and $h_n<\veps/2$.  Fix any $n>N$.  Let $[0,2h_n]\xra{\gam}\D_*$ be the hyperbolic arclength parametrization of the path $[b_n,a_n]*[a_n,b_n]$ (from $b_n$ to $a_n$ back to $b_n$).  Evidently, for all $s,t\in[0,2h_n]$,
\[
  |s-t|-\veps\le2h_n-\veps\le0\le h_*(\gam(s),\gam(t))\le 2h_*(a_n,b_n) = 2h_n \le \veps\,,
\]
so $\gam$ is a hyperbolic $\veps$-rough geodesic.

Clearly, with $d_n:=b_ne^{-\mu}$ we have $A:=\A(0;d_n,\mu)\in\mcA_{\D_*}$.  Also, the subpath $\alf:=[d_n,a_n]\star[a_n,d_n]$ of $\gam$ that goes from $d_n$ to $a_n$ back to $d_n$ has its endpoints on $\mfS^1(A)$.  However, as $a_n\not\in A$, $|\alf|\not\subset A$.  Thus $\gam$ fails to have the $(\mu,\mu)$-ABC property.

For item (d), we demonstrate that for each $C>0$ and $\veps>0$ there exists a hyperbolic $\veps$-rough geodesic $\gam$ in $\D_*$ such that for all $K\ge1$,  $\gam$ fails to be a quasihyperbolic $(K,C)$-rough-chordarc path.  Indeed, let $C>0$ and $\veps>0$ be given and let $\gam$ be the hyperbolic arclength parametrization of the path $[b_n,a_n]*[a_n,b_n]$ as described above.  Assume $n$ is large enough so that $h_n<\veps/2$ and $k_n>C/2$.  Then $\gam$ is a hyperbolic $\veps$-rough geodesic.  Also, for any $K\ge1$,
\[
  \ell_{k_*}(\gam)=2k_*(a_n,b_n)=2k_n>C=K k_*(\gam(0),\gam(2h_n))+C
\]
and so $\gam$ is not a quasihyperbolic $(K,C)$-rough-chordarc path.

A minor modification of the above shows that for each $K\ge1$, $C>0$, and $\veps>0$ there exists a hyperbolic $\veps$-rough geodesic $\gam$ in $\D_*$ such that the quasihyperbolic arclength parametrization of $\gam$ fails to be a quasihyperbolic $(K,C)$-rough-quasi-geodesic.
\end{proof}%

\subsubsection{Hyperbolic Estimates} 	\label{ss:h ests} %
The standard technique for estimating the hyperbolic metric and hyperbolic distance is via domain monotonicity, a consequence of Schwarz's Lemma.  That is, if $\Om_{\rm in}\subset\Om\subset\Om_{\rm out}$, then in $\Om_{\rm in}$, $\lam_{\rm in}ds \ge\lam\,ds\ge\lam_{\rm out}ds$ and $h_{\rm in}\ge h \ge h_{\rm out}$.

Notice that the largest hyperbolic plane regions are twice punctured planes.  We write $\lam_{ab}ds$ for the hyperbolic metric in the twice punctured plane $\Cab$.  The `standard' twice punctured plane is $\Coo$ and its hyperbolic metric has been extensively studied by numerous researchers including \cite{Hempel-twice}, \cite{Minda-LNM}, \cite{SV-estimates}, \cite{SV-ineqs}.  We mention only the following information (which we require in the sequel).

\begin{fact}  \label{F:lam_01} %
For all $z\in\Coo$, $\lam_{01}(z)\ge\lam_{01}(-|z|)\ge\bigl( |z|[\kk+\bigl|\log|z|\bigr|] \bigr)^{-1}$, with equality at $z$ \ifff $z=-1$.
Here $\kk:= \bigl(\lam_{01}(-1)\bigr)^{-1} = \Gam^4(1/4)/\bigl(4\pi^2\bigr)=4.3768796\dots$.
\end{fact}%
\noindent
\rf{F:lam_01} was first proved by Lehto, Virtanen and \Va\ (see \cite{LVV-distortion}); later proofs were given by Agard \cite{Agard-distortion}, Jenkins \cite{Jenkins-LandauII}, and Minda \cite{Minda-LNM}. 

Here is a useful consequence of the above.  (The alert reader recognizes that in the setting below we actually have $\bp(z)=\log(R/|z|)$ for all $0<|z|<R/2$.)

\begin{lma} \label{L:lam-del ests in D*} %
Suppose that for some $R>0$ the punctured disk $\D_*(0;R)\in\mcA^2_\Om$.  Then for all $z\in\Om$ with $|z|<R/2$, $\del(z)=|z|$ and
\[
  \Bigl(|z|\bigl(\kk+\log(R/|z|)\bigr)\Bigr)^{-1} \le \lam(z) \le \Bigl(|z|\log(R/|z|)\Bigr)^{-1}\,.
\]
\end{lma} 
\begin{proof}%
Recall that $D_*:=\D_*(0;R)\in\mcA^2_\Om$ means $D_*\subset\Om$, $0\in\bOm$, and there exists a point $\xi\in\bOm$ with $|\xi|=R$.  Thus $D_*\subset\Om\subset\mfC_{0\xi}$ and so the above estimates on $\lam(z)$ follow at once from $\lam_*\ge\lam\ge\lam_{0\xi}$ in conjunction with \rf{F:lam_01}, where $\lam_*:=\lam_{D_*}$.

Indeed, $\lam_*(z)|dz|=|dz|/\bigl(|z|\log(R/|z|)\bigr)$, and as $\Om\supset\mfC_{0\xi}$, the change of variables $w=z/\xi$ together with \rf{F:lam_01} yields
\begin{align*}
  \lam(z) &\ge \lam_{0\xi}(z) = \frac1{|\xi|}\, \lam_{01}(w) \ge \frac1{|\xi|}\, \frac1{|w| \bigl( \kk + \bigl|\log|w|\bigr| \bigr)} \\
          &= \frac1{|z| \bigl( \kk + \bigl|\log|z/\xi|\bigr| \bigr)} = \frac1{|z| \bigl( \kk + \log(R/|z|) \bigr)}\,.  \qedhere
\end{align*}
\end{proof}%

We can also use \rf{F:lam_01} to obtain a lower bound for certain hyperbolic distances.

\begin{lma} \label{L:h est in C01} %
Let $a,b\in\Coo$.  Assume $|a|\le|b|\le1$.  Then $\ds h_{01}(a,b)\ge \log\frac{\kk+\log(1/|a|)}{\kk+\log(1/|b|)}$.
\end{lma} 
\begin{proof}%
Writing $[a,b]_{01}$ for a fixed hyperbolic geodesic joining $a,b$ in $\Coo$ we have
\begin{align*}
  h_{01}(a,b) &= \int_{[a,b]_{01}}\lam_{01}ds \ge \int_{[a,b]_{01}}\frac{|dz|}{|z|\bigl(\kk+\bigl|\log|z|\bigr|\bigr)} \ge \\ &\ge \int_{|a|}^{|b|}\frac{|dz|}{|z|\bigl(\kk+\bigl|\log|z|\bigr|\bigr)} = \log\frac{\kk+\log(1/|a|)}{\kk+\log(1/|b|)}\,.  \qedhere
\end{align*}
\end{proof}%
\begin{cor} \label{C:h ge ln3/2} %
Suppose $q\ge\kk$ and $A\in\mcA^2_\Om(4q)$.  Let $C$ be one of the collars of $\core_{2q}(A)$ relative to $\core_q(A)$.  For any path $\gam$ in $\Om$ that crosses $C$ we have $\ell_h(\gam)\ge\log(3/2)$.
\end{cor} 
\begin{proof}%
First, note that $q/(\kk+q)\ge1/2$, so $(\kk+2q)/(\kk+q)\ge3/2$.  Thus if $|a|=e^{-2q}$ and $|b|=e^{-q}$, then by \rf{L:h est in C01}, $h_{01}(a,b)\ge\log(3/2)$.

Now suppose $A$ is a degenerate annulus.  We may assume $A:=\D_*(0;R)$ for some $R>0$.  Then $A\subset\Om$, $0\in\bOm$, and there exists a point $\xi\in\bOm$ with $|\xi|=R$.  Here $\core_q(A)=\D_*(0;e^{-q}R)$ and $C=\D(0;e^{-q}R)\sm\D[0;e^{-2q}R]$.  As $\gam$ crosses $C$, we can pick $a,b\in|\gam|$ with $|a|=e^{-2q}R$ and $|b|=e^{-q}R$.  Since $0,\xi\in\bOm$, $\Om\subset\mfC_{0\xi}$.  Therefore,
\[
  \ell_h(\gam) \ge h(a,b) \ge h_{0\xi}(a,b) = h_{01}(a/\xi,b/\xi) \ge \log(3/2)\,.
\]

Next, suppose $A$ is a non-degenerate annulus.  We may assume $A:=\A(0;1,r)$ for some $r>2q$.  Then $A\subset\Om$, $0\in\bOm$, and there exist points $\xi,\eta\in\bOm$ with $|\xi|=de^r$ and $|\eta|=de^{-r}$.  Here $C$  is one of the two components of $\core_q(A)\sm\overline\core_{2q}(A)$.  When $C$ is the outer collar, i.e., $C=\{e^{r-2q}<|z|<e^{r-q}\}$, we proceed as above.

Suppose $C$ is the inner collar, i.e., $C=\{e^{q-r}<|z|<e^{2q-r}\}$.  As $\gam$ crosses $C$, we can pick $a,b\in|\gam|$ with $|a|=e^{2q-r}$ and $|b|=e^{q-r}$.  Since $0,\eta\in\bOm$, $\Om\subset\mfC_{0\eta}$.  Therefore,
\[
  \ell_h(\gam) \ge h(a,b) \ge h_{0\eta}(a,b) = h_{01}(a/\eta,b/\eta) \ge \log(3/2)\,.  \qedhere
\]
\end{proof}%

\begin{rmks}  \label{R:h ge ln3/2} %
\noindent(a)  In \rf{C:h ge ln3/2} it is imperative  that $A\in\mcA^2_\Om$.

\smallskip\noindent(b)  It is illustrative to see how the various propositions in \rf{S:lengths} give $\ell_h(\gam)\ge\veps_o$ for some absolute constant $\veps_o>0$, when $\gam$ is as in \rf{C:h ge ln3/2}.  For example, suppose $A$ is a non-degenerate annulus and $\gam$ crosses the inner collar of $\core_{2q}(A)$ relative to $\core_q(A)$.  Pick $a,b\in|\gam|$ with $|a|=e^{q-r}$ and $|b|=e^{2q-r}$ and apply the proof of \rf{P:A:ell-bad-arcs} to a geodesic $[a,b]_h$.  Noting that $|b|=e^q|a|$ we have $m<q\le m+1$, so for all $1\le j\le m$, $q+j\le q+m \le 2m+1$.  As $\ups_j=r+\log|a|-j=q+j$,
\begin{gather*}
  \sum_{j=1}^m\frac1{\ups_j}=\sum_{j=1}^m\frac1{q+j}\ge\frac{m}{2m+1}\ge\frac13
  \intertext{and thus by \eqref{E:A:lhah}}
  h(a,b)=\sum_{j=1}^m\ell_h(\alf_j)\ge(16e^{2\mu_o})^{-1}\sum_{j=1}^m\frac1{\ups_j}\ge(48e^{2\mu_o})^{-1}=:\veps_o\,.
\end{gather*}
\end{rmks} %

\subsubsection{QuasiHyperbolic Estimates} 	\label{ss:k ests}%
We remind the reader of the following basic estimates for quasihyperbolic distance, first established by Gehring and Palka \cite[2.1]{GP-qch}: for all $a,b\in\Om$,
\begin{equation}\label{E:k ge j}
  k(a,b) \ge \log\left(1+\frac{l(a,b)}{\del(a)\wedge \del(b)}\right)  \ge j(a,b)
    :=\log\left(1+\frac{|a-b|}{\del(a)\wedge \del(b)}\right) \ge\left|\log\frac{\del(a)}{\del(b)}\right|\,;
\end{equation}
here $l(a,b)$ is the (intrinsic) length distance between $a$ and $b$.
See also \cite[(2.3),(2.4)]{BHK-unif}.  The first inequality above is a special case of the more general (and easily proven) inequality
\[
\ell_k (\gam) \ge \log\bigl(1+\ell(\gam)/\min_{z\in|\gam|}\del(z) \bigr)
\]
which holds for any rectifiable path $\gam$ in $\Om$.  As a special case of \eqref{E:k ge j}: if $o\in\Om^c$ and $a,b\in\Om$, then
\begin{equation}\label{E:k ge log}
  k(a,b) \ge \bigr|\log\frac{|a-o|}{|b-o|}\bigr|\,.
\end{equation}
Indeed, assuming $o=0$, writing $k_*:=k_{\Co}$ and $\del_*=\del_{\Co}$, and noting that $\Om\subset\Co$, we deduce that $k(a,b)\ge k_*(a,b)\ge|\log\bigl(\del_*(a)/\del_*(b)\bigr)|=\bigl|\log\bigl(|a|/|b|\bigr)\bigr|$.

There is an alternative way to deduce the lower bound in \eqref{E:k ge log}.  It is well known that $(\Co,k_*)$ is (isometric to) the Euclidean cylinder $\mfS^1\times\mfR^1$ (with its Euclidean length distance inherited from the standard embedding into $\mfR^3$).  One way to realize this is via the \holo\ covering $\mfC\xra{\exp}\Co$ which pulls back the quasihyperbolic metric $\del_*^{-1}ds$ on $\Co$ to the Euclidean metric on $\mfC$, as explained in \cite{MO-qhyp}.  In particular, the quasihyperbolic geodesics in $\Co$ are logarithmic spirals.  Also, we find that for all $a,b\in\Co$,
\begin{gather}
  k_*(a,b)=\bigl|\Log(b/a)\bigr|=\bigl|\log|b/a|+i\Arg(b/a)\bigr| \notag
  \intertext{and thus}
  \bigl|\log\frac{|b|}{|a|}\bigr|\vee\bigl|\Arg\bigl(\frac{b}{a}\bigr)\bigr| \le  k_*(a,b) \le
  \bigl|\log\frac{|b|}{|a|}\bigr| + \bigl|\Arg\bigl(\frac{b}{a}\bigr)\bigr| \le
  \bigl|\log\frac{|b|}{|a|}\bigr| + \frac{\pi}{2}\,\frac{|a-b|}{|a|\wedge|b|}\,.  \label{E:k* ests}
\end{gather}
It is worth calling attention to the special cases of the above that arise when $|a|=|b|$ and when $\Arg(b/a)=0$.  To see the last inequality involving $\tha:=\bigl|\Arg(b/a)\bigr|$, we assume that $|a|\le|b|$.  From the proof of \rf{L:ell chi} we know that $|a-b|\ge|a-c|$ where $c:=\bigl(|a|/|b|\bigr)b$.  Then since $|a|=|c|$ and $\tha=|\Arg(c/a)|$,
\[
  |a-c|^2=4|a|^2\sin^2\frac{\tha}2\,, \quad\text{and so}\quad \frac{|a-b|}{|a|}\ge\frac{|a-c|}{|a|}=2\sin\frac{\tha}2\ge\frac2{\pi}\,\tha\,.
\]

\begin{rmk}  \label{R:k* ests} %
From \rf{L:del deep in A} we have $\half\del_*\le\del\le\del_*$ in $\core_{\log2}(A)$ whenever $A\in\mcA_\Om(\log4)$ with $c(A)=0$, so $k_*\le k\le 2k_*$ and thus \eqref{E:k* ests} provides good estimates (both upper and lower bounds) for quasihyperbolic distances in $\core_{\log2}(A)$.  See, e.g., the proof of \rf{L:k ests in A}.
\end{rmk} %

We make use of the following.

\begin{lma} \label{L:k-geos in pxd plane} %
Let $a,b\in\Co$ and let $[a,b]_*$ be a quasihyperbolic geodesic in $\Co$.  Then
\[
  |a-b| \le \ell([a,b]_*) \le 5|a-b|\,.
\]
\end{lma} 
\begin{proof}%
There is no harm in assuming that $a=1<|b|$.  Let $\beta:=\Arg(b)$, so $\Log(b)=\log|b|+i\beta$ (and $k(1,b)=|\Log b|$).  A \param\ for $\gam:=[a,b]_*$ is given by
\begin{gather*}
  \gam(t):=\exp\bigl(t\Log b\bigr) \quad\text{for $0\le t\le1$}
  \intertext{from which we obtain}
  \dot\gam(t)=\Log(b)\,\gam(t) \quad\text{and}\quad  |\dot\gam(t)|=|\Log b||\gam(t)|=|\Log b|\exp\bigl(t\log|b|\bigr)
  \intertext{and thus}
  \ell(\gam)=\abs{\Log b}\, \int_0^1 e^{t\log|b|}\, dt = \abs{\Log b} \, \frac{|b|-1}{\log|b|}\,.
\end{gather*}

When $|b|\ge e$, $\abs{\Log b}/\log|b|\le 1+\pi\le5$, and so $\ell(\gam)\le5|a-b|$.  Assume $1<|b|\le e$.  Then $\log|b|\le|b|-1$.  Since $|e^{i\beta}-1|^2=4\sin^2(\beta/2)\ge4(\beta/\pi)^2$,
\begin{gather*}
  \abs{\Log b}\le\log|b|+|\beta|\le\bigl(1+\frac{\pi}2\bigr)|b-1|\,.
  \intertext{By Calculus, $x-1\le(e-1)\log x$ for $1\le x\le e$, whence}
  \ell(\gam) = \abs{\Log b} \, \frac{|b|-1}{\log|b|} \le (e-1)\bigl(1+\frac{\pi}2\bigr)|b-1| \le 5|b-1|\,.  \qedhere
\end{gather*}
\end{proof}%

Note that on $[a,b]_*$ we have $|a|\wedge|b|\le|z|\le|a|\vee|b|$ and therefore by \rf{L:k-geos in pxd plane}
\[
  \frac{|a-b|}{|a|\vee|b|} \le k_*(a,b) \le 5\,\frac{|a-b|}{|a|\wedge|b|}\,.
\]
Of course these estimates are most useful when $|a|\eqx|b|$.
We can improve the upper estimate (reducing the $5$ to a $3$) by using \rf{L:ell chi}.  In fact, since $\abs{\log(t)}\le t-1$ for $t\ge1$, $\bigl|\log(|b|/|a|)\bigr|\le|a-b|/\bigl(|a|\wedge|b|\bigr)$ and thus by \eqref{E:k* ests} we have
\begin{equation}\label{E:k* le}
  \frac{|a-b|}{|a|\vee|b|} \le k_*(a,b) \le \bigl(1+\frac{\pi}{2}\bigr)\frac{|a-b|}{|a|\wedge|b|} < 3\,\frac{|a-b|}{|a|\wedge|b|}\,.
\end{equation}

We require the following elementary estimates for certain quasihyperbolic lengths.

\begin{lma} \label{L:k ests in A} 
\addtocounter{equation}{-1}  
Suppose $A\in\mcA_\Om(\log4)$ and $a,b\in\core_{\log2}(A)$.  Let $\chi:=\chi_{ab}=\kap_{ac}\star[c,b]=\kap\star[c,b]$ be the circular-arc-segment path from $a$ through $c$ to $b$ as described in \eqref{E:chi}.  Then
\begin{subequations}\label{E:k ests in A}
  \begin{align}
    \ell_k(\kap) &\le 2 k(a,c)\,,                      \label{E:kap QG}  \\
    \ell_k([b,c])&\le 2 k(b,c)\,,                      \label{E:[b,c] QG}  \\
    \ell_k(\chi) &\le 4 k(a,b)\,;                      \label{E:chi QG}
  \end{align}
\end{subequations}
Also, for any $q\in[\log2,m)$, $\diam_k\core_q(A)\le2\pi+2\md\core_q(A)$. 
\end{lma} 
\begin{proof}%
From \rf{L:del deep in A}, for $z\in\core_{\log2}(A)$, $|z|/2\le\del(z)\le|z|$, so $1/\del_*\le1/\del\le2/\del_*$ and $k_*\le k\le 2 k_*$ in $\core_{\log2}(A)$ where $k_*$ is quasihyperbolic distance in $\Co$.  Put $\tha:=k_*(a,c)=\ell_{k_*}(\kap)$ and $L:=k_*(b,c)=\ell_{k_*}([b,c])$.  Then
\begin{gather*}
  \ell_k(\kap) \le 2\ell_{k_*}(\kap)=2k_*(a,c)\le2k(a,c)\,, \\
  \ell_k([b,c]) \le 2\ell_{k_*}([b,c])=2k_*(a,b)\le2k(a,b)\,. \\
  \intertext{Also, $k_*(a,b)=|\Log(b/a)|=|L+i\tha|\ge(L+\tha)/2$, so}
  \ell_k(\chi)\le2\ell_{k_*}(\chi)=2(\tha+L)\le4k_*(a,b)\le4k(a,b)\,.  \qedhere
\end{gather*}
\end{proof}%

\begin{rmk}  \label{R:k ests in A} %
The inequalities \eqref{E:k ests in A} reveal that $\kap, [b,c]$, and $\chi$ are all quasihyperbolic quasi-geodesics (although one needs to be careful here:-).  It then follows from \rf{TT:geos} that $\kap, [b,c]$, and $\chi$ are also hyperbolic quasi-geodesics.
\end{rmk} %

An important property of hyperbolic distance is its conformal invariance.  While this does not hold for quasihyperbolic distance, it is \mob\ quasi-invariant in the following sense; see \cite[Lemma~2.4, Corollary~2.5]{GP-qch}.

\flag{surely there is a similar fact describing how $\bp$ changes.  note that annuli get inverted to mobius annuli....  such a result might allow us to get rid of some tedious arguments dealing with cases where $\A(z)$ is complement of disk.}

\begin{fact}  \label{F:kQI} %
Let $\RS\xra{T}\RS$ be a \MT.  Let $\Om\subsetneq\mfC$ and suppose $\Om':=T(\Om)\subset\mfC$.  Then
\begin{gather*}
  \forall\; z\in\Om\,, \quad  \frac1{\del(z)} \le2\,\frac{|T'(z)|}{\del'(T(z))}\,.
  \intertext{Consequently, for all rectifiable paths $\gam$ in $\Om$,}
  \half \ell_k(\gam) \le \ell_{k'}(T\comp\gam) \le 2\ell_k(\gam)\,;
  \intertext{in particular,}
  \forall\; a,b\in\Om\,, \quad \half k(a,b) \le k'(a',b') \le 2k(a,b)\,.
\end{gather*}
In the above, $a':=T(a), b':=T(b), \del':=\del_{\Om'}, k':=k_{\Om'}$.
\end{fact} %

\subsubsection{The Beardon-Pommerenke Result} 	\label{ss:BP} %
One desires both upper and lower estimates for the hyperbolic metric.  In general, finding lower estimates 
seems to be the more difficult endeavor.  It is well-known that the hyperbolic and quasihyperbolic metrics are 2-bi-Lipschitz equivalent for simply connected hyperbolic plane regions; this fact is not true, e.g., for any domain with an isolated boundary point (such as the punctured unit disk).  In general, the hyperbolic and quasihyperbolic metrics are bi-Lipschitz equivalent precisely when $\Om^c$ is uniformly perfect (cf.\ \cite{BP-beta}, \cite{Pomm-unifly-perfect1}, \cite{Pomm-unifly-perfect2}).  Beardon and Pommerenke corroborated this latter assertion as an application of their elegant result \cite[Theorem~1]{BP-beta} which says:
\begin{noname}
For any hyperbolic region $\Om$ in $\mfC$ and for all $z\in\Om$,
\[
  \frac{1}{\del(z) [\kk+\bp(z)]} \le \lam(z) \le \frac{\pi}{2}\;\frac{1}{\del(z) \, \bp(z)} \,.  \tag{\BPt}
\]
\end{noname}
\noindent
Beardon and Pommerenke introduced the domain function $\Om\xra{\bp}\mfR$ which is defined via
\[
  \bp(z)=\bp_{\Om}(z):=
  \inf_{\substack{
    {\zeta\in\B(z)}\\
     \xi\in\Om^c\sm\{\zeta\}}}
    \biggl|\log\Bigl|\frac{\zeta-z}{\zeta-\xi}\Bigr|\biggr|\,;
\]
note that the infimum is restricted to nearest boundary points $\zeta\in\B(z)=\bOm\cap\bD(z)$ for $z$ (that is, $\zeta\in\bOm$ with $\del(z)=|z-\zeta|$).

The definition of $\bp$ can be motivated by examining the standard lower bound for the hyperbolic metric on a twice punctured plane; see Fact~\ref{F:lam_01}. The (\BPt) inequalities above (established by Beardon and Pommerenke) follow by using domain monotonicity: the upper bound for $\lam(z)$ holds because $z$ lies on the conformal center of a certain annulus in $\Om$, and the lower bound holds because $\Om$ lies in a certain twice punctured plane.

To obtain a geometric interpretation for $\bp(z)$, we define $\bp(z,\zeta)$---for points $z\in\Om$ and $\zeta\in\B(z)$---as follows:
\[
  \bp(z,\zeta):=\inf_{\xi\in\Om^c\sm\{\zeta\}} \biggl|\log\Bigl|\frac{\zeta-z}{\zeta-\xi}\Bigr|\biggr|\,.
  \qquad\text{(Thus, $\ds\bp(z)=\inf_{\zeta\in\B(z)}\bp(z,\zeta)$)}\,.
\]
We note that given $z\in\Om$ and $\zeta\in\B(z)$, there always exists a point $\xi\in\bOm\sm\{\zeta\}$ with $\bp(z,\zeta)=\bigl|\log|\zeta-z|/|\zeta-\xi|\bigr|$, and there is always a $\zeta\in\B(z)$ with $\bp(z)=\bp(z,\zeta)$.

Now fix points $z\in\Om$ and $\zeta\in\B(z)$.  Let $C=\mfS^1(\zeta;\del(z))$ be the circle through $\zeta$ with radius $\del(z)=|z-\zeta|$.  We have $\bp(z,\zeta)=0$ \ifff $C\cap\bOm\ne\emptyset$.  Suppose $C\subset\Om$, so $\bp(z,\zeta)>0$.
Choose a point $\xi\in\bOm\sm\{\zeta\}$ with $\upsilon:=\bp(z,\zeta)=\bigl|\log|\zeta-z|/|\zeta-\xi|\bigr|$.  Let
\[
  A=\mfA(\zeta;\del(z),\upsilon) = \{w\in\mfC  :  \del(z)e^\upsilon <|w-\zeta|< \del(z)e^{\upsilon} \}\,.
\]
Note that $\md(A)=2\upsilon$.  We claim that $A$ is the maximal  annulus that is contained in $\Om$ and is symmetric \wrt $C$.  This is because on the one hand, the point $\xi$ lies on $\bd A$, so $\bd A\cap\bOm\ne\emptyset$.  On the other hand, from the definition of $\bp(z,\zeta)$ we must have $A\subset\Om$.

Thus, $2\bp(z,\zeta)$ is the conformal modulus of the maximal Euclidean annulus that is contained in $\Om$ and symmetric \wrt the circle $\mfS^1(\zeta;\del(z))$.  It follows that $2\bp(z)$ is the minimum of these numbers; so, $2\bp(z)$ is the smallest of these maximal moduli.

From the discussion above we see that whenever $\bp(z)>0$, there is an annulus
\begin{gather*}
  \BP(z):=\A(\zeta;\del(z),\bp(z)) \in\mcA^1_\Om
  \intertext{associated with $z$; here $\zeta\in\B(z)$ is any nearest boundary point for $z$ that realizes $\bp(z)$, and}
  \forall\;\xi\in\bd\BP(z)\cap\bOm\,,\quad \bp(z)=\bp(z,\zeta)=\biggl|\log\Bigl|\frac{\zeta-z}{\zeta-\xi}\Bigr|\biggr|\,.
\end{gather*}
We call $\BP(z)$ a \emph{\BPt\ annulus} (or a \emph{Beardon-Pommerenke annulus}) associated with the point $z$; it needn't be unique.

\smallskip

Note that when $\bp(z,\zeta)>0$, any extremal point $\xi\in\Om^c$ that realizes $\bp(z,\zeta)$ (so, $\bp(z,\zeta)=\bigl|\log|\zeta-z|/|\zeta-\xi|\bigr|$) must lie on $\bOm$ (because then $\xi\in\bd \mfA(\zeta;\del(z),\bp(z,\zeta)$).  This is not necessarily true when $\bp(z,\zeta)=0$; nonetheless, even in this case we can still locate a point $\eta\in\bOm$ with $\bp(z,\zeta)=\bigl|\log|\zeta-z|/|\zeta-\eta|\bigr|$ (just take any $\eta\in\bOm\cap\mfS^1(\zeta;\del(z))$).

We mention one last issue:  while $\bp(z,\zeta)$ is continuous as a function of $(z,\zeta)\in\Om\times\Om^c$, $z\mapsto\bp(z)$ is not continuous!  Indeed, for the punctured unit disk $\Do$ we have $\bp(z)=\bigl|\log|z|\bigr|$ for $0<|z|<1/2$ whereas $\bp(z)=0$ for $1/2\le|z|<1$.  However, $z\mapsto\bp(z)$ is continuous on $\bp>\log2$ \flag{proofs?}.

\smallskip

In fact, in this paper \emph{we only care about large values} of $\bp$.  In particular, when $\bp\ge\kk$ the (\BPt) inequalities imply that
\begin{equation}\label{E:BP est}
  \frac1{2\,\del\,\bp} \le \lam \le \frac2{\del\,\bp}\,.
\end{equation}
In \rf{s:BPEP} we establish useful estimates for $\bp$; see especially \rf{L:bp ge} and \rf{P:BPEP}. 
%
%
\section{Technical Tools for Proofs}  \label{S:Tools} 
In this section we discuss various estimates for the Beardon-Pommerenke $\bp$ function, construct the all important annulus $\A(z)$, introduce the \emph{annulus bounce cross property}, and establish various properties of certain cores of the $\A(z)$ annuli.

\subsection{Estimating $\bp$}  \label{s:BPEP} %
Here we provide several means for estimating the Beardon-Pommerenke domain function $\bp$.  In particular, we prove \rf{P:BPEP} (the \BPEP\footnote{Beardon-Pommerenke $\bp$ Estimates Proposition}) which gives especially useful estimates for $\bp$.  As a warm up, the reader can verify the following geometric estimates for $\bp$.

\begin{lma} \label{L:bp ests} %
For $\ups>0$ and any point $z$ in a hyperbolic plane domain $\Om$,
\begin{align*}
  &\bp(z)\ge\ups \iff \forall\;\zeta\in\B(z)\,, \quad \A(\zeta;\del(z),\ups)\subset\Om
  \intertext{and}
  &\bp(z)\le\ups \iff \exists\;\zeta\in\B(z)\;\text{\rm so that} \quad \A[\zeta;\del(z),\ups]\cap\Om^c\neq\emptyset\,.
\end{align*}
\end{lma} 
\flag{and, $\bp(z)>\ups \iff \forall\;\zeta\in\B(z)\,, \quad \A[\zeta;\del(z),\ups]\subset\Om$}

Here is a general lower bound for $\bp$.

\begin{lma} \label{L:bp ge} %
Fix $r>q\ge \log6$.  Suppose $A:=\A(o;d,r)\in\mcA_\Om$.  Then for all $z\in\core_q(A)$,
\[
  \bp(z)> \frac12 \,q\,.
\]
\end{lma} 
\begin{proof}%
We may assume that $o=0$ and $d=1$.  Let $z\in\core_q(A)=\A(o;d,r-q)$; so,
\[
  e^{q-r}<|z|<e^{r-q}\,.
\]
Let $\zeta\in\B(z)$.  Then by \rf{L:del deep in A}, $|\zeta|\le e^{-r}$ and $|z|-e^{-r} \le \del(z)=|z-\zeta| \le |z|$.

To verify the lower bound on $\bp(z)$, it suffices to demonstrate that
\[
  B:=\A[\zeta;\del(z),q/2]\subset A:=\A(o;d,r)=\A(0;1,r)\,.
\]
So, let $w\in B$.  We show that $e^{-r}<|w|<e^r$.  Now
\begin{gather*}
  |w|\le|w-\zeta|+|\zeta| \le \del(z)e^{q/2} + e^{-r} < e^{r-q/2} + e^{-r} \le e^r,
  \intertext{because $e^{2r}\ge2\ge1/(1-e^{-q/2})$; also}
  |w|\ge|w-\zeta|-|\zeta| \ge \del(z)e^{-q/2} - e^{-r} > \lp e^{q-r}-e^{-r} \rp e^{-q/2} - e^{-r} \ge e^{-r}\,,
\end{gather*}
because $e^q-1\ge2e^{q/2}$ since $q\ge\log6$.
\end{proof}%

Next we establish the \BPEP.
This result is at the heart of our proof of \rf{TT:geos}.  It says that in a LFS annulus (in $\mcA^2_\Om$), the domain function $\bp$ decays `linearly' as we move away from the center circle.  The assumption that both boundary circles of the annulus meet $\bOm$ is crucial for obtaining the upper bounds; when only one boundary circle has a boundary point, $\bp$ can actually increase when we move away from the center circle towards the boundary circle that does not meet $\bOm$.

\begin{prop} \label{P:BPEP} %
Suppose $\A(o;d,r)\in\mcA^2_\Om(8\log2)$, so $r>\log16$. Then for all $|t|\lex r$ and all $z\in\mfS^1(o;de^t)$, $\bp(z)\eqx r-|t|$.  More precisely, for all $|t|\le r-\log16$ and for each $z\in\mfS^1(o;de^t)$,
\begin{gather*}
  \bp(z)\ge\frac12 \lp r-|t| \rp \quad\text{and}\quad
  \begin{cases}
       \bp(z)\le2 \lp r-|t| \rp &\text{for $t\ge0$}\, \\
       \bp(z)\le \lp r-|t| \rp + \log2 &\text{for $t\le0$}\,.
  \end{cases}
\end{gather*}
\end{prop} 
\begin{proof}%
We may assume that $o=0$ and $d=1$.  Pick $\eta_\pm\in\bOm$ with $|\eta_\pm|=e^{\pm r}$.
Let $z\in\A(o;d,r)=\A(0;1,r)$ with $|z|=e^t$ and $|t|\le r-\log16$.  Let $\zeta\in\B(z)$.  Then
\[
  |\zeta|\le e^{-r} \quad\text{and}\quad |z|-e^{-r} \le \del(z)=|z-\zeta| \le |z|\,.
\]

The lower bound on $\bp(z)$ follows at once from \rf{L:bp ge}.
To establish the upper bounds on $\bp(z)$, we simply use the fact that
\[
  \bp(z)\le\left| \log\frac{\del(z)}{|\zeta-\xi|} \right| \quad\text{for each $\xi\in\{0,\eta_-,\eta_+\}\sm\{\zeta\}$}\,.
\]
When $t\ge0$, we take $\xi=\eta_+$.  Now $e^t-e^{-r}\ge\half e^t$, so
\begin{gather*}
  1\le \frac{|\zeta-\eta_+|}{\del(z)} \le \frac{e^r+e^{-r}}{e^t-e^{-r}} \le 2e^{-t}(e^r+e^{-r})\le 1+2e^{r-t}\,;
  \intertext{since $\log(1+2x)\le2\log x$ for all $x\ge3$, we now see that}
  \bp(z)\le\log(1+2e^{r-t})\le2(r-t)\,.
\end{gather*}

Suppose $t\le0$, so $|z|=e^t=e^{-|t|}$.  Assume $|\zeta|\le\half e^{-r}$; we take $\xi=\eta_-$.  Then
\begin{gather*}
  1\le \frac{\del(z)}{|\zeta-\eta_-|} \le 2e^r\del(z) \le 2e^r|z| = 2e^{r-|t|}
  \intertext{so}
  \bp(z)\le\log\frac{\del(z)}{|\zeta-\eta_-|}\le r-|t|+\log2\,.
\end{gather*}
Next, assume $|\zeta|\ge\half e^{-r}$; we take $\xi=0$.  Then
\begin{gather*}
  1\le \frac{\del(z)}{|\zeta|} \le 2e^r\del(z) \le 2e^r|z| = 2e^{r-|t|}
  \intertext{so}
  \bp(z)\le\log\frac{\del(z)}{|\zeta|}\le r-|t|+\log2\,. \qedhere
\end{gather*}
\end{proof}%

\begin{rmks}  \label{R:bp ests} %
In the above, we are looking at points $z\in\overline\core_{\log16}(A)$ where $A:=\A(o;d,r)$, $r>\log16$, $A\in\mcA^2_\Om$, and $t:=\log\lp|z-o|/d\rp$.  In particular:

\smallskip\noindent(a)  In $\core_{\log16}(A)$, $\bp\le 2r=\md(A)$, and on $\mfS^1(A)=\mfS^1(o;d)$, $\bp\ge\half r$.

\smallskip\noindent(b)  On both boundary circles $\mfS^1(o;16^{\pm1}de^{\mp r})$, we have $2\log2 \le \bp \le 8\log2$.

\smallskip\noindent(c)  For $\log16<q<r$ we have the strict inequalities
\[
  \text{$\bp>\half\,q$ in $\core_q(A)$ \hspace{1em} and \hspace{1em} $\bp<2\,q$ in $\core_{\log16}(A)\sm\overline\core_q(A)$}\,.
\]

\smallskip\noindent(d)  We also have estimates for $\bp$ in annuli $A=\A(o;d,r)\in\mcA^1_\Om$, as long as we move towards the boundary circle that has a boundary point.  To be precise, suppose $|z-o|=de^t$ with $|t|\le r-\log16$.  Then
\begin{align*}
   \bp(z) &\le2 \lp r-|t| \rp        &\text{when $t\ge0$ and $\bdo A\cap\bOm\ne\emptyset$}\,, \\
   \bp(z)& \le \lp r-|t| \rp + \log2 &\text{when $t\le0$ and $\bdi A\cap\bOm\ne\emptyset$}\,.
\end{align*}
\end{rmks} %
\flag{I think we only need $-r+\log16\le t \le r-\log3$ :-)}


\subsection{The Annulus $\A(z)$}  \label{s:BP&A} %

Recall from \rf{ss:BP} that for each $z\in\Om$ with $\bp(z)>0$ there is an associated \emph{Beardon-Pommerenke annulus}
\[
  \BP(z):=\A(\zeta;\del(z),\bp(z))\in\mcA^1_\Om \quad\text{(so, $\BP(z)\subset\Om$ and $\bd\BP(z)\cap\bOm\ne\emptyset$)}.
\]
For our purposes, there is a second more important annulus $\A(z)$ which has the crucial property that \emph{both} of its boundary circles meet $\bOm$; this then permits the use of \rf{P:BPEP} to estimate $\bp$ in $\A(z)$.  If $\BP(z)\in\mcA^2_\Om$ (i.e., if $\BP(z)$ already has the property that both of its boundary circles meet $\bOm$), then we set $\A(z):=\BP(z)$.  Otherwise, we obtain $\A(z)$ by `enlarging' $\BP(z)$ until we hit $\bOm$; that is, we start with the boundary circle of $\BP(z)$ that does not meet $\bOm$ and then grow or shrink this circle (keeping its center fixed) until we hit $\bOm$.

To be precise \flag{pictures here!?}, suppose
\begin{gather*}
  m:=\bp(a)=\biggl|\log\Bigl|\frac{\zeta-a}{\zeta-\xi}\Bigr|\biggr|
  \intertext{with $\zeta\in\B(a)$ and $\xi\in\bOm$, so}
  \BP(a)=\A(\zeta;\del(a),m)=\{z\in\mfC  :  \del(a)e^{-m}<|z-\zeta|<\del(a)e^m\}\,.
\end{gather*}
We assume one of the boundary circles of $\BP(a)$ does \emph{not} meet $\bOm$; so, $\BP(a)\in\mcA^1_\Om\sm\mcA^2_\Om$.  There are two cases depending on whether $|\zeta-\xi|=\del(a)e^m$ or $|\zeta-\xi|=\del(a)e^{-m}$.

Assume $\xi\in\bdi\BP(a)$, so $|\zeta-\xi|=\del(a)e^{-m}$.  Put
\begin{gather*}
  R:=\sup\{r\ge\del(a)e^m : \{z\in\mfC  :  \del(a)e^{-m}<|z-\zeta|<r\}\subset\Om \}
  \intertext{and define}
  \A(a):=\{z\in\mfC  :  |\zeta-\xi| < |z-\zeta| < R\}\,.
\end{gather*}
Note that if $R=+\infty$, then $\A(a)=\mfC\sm\D[\zeta;|\zeta-\xi|]$ which should be viewed as a punctured disk on the Riemann sphere $\RS$.  Suppose $R<\infty$.  Then there exists a point $\eta\in\mfS^1(\zeta;R)\cap\bOm$, and now
\begin{gather*}
  \A(a)=\A(\zeta;d,\ups)=\{z\in\mfC  :  d e^{-\ups}<|z-\zeta|<d e^\ups\}\,,
  \intertext{where}
  d:=\lp |\zeta-\xi||\zeta-\eta| \rp^\half = \lp \del(a)e^{-m} R \rp^\half > \del(a)
  \intertext{and}
  \ups:=\log\frac{R}{d}=m+\log\frac{d}{\del(a)}>m=\bp(a)\,.
\end{gather*}

Assume $\xi\in\bdo\BP(a)$, so $|\zeta-\xi|=\del(a)e^{m}$.  Put
\begin{gather*}
  \veps:=\inf\{0<r\le\del(a)e^{-m} : \{z\in\mfC  :  r<|z-\zeta|<\del(a)e^{m}\}\subset\Om \}
  \intertext{and define}
  \A(a):=\{z\in\mfC  :  \veps < |z-\zeta| < |\zeta-\xi|\}\,.
\end{gather*}
Now if $\veps=0$, then $\A(a)=\D_*(\zeta;|\zeta-\xi|)$ which is a punctured disk.  Suppose $\veps>0$.  Then there exists a point $\eta\in\mfS^1(\zeta;\veps)\cap\bOm$, and
\begin{gather*}
  \A(a)=\A(\zeta;d,\ups)=\{z\in\mfC  :  d e^{-\ups}<|z-\zeta|<d e^\ups\}\,,
  \intertext{where}
  d:=\lp |\zeta-\xi||\zeta-\eta| \rp^\half = \lp \del(a)e^{m} \veps \rp^\half < \del(a)
  \intertext{and}
  \ups:=\log\frac{d}{\veps}=m+\log\frac{\del(a)}{d}>m=\bp(a)\,.
\end{gather*}

We summarize the above.
\begin{itemize}
  \item  $\A(z)\in\mcA^2_\Om$ is either a punctured disk in $\RS$ (which is a very large fat annulus:-) or $\A(z)=\A(\zeta;d,\ups)$ with $\zeta\in\B(z)$, some $d>0$, and some $\ups\ge\bp(z)$.  Thus we can always talk about $\core_q\A(z)$ where either $q>0$ in the first (two) case(s) or $0<q<\ups$ in the last case.
  \item  Also, $z$ lies on the boundary of a concentric subannulus of $\A(z)$ (in fact, $z$ lies on the boundary of a collar) that has modulus $\bp(z)$.  In particular,
  \begin{equation}\label{E:z in core bdy}
    z\in\bd\core_{\bp(z)}\A(z)\subset\overline\core_{\bp(z)}\A(z)\,.
  \end{equation}
\end{itemize}
\flag{perhaps indicate values of $\bp$ in certain cores of $\A(z)$?}  Notice that we can only apply the \BPEP\ to obtain estimates when $\A(z)$ is non-degenerate, but when $\A(z)$ is degenerate, we actually have a precise formula for $\bp$.  For example, if say $\A(a)=\D_*(0;R)$, then for $z\in\D_*(0;R/2)$ we have $\bp(z)=\log(R/|z|)$.

We conclude this subsection with the following elementary, but surprisingly useful, observation.  We often apply this result to curves.

\begin{lma} \label{L:bp ge tau} %
Fix $\tau>8\log2$.  Let $E$ be a connected subset of $\Om$.  Suppose that $\bp\ge\tau$ in $E$.  Then
\[
  E \subset \bigcap_{z\in E} \overline\core_{\tau/2}\A(z)\,.
\]
\end{lma} 
\begin{proof}%
Fix a point $p\in E$.  We examine the three cases depending on whether $\A(p)\in\mcA^2_\Om$ is a punctured disk, the complement of a disk, or a non-degenerate annulus.  Suppose $\A(p)=\D_*(0;R)$.  Assume $|p|=1$, so $R=e^{\bp(p)}\ge e^\tau$.  We claim that $E\subset\D[0;e^{-\tau}r]=\overline\core_\tau\A(p)$.

Note that $E\subset\D_*(0;R/2)$, for if not then we could find a point $z\in E$ with $|z|=R/2$ but then $\tau\le\bp(z)\le\log2$.  Thus for all $z\in E$, $\tau\le\bp(z)\le\log(R/|z|)$ whence $|z|\le e^{-\tau}R$ as asserted.

\smallskip

Suppose $\A(p)=\mfC\sm\D[0;R]$.  Again, assume $|p|=1$.  Then there is a point $\xi\in\bOm$ with $|\xi|=R=e^{-\bp(p)}\le e^{-\tau}$.  We claim that $E\subset\mfC\sm\D(0;\half e^\tau R)=\overline\core_{\tau-\log2}\A(p)\subann\core_{\tau/2}\A(p)$.

To see this, we first show that $E\subset\mfC\sm\D[0;3R]$.  For suppose there is a point $z\in E$ with $|z|=3R$.  Let $\zeta\in\B(z)$.  Then
\[
  |\zeta|\le R \quad\text{and}\quad 2R\le|z-\zeta|=\del(z)\le|z|\le3R\,.
\]
Assume $|\zeta|\ge R/2$.  Then
\begin{gather*}
  \bp(z)\le\bigl|\log\frac{\del(z)}{|\zeta|}\bigr|=\log\frac{\del(z)}{|\zeta|}\le\log\frac{2\del(z)}{R} \le\log6
\end{gather*}
which implies the contradiction $\tau\le\log6$.
Assume $|\zeta|\le R/2$.  Then
\begin{gather*}
  \half\,R \le |\xi|-|\zeta| \le |\zeta-\xi| \le |\zeta|+|\xi| \le \frac32\,R
  \intertext{so again}
  \bp(z)\le\bigl|\log\frac{\del(z)}{|\zeta-\xi|}\bigr| = \log\frac{\del(z)}{|\zeta-\xi|} \le\log\frac{2\del(z)}{R} \le\log6\,.
\end{gather*}

Now we repeat the above argument for an arbitrary point $z\in E$.  We still have $|\zeta|\vee|\zeta-\xi|\le2R\le\del(z)\le|z|$ and thus
\[
  \tau\le\bp(z)\le\log\frac{2|z|}R\,, \quad\text{whence}\quad  |z|\ge\half\,e^\tau R
\]
as asserted.

\smallskip

Finally, suppose $A:=\A(p)=\A(0;1,\ups)$ is a non-degenerate annulus.  We claim that $E\subset\overline\core_{\tau/2}(A)$.  Here $\ups\ge\bp(z)\ge\tau$ and $\tau/2\ge4\log2=\log16$, so we can appeal to the BPEP and especially Remarks~\ref{R:bp ests}.  Recall from \eqref{E:z in core bdy} that $p\in\bd\core_{\bp(p)}(A)\subset\overline\core_\tau(A)\subann \core_{\log16}(A)$.

Since $\bp\le8\log2$ on $\bd\core_{\log16}(A)$ (see \rf{R:bp ests}(b)), it follows that $E\subset\core_{\log16}(A)$.  Thus by \rf{R:bp ests}(c), if there were a point $z$ in $E\sm\overline\core_{\tau/2}(A)$, we would get $\tau\le\bp(z)<\tau$; therefore, $E\subset\overline\core_{\tau/2}(A)$.
\end{proof}%

\subsection{The \emph{ABC} Property}  \label{s:ABC} %
A path $\mfR\supset I\xra{\gam}\Om$ has the \emph{ABC property}\footnote{\emph{ABC} is short for \emph{annulus bounce or cross}.}, with parameters $\mu>0$ and $\nu>0$, \ifff for each compact subpath $\alf$ of $\gam$ and for each annulus $A:=\A(o;d,\nu)\in\mcA_\Om$ such that $\mfS^1(A)\supset\bd\alf$, we have $|\alf|\subset\A(o;d,\mu)$.

The ABC property implies that the path can cross a moderate size annulus

 at most once, so if the path enters deep into an annulus, it either stays there  or it crosses (once) and never returns.  In particular, if the path goes near an ``isolated island or archipelago'' of $\Om^c\cup\{\infty\}$, then it stays near it.  Here are various precise statements of this phenomenon.

\begin{lma} \label{L:ABC} %
Suppose a path $\mfR\supset I\xra{\gam}\Om$ has the \emph{ABC property} with parameters $\mu\ge\nu>0$.  Let $A:=\A(o;d,m), A_\nu:=\A(o;d,\nu), A_\mu:=\A(o;d,\mu)$.  Suppose $\alf$ is a compact subpath of $\gam$.
\begin{aenum}
  \item  Assume $A_\nu\in\mcA_\Om$.  If $m\ge\mu$, then $\alf$ crosses $A$ at most once. 
  \item  Assume $\band_{2\nu}(A)\in\mcA_\Om$.  If $\bd\alf\subset\bA$, then $|\alf|\subset\band_{2\mu}(A)$.
  \item  Assume $A\in\mcA_\Om$ and $m\ge\mu+\nu$.  Suppose $\bd\alf\cap\A_\mu=\emptyset$.  Then
  \[
    \text{either}\quad |\alf|\cap\mfS^1(A)=\emptyset \quad\text{or $\quad\mfS^1(A)$ separates $\bd\alf$}\,.
  \]
  \hspace{1.3em}In fact, if $|\alf|\cap\mfS^1(A)\ne\emptyset$, then $A_\mu$ separates $\bd\alf$.
  \item  Assume $A\in\mcA_\Om$ and $m>2\mu$.  Suppose $\bd\alf\cap A=\emptyset$.  Then
  \[
    \text{either}\quad |\alf|\cap\overline\core_{2\mu}(A)=\emptyset \quad\text{or $A$ separates $\bd\alf$}\,.
  \]
\end{aenum}
\end{lma} 
\begin{proof}%
To validate (a), suppose $\alf$ crosses $A$ twice.  Then there is a subpath $\beta$ of $\alf$ with $\bd\beta\subset\mfS^1(A)=\mfS^1(A_\nu)$ and such that $\bd\beta\cap\bA\ne\emptyset$.  The ABC property asserts that $|\beta|\subset A_\mu$, which in turn implies that $\cA\csubann A_\mu$, so $m<\mu$.

\smallskip

Item (b) follows from (a), because each collar of $A$ relative to $\band_{2\nu}(A)$ belong to $\mcA_\Om$ (and has modulus $2\nu$), so $\alf$ cannot cross either collar of $A$ relative to $\band_{2\mu}(A)$ (as each of these has modulus $2\mu$).

\smallskip

To verify (c), assume $\mfS^1(A)$ does not separate $\bd\alf$.  Then $A_\mu$ does not separate $\bd\alf$, and so the endpoints of $\alf$ both lie on the ``same side'' of $A_\mu$ (i.e., both are inside or both outside).

Next, suppose $|\alf|\cap\mfS^1(A)\ne\emptyset$ and pick a point $p\in|\alf|\cap\mfS^1(A)$.  We select a subpath $\beta$ of $\alf$ by moving away from $p$ towards each of the endpoints of $\alf$ until we first meet $\bA_\mu$. (Technically, here we should work with a parametrization of $\alf$.)  Thus we have a subpath $\beta$ of $\alf$ with $\bd\beta\subset C$ where $C$ is a component of $\bA_\mu$ ($C$ is the inner boundary circle if both endpoints of $\alf$ lie inside $\mfS^1(A)$ and otherwise $C$ is the outer boundary circle).

Now $B_\nu:=\band_\nu(C)\in\mcA_\Om$ and $\beta$ has endpoints on $C=\mfS^1(B_\nu)$, but as $p\in|\beta|$, $|\beta|\not\subset\band_\mu(C)$ which contradicts the ABC property for $\gam$.

\smallskip

To prove (d), suppose $|\alf|\cap\overline\core_{2\mu}(A)\ne\emptyset$.  If $A$ did not separate $\bd\alf$, then $\alf$ would cross one of the collars of $\core_{2\mu}(A)$ relative to $A$ \emph{twice}, but each of these collars has modulus $2\mu$ and this would contradict (a).
\end{proof}%

Item \ref{L:ABC}(d) above is the origin for the term ``ABC property''.  It says that for each sufficiently fat annulus $A\in\mcA_\Om$, each subpath $\alf$ (with endpoints not in $A$) either crosses $A$ or bounces off $A$ in the sense that it misses $\core_{2\mu}(A)$.  This ``ABC property'' is a fundamental characteristic of hyperbolic and quasihyperbolic geodesics and quasi-geodesics.

\smallskip

Now we demonstrate that both hyperbolic and quasihyperbolic geodesics have the ABC property.  \Va\ established this for the quasihyperbolic metric; see \cite[Lemma~3.6]{V-hyp-unif}.  Other metrics with this property include the Ferrand and Kulkarni-Pinkhall metrics.

\begin{prop} \label{P:ABC} %
In any hyperbolic plane domain: 
\begin{itemize}
  \item[(a)] quasihyperbolic geodesics have the $(\pi,\log2)$-ABC property,
  \item[(b)] hyperbolic geodesics have the $(3\kk,5/2)$-ABC property.
\end{itemize}
\end{prop} 
\begin{proof}%
For the reader's convenience, to motivate our argument for (b), and since our constants are different, we first prove (a).
Let $A=\A(o;d,\nu)\in\mcA_\Om$ and fix points $a,b\in\mfS^1(A)$.  We assume $o=0$ and $d=1$; so $\{e^{-\nu}<|z|<e^\nu\}\subset\Om$, $0\in\Om^c$, and  $|a|=1=|b|$.

Consider any quasihyperbolic geodesic $[a,b]_k$.  Let $\alf$ be (one of) the shorter subarc(s) of $\mfS^1$ with endpoints $a$ and $b$.
Suppose $\nu:=\log2$ (so, $\{1/2<|z|<2\}\subset\Om$).

Let $c\in[a,b]_k$.  We claim that $e^{-\pi}<|c|<e^\pi$.  First, note that for all $|z|=1$, $\del(z)\ge1/2$, so
\[
  k(a,b) \le \ell_k(\alf) = \int_\alf \frac{|dz|}{\del(z)} \le 2\,\ell(\alf) \le 2\pi\,.
\]
In fact, we may assume that $k(a,b)<2\pi$.
Writing $k_*$ for the quasihyperbolic metric in $\Co\supset\Om$, we have (for all points $z\in\Om$)
\begin{gather*}
  k(z,a)\ge k_*(z,a) \ge k_*(|z|,|a|) = \bigl|\log|z|\bigr|
  \intertext{and similarly, $k(z,b)\ge\bigl|\log|z|\bigr|$, so}
  2\bigl|\log|c|\bigr| \le k(a,c)+k(c,b)=k(a,b) < 2\pi
\end{gather*}
as asserted.

Now suppose $\nu:=5/2$ and consider any hyperbolic geodesic $[a,b]_h$.  First, by enlarging $\nu$, if necessary, we may assume that $\bA\cap\bOm\ne\emptyset$.  That is, we may assume that $A=\A(0;1,m)$ for some $m\ge\nu=5/2$ and that there exists a point $\xi\in\bA\cap\bOm$; so, $\bigl|\log|\xi|\bigr|=m$.

Let $c\in[a,b]_h$.  We claim that $e^{-3\kk}<|c|<e^{3\kk}$.  It suffices to verify the lower bound, for then the upper bound follows by considering the image of $\Om$ under the inversion $z\mapsto1/z$.  We may assume $|c|<1$, and---by considering a subarc of $[a,b]_h$ if necessary---that $[a,b]_h\subset\mfD_*\cup\{a,b\}$.

Now, as $|\alf|\subset A\subset\Om$, the hyperbolic length of $\alf$ in $A$ is at least the hyperbolic length of $\alf$ in $\Om$.  Since the hyperbolic metric in $A$ is the constant $\pi/2m$ on the center circle of $A$ (e.g., see \cite{BP-beta} or \cite[\S4.E, p.89]{HMM-mob3}), we thus obtain
\[
  h(a,b) \le \ell_{(\Om,h)}(\alf) \le \ell_{(A,h)}(\alf) \le \frac{\pi^2}{2m}\,.
\]
In fact, we may and do assume that $h(a,b)<\pi^2/2m$.

Next, we claim that for all $z\in[a,b]_h$, $\lam(z)\ge\left[|z| \bigl( \kk+m+\log(1/|z|) \bigr) \right]^{-1}$.  To see this, recall that $0\in\Om^c$ and also $\xi\in\bOm$ with $\bigl|\log|\xi|\bigr|=m\ge5/2$.  Thus $\Om\supset\mfC_{0\xi}$ and so using the change of variables $w=z/\xi$ in conjunction with \rf{F:lam_01} we deduce that
\begin{align*}
  \lam(z) &\ge \lam_{0\xi}(z) = \frac1{|\xi|}\, \lam_{01}(w) \\
  & \ge \frac1{|\xi|}\, \frac1{|w| \bigl( \kk + |\log|w|| \bigr)} = \frac1{|z| \bigl( \kk + |\log|z/\xi|| \bigr)} \\
  & \ge \frac1{|z| \bigl( \kk + m + \log(1/|z|) \bigr)}\,;
\end{align*}
here the last inequality above follows from the observation that
\[
  \bigl|\log|z/\xi|\bigr|\le\bigl|\log|z|\bigr|+\bigl|\log|\xi|\bigr|=m+\log\bigl(1/|z|\bigr)\,.
\]

Using the above estimate, we calculate that
\begin{gather*}
  h(a,c) = \int_{[a,c]_h} \lam\,ds \ge \int_{|c|}^1 \frac{|dz|}{|z| \Bigl( \kk+m+\log\ds\frac1{|z|} \Bigr)} \\
  = \int_{\kk+m}^{\kk+m+\log(1/|c|)} \frac{|dw|}{|w|} = \log \lp 1 + \frac{\log(1/|c|)}{\kk+m} \rp
  \intertext{and similarly}
  h(b,c) \ge \log\lp 1 + \frac{\log(1/|c|)}{\kk+m} \rp\,.
\end{gather*}
It now follows from our initial upper bound for $h(a,b)$ that
\begin{gather*}
  2 \log\lp 1 + \frac{\log(1/|c|)}{\kk+m} \rp \le h(a,b) < \frac{\pi^2}{2m}
  \intertext{whence}
  \log\frac1{|c|} < (\kk+m) \Bigl[ \exp \bigl(\frac{\pi^2}{4m}\bigr) -1 \Bigr]
  \intertext{and therefore}
  \frac1{|c|} < \exp\Bigl( (\kk+m) \Bigl[ \exp \bigl(\frac{\pi^2}{4m}\bigr) -1 \Bigr] \Bigr) \le e^{3\kk}\,.
\end{gather*}

It remains to corroborate the final inequality immediately above.  Since $m\ge5/2$, setting $p:=e^{\pi^2/10}$ we have $\pi^2/4m\le\log p$.  Since $u\mapsto(e^u-1)/u$ is increasing for $u\ge0$, it follows that
\begin{gather*}
  \frac{\exp(\pi^2/4m)-1}{\pi^2/4m} \le \frac{e^u-1}{u}\Big|_{u=\log p}=\frac{10}{\pi^2}\,(p-1)
  \intertext{so}
  m \bigl( \exp(\pi^2/4m)-1 \bigr) \le \frac{\pi^2}4\, \frac{10}{\pi^2}\,(p-1) = \frac52 \, (p-1) < 4.3 < \kk\,.
  \intertext{But also,}
  \exp(\pi^2/4m)-1 \le p-1 \le 2\,, \quad\text{so}\quad \kk \lp \exp(\pi^2/4m)-1 \rp \le 2\kk
  \intertext{and thus}
  (\kk+m) \bigl( \exp(\pi^2/4m)-1 \bigr) \le 3\kk\,.  \qedhere
\end{gather*}
\end{proof}%

\begin{rmks}  \label{R:ABC} %
Using the same arguments as above, it is not hard to demonstrate that hyperbolic and quasihyperbolic quasi-geodesics have the ABC property.  In fact, quasihyperbolic rough quasi-geodesics have the ABC property.  (However this is not true in general even for hyperbolic rough geodesics; see \rf{X:D*}(c).  This seems to be the main large scale difference between hyperbolic and quasihyperbolic geometry.)   More precisely:

\smallskip\noindent(a)  a quasihyperbolic $(\Lam,C)$-rough-chordarc path has the ABC property with parameters $\mu=\pi\Lam + C/2$ and $\nu=\log2$;

\smallskip\noindent(b)  a hyperbolic $\Lam$-chordarc path has the ABC property with parameters $\mu=7e^\Lam$ and $\nu=5/2$. (Note that $7e>\mu_o:=3\kk$.)
\end{rmks} %

Everywhere below $\mu_o:=\max\{\pi,3\kk,\log16\}=3\kk\le15$.  This absolute constant $\mu_o$ is an ABC parameter for both hyperbolic and quasihyperbolic geodesics.  Henceforth, we reserve the symbol $\mu$ for a general ABC parameter.  In \rf{S:lengths} (and later) $\mu=\mu(\Lam)$ will depend on a chordarc parameter $\Lam$; usually we take $\mu:=7e^\Lam$.


\subsection{Disjoint Cores \& More}  \label{s:DC} %

Here we establish some basic properties of the annuli $\A(z)$ and their cores; these are essential for our proof of \rf{TT:geos}.

In our proofs of \rfs{TT:geos} and \ref{TT:GrHyp}, we often use \rf{P:BPEP} to estimate $\bp$ in the cores of certain $\A(z)$.  For this, we need to know that these $\A(z)$ are not punctured disks.

\begin{lma} \label{L:not punxd} %
Let $\mu>0, \nu>0, \sig>\log6$, and $\tau>\mu+(\nu\vee2\sig)$.  Let $\gam$ be a compact path in $\Om$ with the $(\mu,\nu)$-ABC property. 
Suppose that $\bp\le\sig$ on $\bd\gam$ and that there exists a point $z\in|\gam|$ with $\bp(z)\ge\tau$.  Then $\A(z)$ is a non-degenerate annulus. 
\end{lma} 
\begin{proof}%
Assume that $A:=\BP(z)=\A(0;1,\bp(z))$; so $0\in\bOm$, $\del(z)=|z|=1$, and there is a point $\xi\in\bA\cap\bOm$ with $\bp(z)=\bigl|\log|\xi|\bigr|$.  Recall that we enlarge $A$ to get $\A(z)$.

First, \rf{L:ABC}(c) (with $\alf:=\gam$) tells us that $\mfS^1=\mfS^1(A)$ separates $\bd\gam$; we check the details.  First, $\bp(z)\ge\tau>\mu+\nu$ holds.  Next
\begin{gather*}
  A_\mu:=\A(0;1,\mu)=\A(0;1,\bp(z)-q)=\core_q\A(z) \quad\text{where $q:=\bp(z)-\mu\ge\tau-\mu$}
  \intertext{and so by \rf{L:bp ge}, in $A_\mu$ we have}
  \bp>\half\,q\ge\half(\tau-\mu)>\sig\,.
\end{gather*}
It follows that $\bd\gam\cap A_\mu=\emptyset$.  Since $z\in|\gam|\cap\mfS^1$, we deduce that $\mfS^1$ separates $\bd\gam$.


Assume $\bd\gam=\{a,b\}$, where, say, $a$ lies inside $\mfS^1$ (i.e., $a\in\mfD$) and $b$ lies outside $\mfS^1$ (i.e., $b\in\mfC\sm\cmfD$).  There are two cases now depending on the location of $\xi$.

Suppose $\xi\in\bdo A$, so $|\xi|=e^{\bp(z)}$.  We claim that $\bOm\cap\mfD\ne\{0\}$; so, suppose $\bOm\cap\mfD=\{0\}$.  Then $\del(a)=|a|$ and
\[
  \bp(a)=\Bigl| \log \frac{|a|}{|\xi|} \Bigr| = \log \frac{|\xi|}{|a|} = \bp(z)-\log|a|
\]
which implies $\log|a|=\bp(z)-\bp(a)\ge\tau-\sig>0$ and contradicts $|a|<1$.

Suppose $\xi\in\bdi A$, so $|\xi|=e^{-\bp(z)}$.  We claim that $\bOm\sm\cmfD\ne\emptyset$; so, suppose $\bOm\sm\cmfD=\emptyset$.  Pick $\vth,\eta\in\bOm$ with $\del(b)=|b-\vth|$ and $\bp(b)=\bigl|\log\bigl(\del(b)/|\vth-\eta|\bigr)\bigr|$.  Note that $|\vth|\le e^{-\bp(z)}$ and $|\eta|\le e^{-\bp(z)}$.

Now
\begin{gather*}
  \del(b)=|b-\vth|\ge|b|-|\vth|\ge1-e^{-\bp(z)}\ge1-e^{-\tau}\ge\half
  \intertext{and}
  |\vth-\eta|\le2 e^{-\bp(z)}\le2e^{-\tau}<\half
  \intertext{so}
  \sig\ge\bp(b)=\log \frac{\del(b)}{|\vth-\eta|} \ge \log\del(b) + \tau-\log2
  \intertext{which implies the contradiction}
  \half \le \del(b) \le 2 e^{\sig-\tau} < \half\,.  \qedhere
\end{gather*}
\end{proof}%

The following results will be useful in our proofs of our main theorems.  First, we show that certain annulus cores separate certain points.

\begin{lma} \label{L:cores separate} %
Let $\mu>0, \nu>0, \sig>\log6$, and $\tau>(\mu+\nu)\vee(2\mu+2\sig)$.  Let $\gam$ be a compact path in $\Om$ with the $(\mu,\nu)$-ABC property.  Suppose that $\bp\le\sig$ on $\bd\gam$ and that there exists a point $z\in|\gam|$ with $\bp(z)\ge\tau$.  Then $\core_{2\sig}\A(z)$ separates $\{a,b\}$.
\end{lma} 
\begin{proof}%
According to \rf{L:not punxd}, $\A(z)$ is a non-degenerate annulus with finite modulus, say $2\ups:=\md\bigl(\A(z)\bigr)$; here $\ups\ge\bp(z)\ge\tau>2\mu+2\sig$.  We appeal to \rf{L:ABC}(d) with $A:=\core_{2\sig}\A(z)$.  Note that
\[
  \half\md(A)=\half(2\ups-4\sig)=\ups-2\sig>2\mu\,.
\]
Next, according to \rf{L:bp ge}, $\bp>\sig$ in $A$, so $\bd\gam\cap A=\emptyset$.  Finally, $z\in|\gam|\cap\overline\core_{2\mu}(A)$, because by \eqref{E:z in core bdy},
\[
 z\in\bd\core_{\bp(z)}\A(z) \subset \overline\core_{\bp(z)}\A(z) \subset \overline\core_\tau\A(z) \subset \overline\core_{2\mu+2\sig}\A(z)=\overline\core_{2\mu}(A)\,;
\]
here the two right-most containments hold because $\bp(z)\ge\tau\ge2\mu+2\sig$.
\end{proof}%

Here is a useful consequence of the above.

\begin{lma} \label{L:large bp} %
Fix $\sig>\mu\ge\nu>0$ with $\sig\ge\log6$.  Let $\alf$ and $\beta$ be compact paths in $\Om$ with common endpoints.  Suppose $\alf$ has the $(\mu,\nu)$-ABC property, and that there exists a point $z\in|\alf|$ with $\bp(z)\ge4\sig$.  Then there exists a point $w\in|\beta|$ with $\bp(w)>\sig$.
\end{lma} 
\begin{proof}%
We may and do assume that $\bp\le\sig$ on $\bd\alf=\bd\beta$.  Since $\tau:=4\sig>2\mu+(\nu\vee2\sig)$, we can appeal to \rfs{L:not punxd} and \ref{L:cores separate} to assert that $\A(z)\in\mcA^2_\Om$ is a non-degenerate annulus and that $\core_{2\sig}\A(z)$ separates $\bd\beta$.  In particular then, 
there exists a point $w\in|\beta|\cap\core_{2\sig}\A(z)$, and then by \rf{L:bp ge}, $\bp(w)>\sig$.
\end{proof}%

\begin{cor}  \label{C:large bp} %
Let $\Lam\ge1$, $C_h>\mu_h:=7e^\Lam$, $C_k>\mu_k:=\pi\Lam$.\footnote{When $\Lam=1$ we can take $\mu_h=3\kk$.}  Let $a,b\in\Om$.  Suppose $\gam^h$ and $\gam^k$ are hyperbolic and quasihyperbolic $\Lam$-chordarc paths with endpoints $a,b$.  Then
\begin{gather*}
  \bp\le C_h \;\text{ on $|\gam^k|$} \implies \bp< 4C_h \;\text{ on $|\gam^h|$}
  \intertext{and similarly,}
  \bp\le C_k \;\text{ on $|\gam^h|$} \implies \bp< 4C_k \;\text{ on $|\gam^k|$}\,.
\end{gather*}
\end{cor} 

Finally, we prove that certain cores are always disjoint.

\begin{prop} \label{P:dsjt cores} %
Let $\mu>0, \nu>0, \sig>\log6, q>16\log2$, and $\tau>q\vee(\mu+\nu)\vee(4\mu+4\sig)$.  Let $\alf$ be an arc in $\Om$ with the $(\mu,\nu)$-ABC property.  Suppose there are successive points $a, z, p, w, b$ along $\alf$ with $\bp\le\sig$ at each of $a,p,b$ and $\bp\ge\tau$ at each of $z,w$.  Then
\[
  \core_q\A(z)\cap\core_q\A(w)=\emptyset\,.
\]
\end{prop} 
\begin{proof}%
According to \rfs{L:not punxd} and \ref{L:cores separate}, $A:=\A(z)$ and $B:=\A(w)$ are non-degenerate annuli, $\core_{2\sig}(A)$ separates both $\{a,p\}$ and $\{a,b\}$, and $\core_{2\sig}(B)$ separates both $\{p,b\}$ and $\{a,b\}$.  Also, $q<\tau\le\bp(z)\wedge\bp(w)$ implies that $z\in\core_q(A)$ and $w\in\core_q(B)$.

\smallskip

Now we have two cases, ``concentric'' or ``spectacles'', depending on the positions of the points $a$ and $b$ relative to the separating annuli $\core_{2\sig}(A)$ and $\core_{2\sig}(B)$.  In the ``concentric'' case, we assume that one of these points, say $a$, lies inside both of the separating annuli.  Then $b$ is outside of both annuli while $p$ is outside $\core_{2\sig}(A)$ but inside $\core_{2\sig}(B)$.

We claim that $w\not\in\overline\core_{\tau/2}(A)$, and therefore by \rf{R:bp ests}(c), $w\not\in\core_{\log16}(A)$.  To see this, we first note that as $p$ is outside $\core_{2\sig}(A)$, $\alf[a,p]$ crosses $\core_{2\sig}(A)$.  According to \rf{L:ABC}(a), $\alf[p,b]$ cannot cross the outer collar of $\core_{2\sig}(A)\sm\overline\core_{2\sig+2\mu}(A)$ (i.e., the outer collar of $\core_{2\sig+2\mu}(A)$ relative to $\core_{2\sig}(A)$).  Thus $w\not\in\overline\core_{2\sig+2\mu}(A)$.  Since $2\sig+2\mu<\tau/2$, $\overline\core_{\tau/2}(A)\subset\core_{2\sig+2\mu}(A)$, so $w\not\in\overline\core_{\tau/2}(A)$ (and hence $w\not\in\core_{\log16}(A)$).

Suppose there were a point $x\in\core_q(A)\cap\core_q(B)$.  Then, as $w\in\core_q(B)$, there would exist a path $\gam$ from $w$ to $x$ with $|\gam|\subset\core_q(B)$.  By \rf{R:bp ests}(c), $\bp>q/2$ along $\gam$.  However, $\gam$ would meet $\bd\core_{\log16}(A)$ at some point $y$, and then by \rf{R:bp ests}(b) we would obtain $q/2<\bp(y)\le8\log2$ which contradicts $q>16\log2$.

\smallskip
It remains to consider the ``spectacles case'' where neither $a$ nor $b$ lies inside both separating annuli $\core_{2\sig}(A)$ and $\core_{2\sig}(B)$.  If $a$ lies inside (so $b$ lies outside) $\core_{2\sig}(A)$ and $b$ lies inside (so $a$ lies outside) $\core_{2\sig}(B)$, then \rf{C:cores disjoint} asserts that $\core_q(A)\cap\core_q(B)=\emptyset$.  The other possibility cannot occur by \rf{L:NOflip-flop} below.
\end{proof}%

The following shows that the ``flip-flop spectacles'' case never occurs.

\begin{lma} \label{L:NOflip-flop} %
Let $\mu>0, \nu>0, \sig>\log6$, and $\tau>(\mu+\nu)\vee(2\mu+2\sig)\vee(3\sig)$.  Let $\alf$ be an arc in $\Om$ with the $(\mu,\nu)$-ABC property.  Suppose there are successive points $a, z, p, w, b$ along $\alf$ with $\bp\le\sig$ at each of $a,p,b$ and $\bp\ge\tau$ at each of $z,w$.  Then it cannot be that
\[
  \text{$a$ lies inside $\core_{2\sig}\A(w)$ and $b$ lies inside $\core_{2\sig}\A(z)$}\,.
\]
\end{lma} 
\begin{proof}%
According to \rfs{L:not punxd} and \ref{L:cores separate}, $A:=\A(z)$ and $B:=\A(w)$ are non-degenerate annuli in $\mcA^2_\Om$, $\core_{2\sig}(A)$ separates both $\{a,p\}$ and $\{a,b\}$, and $\core_{2\sig}(B)$ separates both $\{p,b\}$ and $\{a,b\}$.

Suppose it were true that $a$ lies inside $\core_{2\sig}(B)$ and $b$ lies inside $\core_{2\sig}(A)$.  This would mean that $b$ lies outside $\core_{2\sig}(B)$, $a$ lies outside $\core_{2\sig}(A)$, and therefore $p$ lies inside \emph{both} cores.  In particular, the two cores would have non-disjoint ``insides''.  It is not difficult to check that the two cores cannot be disjoint, so $\core_{2\sig}(A)\cap\core_{2\sig}(B)\ne\emptyset$.

Now we appeal to \rf{L:both centers inside}.  Assume $A=\A(z)=\A(\eta;d,r)$ and $B=\A(w)=\A(\vth;c,s)$; note that $r\wedge s\ge\tau>2\sig$ so the lemma applies (with $q:=2\sig$).  Then since $a$ lies outside $\core_{2\sig}(A)$ but inside $\core_{2\sig}(B)$,
\begin{gather*}
  d e^{r-2\sig} \le |a-\eta| \le |a-\vth|+|\vth-\eta| \le c e^{2\sig-s} + d e^{-r} \le 2d e^{2\sig-s} + d e^{-r}
  \intertext{so}
  e^\tau \le e^r \le 2 e^{4\sig-s} + e^{2\sig-r} \le e^{2\sig-\tau}(2e^{2\sig}+1)
\end{gather*}
but then $e^{4\sig}\le e^{2\tau-2\sig}\le 2e^{2\sig}+1$ which contradicts $\sig>\log6$.
\end{proof}%
%
\section{Length Estimates}  \label{S:lengths} 

Below we provide explicit estimates for the Euclidean, hyperbolic, and quasihyperbolic lengths of certain quasi-geodesics.  As discussed in \rf{S:Intro}, a quasi-geodesic is ``bad'' if $\bp$ is large at each of its points, and ``good'' otherwise.  Roughly speaking, our primary goal here is to show that any two ``associated bad quasi-geodesics'' have comparable hyperbolic lengths and comparable quasihyperbolic lengths.  That is, we establish the length estimates $\ell_g(\beta^h)\eqx\ell_g(\beta^k)$ (for $g=h,k$) for any pair of quasi-geodesics $\beta^h, \beta^k$ that cross a subannulus deep inside some large fat separating annulus in $\Om$; see \rfs{P:D*:ell-bad-arcs}, \ref{P:Dc:ell-bad-arcs}, \ref{P:A:ell-bad-arcs}, which in turn give \rf{TT:bad??}.  These estimates play crucial roles in our proofs of both \rfs{TT:geos} and \ref{TT:GrHyp}.

For a simple special case, suppose that $\beta^h$ and $\beta^k$ are a hyperbolic and a quasihyperbolic $\Lam$-chordarc path with the \emph{same} endpoints $a,b$.  Then
\[
  k(a,b)\le\ell_k(\beta^k)\le\Lam k(a,b)\le\Lam\ell_k(\beta^h)\,,\quad\text{so}\quad \Lam^{-1}\le\frac{\ell_k(\beta^h)}{\ell_k(\beta^k)}
\]
and likewise $\Lam^{-1} \le \ell_h(\beta^k)/\ell_h(\beta^h)$.  Below we provide similar explicit upper and lower bounds, and we do not presume that the endpoints coincide.

\smallskip

Before diving into the details, we discuss the setting and some notational conventions.  We consider a hyperbolic quasi-geodesic $\beta^h$ and a quasihyperbolic quasi-geodesic $\beta^k$; the endpoints $\bd\beta^h=\{a^h,b^h\}$ and $\bd\beta^k=\{a^k,b^k\}$ need not be the same, but, we do require that corresponding endpoints satisfy certain conditions, such as being concentric, as explained below.

When it does not matter which quasi-geodesic we are discussing, we may write $\beta^g$ (where $g\in\{h,k\}$) or---dropping the superscript $g$---sometimes we even just write $\beta$, and similarly write $a,b$ for the endpoints of $\beta$.

We assume that both quasi-geodesics $\beta^h$ and $\beta^k$ join the boundary circles of a concentric subannulus that lies deep inside some large fat separating annulus $A$ in $\Om$; in our proof of \rf{TT:geos}, this will often be $\A(z)$ for some $z\in\Om$, but now this is not relevant.  What is important is that there are three cases depending on whether $A$ is a punctured disk, the complement of a closed disk (so a punctured disk in $\RS$), or a non-degenerate annulus (with finite modulus).  Thus we consider the cases where
\[
  A=\D_*(0;R) \quad\text{or}\quad  A=\mfC\sm\D_*[0;R] \quad\text{or}\quad A=\A(0;d,r)
\]
for some $R>0$ or some $d>0, r>0$.  In all three cases we assume $A\in\mcA_\Om$ and so the center $c(A)=0$ of $A$ lies in $\Om^c$.

We emphasize that, except when explicitly indicated otherwise, the endpoints of $\beta^h$ and $\beta^k$ need not be the same.  However, we do require that these endpoints lie on concentric circles centered at $c(A)$; i.e., that
\[
  |a^h|=|a^k| \quad\text{and}\quad |b^h|=|b^k|\,.
\]
There will be additional constraints in each of the three cases, but these essentially say that the subannulus that each $\beta$ crosses is ``deep inside'' $A$.

Again, our primary goal in this section is to prove that for both $g=h$ and $g=k$, $\ell_g(\beta^h)\eqx\ell_g(\beta^k)$ with \emph{explicit constants that depend only on the quasi-geodesic constant}; see \eqref{E:D*:ell-bad-arcs}, \eqref{E:D^c:ell-bad-arcs}, and \eqref{E:A:ell-bad-arcs}.  In each case we establish this by finding estimates for the four lengths $\ell_g(\beta)$ where $g=h$ or $g=k$ and $\beta=\beta^h$ or $\beta=\beta^k$.  We accomplish this task by writing each $\beta$ as a concatenation product $\beta=\Star\alf_j$ and finding estimates for $\ell_g(\alf_j)$ where $g=h$ or $g=k$ and $\alf_j=\alf^h_j$ or $\alf_j=\alf_j^k$.  Our procedure is basically the same in all three cases: we approximate $\ell_g(\alf_j)$ via estimates for $\ell(\alf_j)$ and estimates for each of $1/\del$ and $\lam$ (in certain subannuli that contain the sub-quasi-geodesics $\alf_j$).

\begin{rmks}  \label{R:ell alf ests} %
It is worthwhile to describe our procedure to estimate the lengths of each $\alf_j$ as we perform it repeatedly; see \eqref{E:D*:cp ell},\eqref{E:D*:cp lk},\eqref{E:D*:cp lh} and \eqref{E:D*:ell alf estimates},\eqref{E:D*:ellk alf estimates},\eqref{E:D*:ellh alf estimates} (for the punctured disk case) and \eqref{E:A:cp ell},\eqref{E:A:cp lk},\eqref{E:A:cp lh} and \eqref{E:A:ell alf estimates},\eqref{E:A:ellk alf estimates},\eqref{E:A:ellh alf estimates} (for the non-degenerate annulus case).

\smallskip\noindent(a)  The lower bounds---which hold for arbitrary paths---are straightforward.  For Euclidean length and quasihyperbolic length we use the elementary fact that the length of a path is bounded below by the distance between its endpoints; here we employ \eqref{E:k* ests} to estimate quasihyperbolic distances.  The lower bounds for hyperbolic length ($g=h$) follow from those for Euclidean length in conjunction with estimates for $\lam$ (in appropriate subannuli).

\smallskip\noindent(b)  The upper bounds for $\ell_k(\alf^k_j)$ also arise from quasihyperbolic distance estimates.  These upper bounds, in conjunction with estimates for $1/\del$ (in appropriate subannuli), then produce upper bounds for $\ell(\alf^k_j)$.

\smallskip\noindent(c)  The upper bounds for $\ell_h(\alf^h_j)$ follow similarly, although here we estimate hyperbolic distance via $\ell_h(\chi)$ for the appropriate circular-arc-segment path $\chi$.  These upper bounds, in conjunction with estimates for $\lam$, then produce upper bounds for $\ell(\alf^h_j)$.

\smallskip\noindent(d)
  The upper bounds for $\ell_h(\alf^k_j)$ and $\ell_k(\alf^h_j)$ now follow using estimates for $\lam$ and $1/\del$ and the upper bounds for $\ell(\alf^k_j)$ and $\ell(\alf^h_j)$, respectively.
\end{rmks} %

We start with the punctured disk case where we know exactly the quasihyperbolic metric as well as $\bp$ and we have good estimates for the hyperbolic metric.  After this we consider the non-degenerate annulus case.

\subsection{Quasi-Geodesics in Punctured Disks}  \label{s:QGs in D*}%
Here we establish the first set of length estimates that play crucial roles in our proofs of \rfs{TT:geos} and \ref{TT:GrHyp}.  Roughly speaking, any two quasi-geodesics with concentric endpoints deep inside some punctured disk in $\Om$ have comparable hyperbolic lengths and comparable quasihyperbolic lengths.  When we have geodesics, we can get improved constants since we know that the quasihyperbolic geodesic in $\Co$ are logarithmic spirals and can use \rf{L:k-geos in pxd plane}.


\begin{prop} \label{P:D*:ell-bad-arcs} 
For each $\Lam\ge1$, there are explicit constants that depend only on $\Lam$ with the following properties.  Let $\mu=\mu(\Lam)$ be the ABC parameter for hyperbolic $\Lam$-chordarc paths.\footnote{See \rfs{R:ABC}.}  Assume $D_*:=\D_*(0;R)\in\mcA_\Om$ for some $R>0$.
Let $\beta^h$ and $\beta^k$ be, respectively, any hyperbolic and quasihyperbolic $\Lam$-quasi-geodesics with concentric endpoints $\bd\beta^h=\{a^h,b^h\}$ and $\bd\beta^k=\{a^k,b^k\}$ in $\core_{2\mu+\log2}(D_*)$ that satisfy
\[
  |a^h|=|a^k| \le |b^k|=|b^h|\,.
\]
For both $g=h$ and $g=k$, $\ell_g(\beta^h)\eqx\ell_g(\beta^k)$ with explicit constants that depend only on $\Lam$.

\smallskip\noindent
More precisely: assuming that $\beta^h$ and $\beta^k$ are hyperbolic and quasihyperbolic $\Lam$-chordarc paths, then when $|b^h|\le e|a^h|$ and $\bd\beta^h=\bd\beta^k$,
\addtocounter{equation}{-1}  
\begin{subequations}\label{E:D*:ell-bad-arcs}
  \begin{gather}
    \frac1\Lam \le \frac{\ell_k(\beta^h)}{\ell_k(\beta^k)} \le 7 e^{2+4\mu}\Lam \quad\text{and}\quad
    \frac1\Lam \le \frac{\ell_h(\beta^k)}{\ell_h(\beta^h)} \le 24\mu e^{2+6\mu}\Lam                     \label{E:D*:cp ratios}
    \intertext{and when $|b^h|>e|a^h|$ (and possibly $\bd\beta^h\ne\bd\beta^k$),}
    \bigl((1+\pi)\Lam\bigr)^{-1} \le \frac{\ell_k(\beta^h)}{\ell_k(\beta^k)} \le 3 e^{4+4\mu}\Lam       \label{E:D*:ell_k ratio}
    \intertext{and}
    (51 e^{2\mu}\Lam)^{-1} \le \frac{\ell_h(\beta^k)}{\ell_h(\beta^h)} \le 491\mu e^{6\mu}\Lam\,.       \label{E:D*:ell_h ratio}
  \end{gather}
\end{subequations}
\end{prop} 
\begin{proof}%
By increasing $R$ if necessary, we may assume that $D_*\in\mcA^2_\Om$; so, there exists a point $\xi\in\bOm$ with $R=|\xi|$.  Note that in $\core_{\log2}(D_*)=\D_*(0;\half R)$ we have $\del=\del_*$ and $k=k_*$ where $k_*$ is quasihyperbolic distance in $\Co$.  Thus \eqref{E:k* ests} already provides estimates for $k(a,b)$ when $a,b\in\core_{\log2}(D_*)$.  Also, for $z\in\core_{\log2}(D_*)$, $\bp(z)=\log(R/|z|)$ which provides the hyperbolic metric estimates given in \rf{L:lam-del ests in D*}.  \flag{I believe I can prove similar estimates for $h(a,b)$, but it takes some effort!}

Let
\[
  A_0:=\{z:|a|\le|z|\le|b|\} \quad\text{and}\quad B_0:=\band_{2\mu}(A_0)=\{z:e^{-2\mu}|a|\le|z|\le e^{2\mu}|b|\}\,.
\]
Evidently, $A_0\subann B_0\subann\overline\core_{\log2}(D_*)$.  According to \rf{L:ABC}(b) in conjunction with \rf{P:ABC}(b), both quasi-geodesics $\beta^h$ and $\beta^k$ lie in $B_0$.

\bigskip 

Suppose first that $a:=a^h=a^k$, $b:=b^h=b^k$, and $|b|\le e|a|$. (We call this the ``close points'' case.)  Since the endpoints coincide, the two lower bounds in \eqref{E:D*:cp ratios} are trivial.  Thus we need only verify the asserted upper bounds.  From \eqref{E:k* le} we have
\[
  \frac{|a-b|}{|b|} \le k(a,b)=k_*(a,b) \le 3\, \frac{|a-b|}{|a|}\,.
\]

Put $\chi:=\chi_{ab}=\kap_{ac}\star[c,b]$, the circular-arc-segment path from $a$ through $c:=\bigl(|a|/|b|\bigr)b$ to $b$ as described as described in \eqref{E:chi}.  By \rf{L:ell chi}
\begin{gather*}
  \ell(\chi) \le 3|a-b|\,.
  \intertext{Next, for $z\in A_0$, $|a|\le|z|=\del(z)\le|b|$, so}
  \frac1{|a|}\ge\frac1{|z|}\ge\frac1{|b|} \quad\text{and}\quad  \frac1{\log(R/|a|)}\le\frac1{\log(R/|z|)}\le\frac1{\log(R/|b|)}
  \intertext{and thus by \rf{L:lam-del ests in D*}}
  \frac1{|b|\bigl(\kk+\log(R/|a|)\bigr) }\le \lam(z) \le \frac1{|a|\log(R/|b|)}\,.
  \intertext{Similarly, for $z\in B_0$, $e^{-2\mu}|a|\le|z|=\del(z)\le e^{2\mu}|b|$, so}
  \frac{e^{-2\mu}}{|b|\bigl(\kk+2\mu+\log(R/|a|)\bigr)}\le \lam(z) \le \frac{e^{2\mu}}{|a|\bigl(-2\mu+\log(R/|b|)\bigr)}\,.
\end{gather*}

\smallskip

Now we verify that $\ell_k(\beta^k)\eqx\ell_k(\beta^h)$ and $\ell_h(\beta^h)\eqx\ell_h(\beta^k)$.  More precisely, we establish:
\begin{subequations}\label{E:D*:cp ell}
  \begin{align}
    |a-b| &\le \ell(\beta^k) \le 3e^{1+2\mu}\Lam|a-b| \,,           \label{E:D*:cp ell k} \\
    |a-b| &\le \ell(\beta^h) \le 7e^{1+2\mu}\Lam|a-b| \,;           \label{E:D*:cp ell h}
  \end{align}
\end{subequations}
and also
\begin{subequations}\label{E:D*:cp lk}
  \begin{align}
    \frac{|a-b|}{|b|} &\le \ell_k(\beta^k) \le 3\Lam \frac{|a-b|}{|a|}   \,,                \label{E:D*:cp lkk} \\
    \frac{|a-b|}{|b|} &\le \ell_k(\beta^h) \le 7e^{1+4\mu}\Lam \frac{|a-b|}{|a|}  \,,       \label{E:D*:cp lkh}
  \end{align}
\end{subequations}
and
\begin{subequations}\label{E:D*:cp lh}
  \begin{align}
    \frac{e^{-2\mu}}{\kk+2\mu+\log(R/|a|)}\,\frac{|a-b|}{|b|} &\le \ell_h(\beta^h) \le
    \frac{3\Lam}{\log(R/|b|)}\,\frac{|a-b|}{|a|}\,,                                         \label{E:D*:cp lhh} \\
    \frac{e^{-2\mu}}{\kk+2\mu+\log(R/|a|)}\,\,\frac{|a-b|}{|b|} &\le \ell_h(\beta^k) \le
    \frac{3e^{1+4\mu}\Lam}{-2\mu+\log(R/|b|)}\,\frac{|a-b|}{|a|}\,.                         \label{E:D*:cp lhk}
  \end{align}
\end{subequations}

The six lower bounds above, which hold for any path joining $a$ and $b$, are straightforward (as explained in \rf{R:ell alf ests}(a)).  Thus we need only verify the upper bounds.  Noting that $|\beta^k|\subset B_0$ and using our above estimates for $\del^{-1}ds$ (in $B_0$) we obtain
\[
  \frac{\ell(\beta^k)}{e^{2\mu}|b|} \le \int_{\beta^k}\frac{ds}{\del}= \ell_k(\beta^k)\le\Lam k(a,b) \le 3\Lam\frac{|a-b|}{|a|}
\]
and so the upper bounds in both \eqref{E:D*:cp ell k} and \eqref{E:D*:cp lkk} hold.  With \eqref{E:D*:cp ell k} in hand, we use our above estimates for $\lam\,ds$ (in $B_0$) to obtain
\begin{align*}
  \frac{e^{-2\mu}\,\ell(\beta^k)}{|b|\bigl(\kk+2\mu+\log(R/|a|)\bigr)} &\le \int_{\beta^k} \lam\,ds = \ell_h(\beta^k) \\
    &\le \frac{e^{2\mu}\,\ell(\beta^k)}{|a|\bigl(-2\mu+\log(R/|b|\bigr)} \\
    &\le \frac{3e^{1+4\mu}\Lam}{-2\mu+\log(R/|b|)}\,\frac{|a-b|}{|a|}
\end{align*}
which gives the upper bound in \eqref{E:D*:cp lhk}.

\smallskip

Next, since $|\chi|\subset A_0$, our above estimates for $\lam\,ds$ (in $A_0$) yield
\begin{gather*}
  \ell_h(\beta^h)\le\Lam h(a,b)\le\Lam \int_{\chi} \lam\,ds \le \frac{\Lam\ell(\chi)}{|a|\log(R/|b|)}
                 \le \frac{3\Lam}{\log(R/|b|)}\,\frac{|a-b|}{|a|}
  \intertext{and}
  \ell_h(\beta^h) \ge \frac{e^{-2\mu}\ell(\beta^h)}{|b|\bigl(\kk+2\mu+\log(R/|a|)\bigr)}
  \intertext{from which the upper bounds in both \eqref{E:D*:cp ell h} and \eqref{E:D*:cp lhh} follow.  With \eqref{E:D*:cp ell h} in hand, we also have}
  \ell_k(\beta^h)=\int_{\beta^h} \frac{ds}{\del} \le \frac{e^{2\mu}}{|a|} \,\ell(\beta^h) \le 7e^{1+4\mu} \Lam \frac{|a-b|}{|a|}
\end{gather*}
which gives the upper bound in \eqref{E:D*:cp lkh}.  \flag{The upper bound in \eqref{E:D*:cp ell h} requires the fact that $\kk+2\mu+\log(R/|a|)\le(2.167)\log(R/|b|)$:-)}

The diligent reader can employ \eqref{E:D*:cp lk} and \eqref{E:D*:cp lh} to corroborate the inequalities in \eqref{E:D*:cp ratios}.  It is perhaps worth mentioning that $\log(R/|b|)\ge2\mu+\log2$ and recalling that $\mu=7e^\Lam$.

\bigskip 

Now we turn to the case where $\beta^h$ and $\beta^k$ need not have the same endpoints, but $|b^h|=|b^k|>e|a^k|=e|a^h|$.  Again, $k=k_*$ in $\core_{\log2}(D_*)$ and so by \eqref{E:k* ests} for any $a,b\in\core_{\log2}(A)$ with $|b|>e|a|$,
\begin{gather*}
  1\le L:=\log\frac{|b|}{|a|} \le k(a,b) \le L+\pi \le (1+\pi)L\,.
  \intertext{In particular, taking $a=a^k$ and $b=b^k$, say, we obtain}
  L \le \ell_k(\beta^k) \le \Lam k(a^k,b^k) \le (1+\pi)\Lam L \le (1+\pi)\Lam k(a^h,b^h) \le (1+\pi)\Lam\ell(\beta^h)
\end{gather*}
which gives the lower bound in \eqref{E:D*:ell_k ratio}.  \flag{If I produce similar estimates for $h(a,b)$, we could do similar....}

Let $m:=\max\{n\in\mfN : e^n|a|<|b|\}$, so $m$ is the unique positive integer that satisfies $e^m|a|<|b|\le e^{m+1}|a|$.  Put $\sig:=\bigl(|b|/|a|\bigr)^{1/m}$.  Then
\begin{gather*}
  e<\sig\le e^{1+1/m}<e^2 \qquad\text{and for $1\le j\le m$,} \quad e^j<\sig^j<e^{j+1}\,.
  \intertext{For $1\le j\le m$, define}
  A_j:=\{\sig^{j-1}|a|\le|z|\le \sig^j|a|\} \quad\text{and}\quad B_j:=\band_{\mu}(A_j)=\{e^{-2\mu}\sig^{j-1}|a|\le|z|\le e^{2\mu}\sig^j|a|\}\,.
\end{gather*}
Since $A_j\subann B_j\subann B_0$, the comments at the beginning of this proof remain in force.

For each $1\le j<m$, pick a point $w_j\in|\beta|\cap\mfS^1(0;\sig^j|a|)$, set $w_0:=a$ and $w_{m}:=b$, and for $1\le j\le m$ let $\alf_j=\alf^g_j:=\beta^g_{w_{j-1}w_j}$.  Thus $\alf^g_j$ is a sub-quasi-geodesic of $\beta^g$, and while $\alf^g_j$ may leave the annulus $A_j$, thanks to the ABC property (\rf{L:ABC}(b)) we know that $\alf^g_j$ lies in the (big/enlarged) annulus $B_j$.  Also, since $k=k_*$, \eqref{E:k* ests} gives
\[
  \log\sig\le k(w_j,w_{j-1})\le\log\sig+\pi<6\,.
\]

Since $\beta=\Star_1^{m}\alf_j$, the lengths of each $\beta$ are just sums of the corresponding lengths of its subpaths $\alf_j$.  To estimate these, we proceed in the same manner as above (in the ``close points'' case).  First we give estimates for the metrics $\del^{-1}\,ds$ and $\lam\,ds$ valid in the subannuli $A_j$ and $B_j$.  Then we estimate the (Euclidean, quasihyperbolic, and hyperbolic) lengths of each $\alf_j$.

We begin with estimates for $\del$ and $\lam$ in $A_j$ and $B_j$, then show that the Euclidean length of $\alf_j$ is essentially $\del_j:=\sig^j|a|\eqx e^j|a|$, and then present estimates for $\ell_g(\alf^g_j)$ and $\ell_h(\alf^k_j)$ and $\ell_k(\alf^h_j)$.  Everywhere below $1\le j\le m$.

For $z\in A_j$, $\del_j/\sig=\del_{j-1}\le|z|=\del(z)\le\del_j:=\sig^j|a|$, so as above, \rf{L:lam-del ests in D*} yields
\begin{gather*}
  \frac1{\kk+\log\sig+\log(R/\del_j)}\,\frac1{\del_j}  \le \lam(z) \le \frac{\sig}{\log(R/\del_j)}\,\frac1{\del_j}\,.
  \intertext{Similarly, for $z\in B_j$, $e^{-2\mu}\del_j/\sig \le|z|=\del(z)\le e^{2\mu}\del_j$, so}
  \frac{e^{-2\mu}}{\kk+2\mu+\log\sig+\log(R/\del_j)}\,\frac1{\del_j}  \le \lam(z) \le \frac{\sig e^{2\mu}}{-2\mu+\log(R/\del_j)}\,\frac1{\del_j}\,.
\end{gather*}
Note that $\del_j\le|b|\le\half e^{-2\mu}R$, so $\log(R/\del_j)\ge2\mu+\log2$.

Now we show that $\ell(\alf_j)\eqx\del_j$.  More precisely:
\begin{subequations}\label{E:D*:ell alf estimates}
  \begin{align}
    (1-\sig^{-1})\del_j &\le \ell(\alf^k_j) \le 6 e^{2\mu} \Lam\,\del_j\,,            \label{E:D*:ell alfk} \\
    (1-\sig^{-1})\del_j &\le \ell(\alf^h_j) \le 5\sig e^{2\mu}\Lam\,\del_j\,.         \label{E:D*:ell alfh}
  \end{align}
\end{subequations}
Then we verify that $\ell_k(\alf_j)\eqx1$ and $\ell_h(\alf_j)\eqx\bigl(\log(R/\del_j)\bigr)^{-1}\eqx\bigl(\log(R/|a|)-j\bigr)^{-1}$.  More precisely:
\begin{subequations}\label{E:D*:ellk alf estimates}
  \begin{align}
    \log\sig &\le \ell_k(\alf^k_j) \le (\log\sig+\pi) \Lam\,,                          \label{E:D*:lkak} \\
    \log\sig &\le \ell_k(\alf^h_j) \le 5\sig^2e^{4\mu}\Lam\,,                          \label{E:D*:lkah}
  \end{align}
\end{subequations}
and
\begin{subequations}\label{E:D*:ellh alf estimates}
  \begin{align}
    \frac{(1-\sig^{-1})e^{-2\mu}}{\kk+2\mu+\log\sig+\log(R/\del_j)} &\le \ell_h(\alf^h_j) \le \frac{2\sig\Lam}{\log(R/\del_j)}\,, \label{E:D*:lhah} \\
    \frac{(1-\sig^{-1})e^{-2\mu}}{\kk+2\mu+\log\sig+\log(R/\del_j)}&\le \ell_h(\alf^k_j)\le\frac{6\sig e^{4\mu}\Lam}{-2\mu+\log(R/\del_j)}\,. \label{E:D*:lhak}
  \end{align}
\end{subequations}

The lower bounds in \eqref{E:D*:ell alf estimates} are trivial, and those in \eqref{E:D*:ellk alf estimates} follow easily from \eqref{E:k ge log}.  Our above estimates for $\lam$ in $B_j$ then produce the lower bounds in \eqref{E:D*:ellh alf estimates}.  Thus it remains to establish the various upper bounds.   To begin, we introduce the circular-arc-segment path $\chi:=\chi_{w_{j-1}w_j}=\kap_{w_{j-1}w}\star[w,w_j]$ from $w_{j-1}$ through $w:=w_j/\sig$ to $w_j$ as described in \eqref{E:chi}.
Note that as $\pi+\sig-1\le\sig$,
\[
  \ell(\chi) \le 2\,\del_j\,.
\]

Using estimates for $\del$ in $B_j$ in conjunction with \eqref{E:k* ests} we obtain
\[
  \frac{e^{-2\mu}}{\del_j}\,\ell(\alf^k_j) \le \int_{\alf^k_j} \frac{ds}{\del} = \ell_k(\alf^k_j) \le \Lam k(w_{j-1},w_j) \le \Lam(\log\sig+\pi) \le 6\Lam
\]
which gives the upper bounds in both \eqref{E:D*:ell alfk} and \eqref{E:D*:lkak}.  Similarly, we use our estimates for $\lam$ (in $A_j\supset|\chi|$ and in $B_j\supset|\alf^h_j|$) to obtain
\begin{align*}
    \frac{e^{-2\mu}\,\ell(\alf^h_j)}{\kk+2\mu+\log\sig+\log(R/\del_j)}\,\frac1{\del_j} &\le \int_{\alf^h_j} \lam\,ds = \ell_h(\alf^h_j)
    \le \Lam\,\ell_h(\chi) = \\
    &= \Lam\int_\chi \lam\ ds \le \frac{\sig\Lam\,\ell(\chi)}{\log(R/\del_j)}\,\frac1{\del_j} \le \frac{2\sig\Lam}{\log(R/\del_j)}\,.
\end{align*}
which gives the upper bound in \eqref{E:D*:lhah}.

We claim that $\kk+2\mu+\log\sig+\log(R/\del_j) \le 2.15 \log(R/\del_j)$, and thus the above also produces the upper bound in \eqref{E:D*:ell alfh}.  To check our claim, note that
\[
  \frac{\kk+2\mu+\log\sig+\log(R/\del_j)}{\log(R/\del_j)}\le1+\frac{\kk+2\mu+\log\sig}{2\mu+\log2}\le2+\frac{\kk+2-\log2}{2\mu+\log2}\,.
\]

It remains to verify the inequalities \eqref{E:D*:lkah} and \eqref{E:D*:lhak}.
Using estimates for $\del$ in $B_j$ together with \eqref{E:D*:ell alfh} we see that
\begin{gather*}
  \ell_k(\alf^h_j) = \int_{\alf^h_j}\frac{ds}{\del} \le \frac{\sig e^{2\mu}}{\del_j}\,\ell(\alf^h_j) \le 5\sig^2 e^{4\mu} \Lam
  \intertext{and then using estimates for $\lam$ in $B_j$ and \eqref{E:D*:ell alfk} we obtain}
  \ell_h(\alf^k_j) = \int_{\alf^k_j} \lam\,ds \le \frac{\sig e^{2\mu}}{-2\mu+\log(R/\del_j)}\,\frac{\ell(\alf^k_j)}{\del_j} \le
      \frac{6 \sig e^{4\mu}\Lam}{-2\mu+\log(R/\del_j)}\,.
\end{gather*}

The diligent reader can now employ the inequalities in \eqref{E:D*:ellk alf estimates} and \eqref{E:D*:ellh alf estimates} to establish \eqref{E:D*:ell_k ratio} and \eqref{E:D*:ell_h ratio}.
\end{proof}%

Here is the analog to \rf{P:D*:ell-bad-arcs} for the case of a degenerate annulus that is a punctured disk in the Riemann sphere.

\begin{prop} \label{P:Dc:ell-bad-arcs} 
For each $\Lam\ge1$, there are explicit constants that depend only on $\Lam$ with the following properties.  Let $\mu=\mu(\Lam)$ be the ABC parameter for hyperbolic $\Lam$-chordarc paths.  Assume $D_*:=\mfC\sm\D[0;R]\in\mcA_\Om$ for some $R>0$.  Let $\beta^h$ and $\beta^k$ be, respectively, any hyperbolic and quasihyperbolic $\Lam$-quasi-geodesics with concentric endpoints $\bd\beta^h=\{a^h,b^h\}$ and $\bd\beta^k=\{a^k,b^k\}$ in $\core_{2\mu+\log2}(D_*)$ that satisfy
\[
  |a^h|=|a^k| \ge |b^k|=|b^h|\,.
\]
For both $g=h$ and $g=k$, $\ell_g(\beta^h)\eqx\ell_g(\beta^k)$ with explicit constants that depend only on $\Lam$.

\smallskip\noindent
More precisely: assuming that $\beta^h$ and $\beta^k$ are hyperbolic and quasihyperbolic $\Lam$-chordarc paths, then when $|a^h|\le e|b^h|$ and $\bd\beta^h=\bd\beta^k$,
\addtocounter{equation}{-1}  
\begin{subequations}\label{E:D^c:ell-bad-arcs}
  \begin{gather}
    \frac1\Lam \le \frac{\ell_k(\beta^h)}{\ell_k(\beta^k)} \le 14 e^{2+4\mu}\Lam \quad\text{and}\quad
    \frac1\Lam \le \frac{\ell_h(\beta^k)}{\ell_h(\beta^h)} \le 24\mu\Lam e^{2+6\mu}                     \label{E:Dc:cp ratios}
    \intertext{and when $|a^h|>e|b^h|$ (and possibly $\bd\beta^h\ne\bd\beta^k$),}
    \bigl((1+\pi)\Lam\bigr)^{-1} \le \frac{\ell_k(\beta^h)}{\ell_k(\beta^k)} \le 6 e^{4+4\mu}\Lam      \label{E:Dc:ell_k ratio}
    \intertext{and}
    (51 e^{2\mu}\Lam)^{-1} \le \frac{\ell_h(\beta^k)}{\ell_h(\beta^h)} \le 491\mu e^{6\mu}\Lam\,.       \label{E:Dc:ell_h ratio}
  \end{gather}
\end{subequations}
\end{prop} 
\begin{proof}%
Since $0\in\Om^c$, the image of $\Om$ under the complex inversion  $z\mapsto1/z$ is a hyperbolic plane domain and now we are in the setting of \rf{P:D*:ell-bad-arcs}.  By conformal invariance, all hyperbolic lengths are unchanged, and according to \rf{F:kQI}, quasihyperbolic lengths only change by a factor of 2.  Moreover, since $\Om\subset\Co$, $k\ge k_*$ and we can still use \eqref{E:k* ests} to obtain good lower bounds for quasihyperbolic distance.
\end{proof}%

\subsection{Quasi-Geodesics in Annuli}  \label{s:QGs in A}%
Here we establish similar estimates as above in \rf{s:QGs in D*} but now the geodesics $\beta^h$ and $\beta^k$ are assumed to lie deep in a non-degenerate annulus.  Again we emphasize that, except for the ``close points case'', the endpoints need not be the same.

\begin{prop} \label{P:A:ell-bad-arcs} 
For each $\Lam\ge1$, there are explicit constants that depend only on $\Lam$ with the following properties.  Let $\mu=\mu(\Lam)$ be the ABC parameter for hyperbolic $\Lam$-chordarc paths.\footnote{See \rfs{R:ABC}.}
Suppose $A\in\mcA_\Om(10\mu)$, say with center $c(A)=0$.  
Let $\beta^h$ and $\beta^k$ be, respectively, any hyperbolic and quasihyperbolic $\Lam$-quasi-geodesics with concentric endpoints $\bd\beta^h=\{a^h,b^h\}$ and $\bd\beta^k=\{a^k,b^k\}$ in $\core_{5\mu}(A)$ that satisfy
\[
  |a^h|=|a^k| \le |b^k|=|b^h|\,.
\]
For both $g=h$ and $g=k$, $\ell_g(\beta^h)\eqx\ell_g(\beta^k)$ with explicit constants that depend only on $\Lam$.

\smallskip\noindent
More precisely: assuming that $\beta^h$ and $\beta^k$ are hyperbolic and quasihyperbolic $\Lam$-chordarc paths, then when $|b^h|\le e|a^h|$ and $\bd\beta^h=\bd\beta^k$,
\addtocounter{equation}{-1}  
\begin{subequations}\label{E:A:ell-bad-arcs}
  \begin{gather}
    \frac1\Lam \le \frac{\ell_k(\beta^h)}{\ell_k(\beta^k)} \le 768 e^{2+4\mu}\Lam \quad\text{and}\quad
        \frac1\Lam \le\frac{\ell_h(\beta^k)}{\ell_h(\beta^h)}\le768 e^{2+6\mu}\Lam=3\cdot2^8e^{2+6\mu}\Lam  \label{E:A:cp ratios}
    \intertext{and when $|b^h|>e|a^h|$ (and possibly $\bd\beta^h\ne\bd\beta^k$),}
    (9\Lam)^{-1} \le \frac{\ell_k(\beta^h)}{\ell_k(\beta^k)} \le 2^{8}e^{4+4\mu}\Lam                        \label{E:A:ell_k ratio}
    \intertext{and}
    (2^9 e^{2+2\mu}\Lam)^{-1} \le \frac{\ell_h(\beta^k)}{\ell_h(\beta^h)} \le11\cdot2^8 e^{2+6\mu}\Lam\,.   \label{E:A:ell_h ratio}
  \end{gather}
\end{subequations}
\end{prop} 
\noindent
Note that when $\bd\beta^h=\bd\beta^k$ the above lower bounds can be replaced with $1/\Lam$.
\begin{proof}%
Thanks to \rfs{P:D*:ell-bad-arcs} and \ref{P:Dc:ell-bad-arcs}, we may, and do, assume $A=\A(0;1,r)\in\mcA^2_\Om$, so $\core_{5\mu}(A)=\A(0;1,r-5\mu)$ and therefore $e^{5\mu-r}\le|a|\le|b|\le e^{r-5\mu}$.  Let
\[
  A_0:=\{z:|a|\le|z|\le|b|\} \quad\text{and}\quad B_0:=\band_{2\mu}(A_0)=\{z:e^{-2\mu}|a|\le|z|\le e^{2\mu}|b|\}\,.
\]
Evidently, $A_0\csubann B_0\csubann\overline\core_{3\mu}(A)\csubann\core_{\log16}(A)$.  According to \rf{L:ABC}(b) in conjunction with \rf{P:ABC}(b), both quasi-geodesics $\beta^h$ and $\beta^k$ lie in $B_0$, and thus we may appeal to \rf{L:del deep in A} (to estimate $\del$) and to \rf{P:BPEP} (to estimate $\bp$).  Also, according to \rf{R:bp ests}(b), $\bp>\half(3\mu)>\kk$ in $\core_{3\mu}(A)$; therefore we may use \eqref{E:BP est} to estimate $\lam$ in $B_0$.

In particular, from \rf{L:del deep in A} we know that in $\core_{\log2}(A)$, $\half\del_*\le\del\le\del_*$, and so $k_*\le k\le2k_*$ where $k_*$ is quasihyperbolic distance in $\Co$.  Thus \eqref{E:k* ests} already provides estimates for $k(a,b)$ when $a,b\in\core_{\log2}(A)$.  \flag{I believe that I can prove similar estimates for $h(a,b)$, but it takes some effort!}

\bigskip 

Suppose first that $a:=a^h=a^k$, $b:=b^h=b^k$, and $|b|\le e|a|$. (We call this the ``close points'' case.)  Since the endpoints coincide, the two lower bounds in \eqref{E:A:cp ratios} are trivial.  Thus we need only verify the asserted upper bounds.  From \eqref{E:k* le} we have
\[
  \frac{|a-b|}{|b|} \le k(a,b) \le 2k_*(a,b) \le 6\, \frac{|a-b|}{|a|}\,.
\]

According to \rf{L:del deep in A}, $\del(z)\eqx|a|$ for $z\in B_0$.  More precisely,
\begin{gather*}
  \forall\; z\in A_0\,,   \quad  \frac1{|b|} \le \frac1{\del(z)} \le \frac2{|a|}
  \intertext{and}
  \forall\; z\in B_0\,,   \quad  \frac{e^{-2\mu}}{|b|} \le \frac1{\del(z)} \le \frac{2e^{2\mu}}{|a|}\,.
\end{gather*}

Next we show that $\bp\eqx r-\log|a|$ in $B_0$.  So, let $z\in B_0$.  Assume $|z|=e^t$.  It is straightforward to check that
\begin{gather*}
  \bigl| t-\log|a| \bigr| \le 2\mu+1 =:M \quad\text{and}\quad  r\ge M+\bigl(|t|\vee\bigl|\log|a|\bigr|\bigr)\,.
  \intertext{Hence by \rf{P:BPEP} in conjunction with \rf{L:NumLma} (with $r,t$ and $s:=\log|a|$) we obtain}
  \frac14\,\bigl(r-\bigl|\log|a|\bigr|\bigr) \le \bp(z) \le 4 \bigl(r-\bigl|\log|a|\bigr|\bigr)\,.
\intertext{Using \eqref{E:BP est} with our above estimates (for $\del$ and $\bp$) we find that}
  \forall\; z\in A_0\,,   \quad  \frac{8^{-1}}{|b|}\,\frac1{r-\bigl|\log|a|\bigr|} \le \lam(z) \le \frac{16}{|a|}\,\frac1{r-\bigl|\log|a|\bigr|}
  \intertext{and}
  \forall\; z\in B_0\,,   \quad  \frac{(8e^{2\mu})^{-1}}{|b|}\,\frac1{r-\bigl|\log|a|\bigr|} \le \lam(z) \le \frac{16e^{2\mu}}{|a|}\,\frac1{r-\bigl|\log|a|\bigr|}\,.
\end{gather*}

\smallskip

Now we verify that $\ell_k(\beta^k)\eqx\ell_k(\beta^h)$ and $\ell_h(\beta^h)\eqx\ell_h(\beta^k)$.  More precisely, we establish:
\begin{subequations}\label{E:A:cp ell}
  \begin{align}
    |a-b| &\le \ell(\beta^k) \le 6e^{2\mu}\Lam\,\frac{|b|}{|a|}\,|a-b| \,,                    \label{E:A:cp ell k} \\
    |a-b| &\le \ell(\beta^h) \le 384e^{2\mu}\Lam\,\frac{|b|}{|a|}\,|a-b| \,;                  \label{E:A:cp ell h}
  \end{align}
\end{subequations}
and also
\begin{subequations}\label{E:A:cp lk}
  \begin{align}
    \frac{|a-b|}{|b|} &\le \ell_k(\beta^k) \le 6\Lam \frac{|a-b|}{|a|}   \,,                            \label{E:A:cp lkk} \\
    \frac{|a-b|}{|b|} &\le \ell_k(\beta^h) \le 768e^{4\mu}\Lam\,\frac{|b|}{|a|}\,\frac{|a-b|}{|a|}\,,    \label{E:A:cp lkh}
  \end{align}
\end{subequations}
and
\begin{subequations}\label{E:A:cp lh}
  \begin{align}
    \frac{(8e^{2\mu})^{-1}}{r-\bigl|\log|b|\bigr|}\,\frac{|a-b|}{|b|} &\le \ell_h(\beta^h) \le \frac{48\Lam}{r-\bigl|\log|a|\bigr|}\,\frac{|a-b|}{|a|}\,,  \label{E:A:cp lhh} \\
    \frac{(8e^{2\mu})^{-1}}{r-\bigl|\log|b|\bigr|}\,\frac{|a-b|}{|b|} &\le \ell_h(\beta^k) \le \frac{96e^{4\mu}\Lam}{r-\bigl|\log|a|\bigr|}\,\frac{|b|}{|a|}\,\frac{|a-b|}{|a|}\,.             \label{E:A:cp lhk}
  \end{align}
\end{subequations}

The six lower bounds above, which hold for any path joining $a$ and $b$, are straightforward (as explained in \rf{R:ell alf ests}(a)).  Thus we need only verify the upper bounds.  Noting that $|\beta^k|\subset B_0$ and using our above estimates for $\del^{-1}ds$ (in $B_0$) we obtain
\begin{gather*}
  \frac{\ell(\beta^k)}{e^{2\mu}|b|} \le \int_{\beta^k}\frac{ds}{\del}=
  \ell_k(\beta^k)\le\Lam k(a,b) \le2\Lam k_*(a,b) \le 6\Lam\frac{|a-b|}{|a|}
\intertext{and so the upper bounds in both \eqref{E:A:cp ell k} and \eqref{E:A:cp lkk} hold.
With \eqref{E:A:cp ell k} in hand, we use our above estimates for $\lam\,ds$ (in $B_0$) to obtain}
  \ell_h(\beta^k)=\int_{\beta^k} \lam\,ds \le \frac{16e^{2\mu}}{|a|}\,\frac{\ell(\beta^k)}{r-\bigl|\log|a|\bigr|}
                 \le \frac{96e^{4\mu}\Lam}{r-\bigl|\log|a|\bigr|}\,\frac{|b|}{|a|}\,\frac{|a-b|}{|a|}
\end{gather*}
which gives \eqref{E:A:cp lhk}.

\smallskip

Now we turn to the estimates for $\beta^h$.  Let $\chi:=\chi_{ab}=\kap_{ac}\star[c,b]$ be the circular-arc-segment path from $a$ to $c:=\bigl(|a|/|b|\bigr)b$ to $b$ as described  in \eqref{E:chi}.  By \rf{L:ell chi}
\begin{gather*}
  \ell(\chi) \le 3|a-b|\,.
  \intertext{Noting that $|\chi|\subset A_0$ and using our above estimates for $\lam\,ds$ (in $A_0$) we obtain}
  \ell_h(\beta^h)\le\Lam h(a,b)\le\Lam \int_{\chi} \lam\,ds \le \frac{16}{|a|}\,\frac{\Lam\,\ell(\chi)}{r-\bigl|\log|a|\bigr|}
                 \le \frac{48\Lam}{r-\bigl|\log|a|\bigr|}\,\frac{|a-b|}{|a|}
\intertext{from which the upper bounds in both \eqref{E:A:cp ell h} and \eqref{E:A:cp lhh} follow.  With \eqref{E:A:cp ell h} in hand, we also have}
  \ell_k(\beta^h)=\int_{\beta^h} \frac{ds}{\del} \le \frac{2e^{2\mu}}{|a|} \,\ell(\beta^h) \le 768e^{4\mu} \Lam\,\frac{|b|}{|a|}\,\frac{|a-b|}{|a|}
\end{gather*}
which gives \eqref{E:A:cp lkh}.

The diligent reader can employ \eqref{E:A:cp lk} and \eqref{E:A:cp lh} to corroborate the inequalities in \eqref{E:A:cp ratios}.

\bigskip 

Now we examine the case where $\beta^h$ and $\beta^k$ need not have the same endpoints, but $|b^h|=|b^k|>e|a^k|=e|a^h|$.  Note that by \rf{L:del deep in A} and \eqref{E:k* ests}, for any $a,b\in\core_{\log2}(A)$ with $|b|>e|a|$,
\begin{gather*}
  1\le L:= \log\frac{|b|}{|a|} \le k(a,b) \le 2\pi+2L \le 2(1+\pi)L \le 9L\,.
  \intertext{In particular,}
  L \le \ell_k(\beta^k) \le \Lam k(a^k,b^k) \le 9\Lam L \le 9\Lam\ell_k(\beta^h)
\end{gather*}
which gives the lower bound in \eqref{E:A:ell_k ratio}.

It is convenient to consider the cases $|a^h|\ge1$ and $|b^h|\le1$ separately.  
So, assume $1\ge|b^h|>e|a^h|$.

Let $m:=\max\{n\in\mfN : e^n|a|<|b|\}$, so $m$ is the unique positive integer with $e^m|a|<|b|\le e^{m+1}|a|$ and $m+\log|a|<0$.  Put $\sig:=\bigl(|b|/|a|\bigr)^{1/m}$.  Then
\begin{gather*}
  e<\sig\le e^{1+1/m}<e^2 \qquad\text{and for $1\le j\le m$,} \quad e^j<\sig^j<e^{j+1}\,.
  \intertext{For $1\le j\le m$, define}
  A_j:=\{\sig^{j-1}|a|\le|z|\le \sig^j|a|\} \quad\text{and}\quad B_j:=\band_{2\mu}(A_j)=\{e^{-2\mu}\sig^{j-1}|a|\le|z|\le e^{2\mu}\sig^j|a|\}\,.
\end{gather*}
Since $A_j\csubann B_j\csubann B_0$, the comments at the beginning of this proof remain in force.

For each $1\le j<m$, pick a point $w_j\in|\beta|\cap\mfS^1(0;\sig^j|a|)$, set $w_0:=a$ and $w_{m}:=b$, and for $1\le j\le m$ let $\alf_j=\alf^g_j:=\beta^g[w_{j-1},w_j]$.  Thus $\alf^g_j$ is a sub-quasi-geodesic of $\beta^g$, and while $\alf^g_j$ may leave the annulus $A_j$, thanks to the ABC property (\rf{L:ABC}(b)) we know that $\alf^g_j$ lies in the big annulus $B_j$.  Also, by \rf{L:del deep in A} and \eqref{E:k* ests},
\[
  \log\sig\le k(w_j,w_{j-1})\le2(\pi+\log\sig)\le2(\pi+2)\le11\,.
\]

Since $\beta=\Star_1^{m}\alf_j$, the lengths of each $\beta$ are just sums of the corresponding lengths of its subpaths $\alf_j$.  To estimate these, we proceed in the same manner as above (in the ``close points'' case).  First we give estimates for the metrics $\del^{-1}\,ds$ and $\lam\,ds$ valid in the subannuli $A_j$ and $B_j$.  Then we estimate the Euclidean, quasihyperbolic, and hyperbolic lengths of each $\alf_j$ as described in \rfs{R:ell alf ests}.

In particular, note that as $e^m|a|\le|b|\le1$,
\[
   \forall\;1\le j\le m\,, \quad \abs{\log|a|+j}=-\log|a|-j\,.
\]

We begin with estimates for $\del, \bp, \lam$ in $A_j$ and $B_j$, then show that the Euclidean length of $\alf_j$ is essentially $\sig^j|a|\eqx e^j|a|$.  Then we present estimates for $\ell_g(\alf^g_j)$ and $\ell_h(\alf^k_j)$ and $\ell_k(\alf^h_j)$.  Everywhere below $1\le j\le m$.

In $B_j$, $\del\eqx\del_j:=\sig^j|a|$, $\bp\eqx\ups_j:=r+\log|a|+j$, and $\lam\eqx\lam_j:=(\ups_j\del_j)^{-1}$.  More precisely:
for all $z\in A_j$,
\begin{subequations}\label{E:A:Aj estimates}
  \begin{align}
    \frac1{2\sig}\,\del_j    &\le \del(z) \le \del_j\,,                       \label{E:A:Aj del est} \\
    \frac14\,\ups_j          &\le \bp(z)  \le 4\, \ups_j\,,                   \label{E:A:Aj bp est} \\
    \frac{8^{-1}}{\ups_j\del_j} &\le\lam(z)\le\frac{16\sig}{\ups_j\del_j}\,.  \label{E:A:Aj lam est}
  \end{align}
\end{subequations}
and for all $z\in B_j$,
\begin{subequations}\label{E:A:Bj estimates}
  \begin{align}
    \frac{e^{-2\mu}}{2\sig}\,\del_j       &\le \del(z) \le e^{2\mu}\del_j\,,                        \label{E:A:Bj del est} \\
    \frac14\, \ups_j                      & \le \bp(z) \le 4\, \ups_j\,,                            \label{E:A:Bj bp est} \\
    \frac{(8e^{2\mu})^{-1}}{\ups_j\del_j} &\le \lam(z) \le \frac{16\sig e^{2\mu}}{\ups_j\del_j}\,.  \label{E:A:Bj lam est}
  \end{align}
\end{subequations}
Evidently, \eqref{E:A:Aj lam est} and \eqref{E:A:Bj lam est} are consequences of \eqref{E:BP est} in conjunction with the remaining inequalities, and \eqref{E:A:Aj del est} and \eqref{E:A:Bj del est} follow from $\half\del_*\le\del\le\del_*$.  It remains to elucidate how \rf{P:BPEP} gives \eqref{E:A:Bj bp est} (which implies \eqref{E:A:Aj bp est}).

Suppose $z\in B_j$ and write $|z|=e^t$.  Then
\begin{gather*}
  e^{j-1-2\mu}|a| \le e^{-2\mu}\sig^{j-1}|a| \le e^t \le e^{2\mu}\sig^j|a| \le e^{j+1+2\mu}|a|
  \intertext{so}
  (\log|a|+j)-2\mu-1 \le t \le (\log|a|+j)+2\mu+1
  \intertext{and therefore (recalling that $\log|a|+j<0$)}
  |t|\le (2\mu+1)-(\log|a|+j)\,.
\intertext{We use \rf{L:NumLma} (and the \BPEP) with $s:=\log|a|+j<0$ and $M:=2\mu+1$.  From an inequality above,}
  |t-s|=|t-(\log|a|+j)|\le2\mu+1=M\,.
  \intertext{Next, $|a|\ge e^{5\mu-r}$ so $s\ge\log|a|\ge 5\mu-r$ whence}
  r\ge 5\mu-s = 5\mu+|s| \ge M+|s|\,.
  \intertext{Also, from inequalities above, $|t|\le(2\mu+1)-s=|s|+(2\mu+1)\le(r-5\mu)+(2\mu+1)$, so}
  r\ge(3\mu+1)+|t|\ge M+|t|\,.
  \intertext{Therefore by \rf{L:NumLma} and \rf{P:BPEP}}
  \bp(z)\le2(r-|t|)\le4(r-|s|)=4\ups_j
  \intertext{and}
  \bp(z)\ge\half(r-|t|)\ge\qtr(r-|s|)=\qtr\,\ups_j\,.
\end{gather*}

Having established the inequalities in \eqref{E:A:Aj estimates} and \eqref{E:A:Bj estimates}, it follows that $\ell(\alf_j)\eqx\del_j= \sig^j|a|$.  More precisely:
\begin{subequations}\label{E:A:ell alf estimates}
  \begin{align}
    \frac12\, \del_j &\le \ell(\alf^k_j) \le 11 e^{2\mu} \Lam \del_j\,,            \label{E:A:ell alfk} \\
    \frac12\, \del_j &\le \ell(\alf^h_j) \le 256\sig e^{2\mu} \Lam\del_j\,.        \label{E:A:ell alfh}
  \end{align}
\end{subequations}
And then we get $\ell_k(\alf_j)\eqx1$ and $\ell_h(\alf_j)\eqx\ups_j^{-1}$.  More precisely:
\begin{subequations}\label{E:A:ellk alf estimates}
  \begin{align}
    \log\sig &\le \ell_k(\alf^k_j) \le 11 \Lam\,,                       \label{E:A:lkak} \\
    \log\sig &\le \ell_k(\alf^h_j) \le 512\sig^2e^{4\mu}\Lam\,,         \label{E:A:lkah}
  \end{align}
\end{subequations}
and
\begin{subequations}\label{E:A:ellh alf estimates}
  \begin{align}
    \frac1{16e^{2\mu}}\,\frac1{\ups_j} &\le \ell_h(\alf^h_j) \le 32\sig\Lam\,\frac{1}{\ups_j}\,,           \label{E:A:lhah} \\
    \frac1{16e^{2\mu}}\,\frac1{\ups_j} &\le \ell_h(\alf^k_j) \le 176\sig e^{4\mu}\Lam\,\frac1{\ups_j}\,.   \label{E:A:lhak}
  \end{align}
\end{subequations}

The six lower bounds above, which hold for any path joining $w_j$ and $w_{j-1}$, are straightforward (as explained in \rf{R:ell alf ests}(a)).  Thus we need only verify the upper bounds.  To begin, we introduce the circular-arc-segment path $\chi:=\chi_{w_{j-1}w_j}=\kap_{w_{j-1}w}\star[w,w_j]$ from $w_{j-1}$ through $w:=w_j/\sig$ to $w_j$ as described as described in \eqref{E:chi}.  Note that as $\pi+\sig-1\le\sig$,
\begin{gather*}
  \ell(\chi) \le 2\del_j\,.
  \intertext{Using \eqref{E:A:Bj del est} for $\alf^k_j$ we see that}
  \frac{e^{-2\mu}}{\del_j}\,\ell(\alf^k_j) \le \int_{\alf^k_j} \frac{ds}{\del} = \ell_k(\alf^k_j) \le \Lam k(w_j,w_{j-1}) \le 11 \Lam\,,
  \intertext{and then using \eqref{E:A:Aj lam est} (for $\chi$) and \eqref{E:A:Bj lam est} (for $\alf^h_j$) in a similar manner we obtain}
  \frac1{8e^{2\mu}}\,\frac{\ell(\alf^h_j)}{\ups_j\del_j} \le  \int_{\alf^h_j} \lam\,ds = \ell_h(\alf^h_j)
    \le \Lam\ell_h(\chi) = \Lam\int_\chi \lam\ ds \le \frac{16\sig\Lam}{\ups_j\del_j}\,\ell(\chi) \le \frac{32\sig\Lam}{\ups_j}\,.
\end{gather*}
Evidently, the upper bounds in each of \eqref{E:A:ell alf estimates}, \eqref{E:A:lkak}, \eqref{E:A:lhah} are contained in the two strings of inequalities displayed immediately above.

It remains to verify the inequalities \eqref{E:A:lkah} and \eqref{E:A:lhak}.  We use \eqref{E:A:Bj del est} and \eqref{E:A:ell alfh} to see that
\begin{gather*}
  \ell_k(\alf^h_j) = \int_{\alf^h_j} \frac{ds}{\del} \le \frac{2\sig e^{2\mu}}{\del_j}\,\ell(\alf^h_j) \le 512 \sig^2 e^{4\mu} \Lam
  \intertext{and then use \eqref{E:A:Bj lam est} and \eqref{E:A:ell alfk} to obtain}
  \ell_h(\alf^k_j) = \int_{\alf^k_j} \lam\,ds \le \frac{16\sig e^{2\mu}}{\ups_j\del_j}\,\ell(\alf^k_j) \le 176\sig e^{4\mu} \Lam\, \frac1{\ups_j}\,.
\end{gather*}

As the length of each $\beta$ is obtained by adding up the lengths of its subarcs $\alf_j$, the diligent reader can now employ \eqref{E:A:ellk alf estimates} and \eqref{E:A:ellh alf estimates} to establish \eqref{E:A:ell_k ratio} and \eqref{E:A:ell_h ratio}.  \flag{NB for $e\le x\le e^2, x^2/\ln x\le e^4/2$.}

\bigskip

The same argument handles the subcase where $|a|\ge1$; here $A_i=\{\sig^{-i}|b|\le|z|\le\sig^{1-i}|b|\}$, $\del_i:=\sig^{1-i}|b|$, $\ups_i:=r-\log|b|+i$ for $1\le i\le m$, and we obtain precisely the same estimates.  Finally, to handle the general case when $|a|<1<|b|$, we simply pick any points where the quasi-geodesics cross $\mfS^1$, apply the appropriate subcases to the various sub-quasi-geodesics, and then add up.
\end{proof}%
\subsection{Proof of Theorem~\ref{TT:bad??}}  \label{s:bad??} %
This follows by combining the three \rfs{P:D*:ell-bad-arcs}, \ref{P:Dc:ell-bad-arcs}, \ref{P:A:ell-bad-arcs}.  We take $Q:=5\mu=35e^\Lam$.  The constants $K$ and $H$ are chosen to be, respectively,  the largest of the corresponding constants in the inequalities \eqref{E:D*:ell_k ratio}, \eqref{E:Dc:ell_k ratio}, \eqref{E:A:ell_k ratio} and \eqref{E:D*:ell_h ratio}, \eqref{E:Dc:ell_h ratio}, \eqref{E:A:ell_h ratio}.  Note that the ``close points'' cases do not occur in the setting of \rf{TT:bad??}.

%
\section{Proof of Theorem~\ref{TT:geos}}  \label{S:geos pf} 

Let $a,b\in\Om$ and fix two quasi-geodesics $\Gam^h$ and $\Gam^k$ both with endpoints $a,b$.  We claim that for both $g=h$ and $g=k$, $\ell_g(\Gam^h)\eqx\ell_g(\Gam^k)$ with comparison constants that depend only on the quasi-geodesic constants.  Our argument involves several cases each having subcases and subsubcases.  There are two Easy Cases where we assume that $\bp$ is bounded (either above or below) on one of the quasi-geodesics, and then three Main Cases where we consider the possible values of $\bp$ at each of the endpoints $a,b$.

We assume $\Gam^h$ and $\Gam^k$ are, respectively, a hyperbolic and a quasihyperbolic $\Lam$-chordarc path for some $\Lam\ge1$.  Put $\mu=\mu(\Lam):=\mu_h\vee\mu_k$ where $\mu_h, \mu_k$ are the ABC parameters for hyperbolic, quasihyperbolic $\Lam$-chordarc paths; so, $\mu_k=\pi\Lam$ and $\mu_h=7e^\Lam$ as explained in \rfs{R:ABC}.  Everywhere below we work with the constants $\SSS, \MM, \LL, \XL$ which are defined by
\[
  \SSS:=10\mu\,,\quad \MM:=10\SSS\,,\quad \LL:=10\MM\,,\quad \XL:=10\LL\,.
\]

Suppose that for some constant $C>\mu_k$, $\bp\le C$ along $\Gam^h$.  Then by \rf{C:large bp}, $\bp\le4C$ along $\Gam^k$.  According to (\BPt) we can now assert that $\lam\,ds$ and $\del^{-1}ds$ are \BL\ equivalent along both quasi-geodesics and our claim holds (in fact, with $K:=2(\kk+C)\Lam$ and $H:=(\kk+C)(\kk+4C)\Lam^2$).

\smallskip

Suppose that for some constant $C>10\mu$, $\bp\ge C$ along $\Gam^h$.  According to \rf{L:bp ge tau}, $b\in|\Gam^h|\subset\overline\core_{C/2}\A(a)\subann\core_{5\mu}\A(a)$, and then our claim follows from 
one of \rfs{P:D*:ell-bad-arcs} or \ref{P:Dc:ell-bad-arcs} or \ref{P:A:ell-bad-arcs} (with the corresponding constants $H, K$).

\medskip

Having dispensed with the Easy Cases, we now assume that $\bp$ is both smaller than $\SSS$ at some points of $|\Gam^h|$ and larger than $\LL$ at some points of $|\Gam^h|$.  
We examine the following \textbf{Main Cases}:
\begin{center}
  \begin{minipage}{3.5in}
\begin{enumerate}
  \item[(I)]  $\bp\le\SSS$ at both endpoints
  \item[(II)] $\bp\le\MM$ at both endpoints
  \begin{enumerate}
    \item[(a)]  $\bp\le \XL$ along $\Gam^h$
    \item[(b)]  $\bp\ge \XL$ at some point of $\Gam^h$
  \end{enumerate}
  \item[(III)] $\bp\ge\MM$ at one (or both) endpoint(s)
  \begin{enumerate}
    \item[(a)]  $\bp(a)\ge\MM\ge\bp(b)$
    \item[(b)]  $\bp(a)\ge\bp(b)\ge\MM$
  \end{enumerate}
\end{enumerate}
  \end{minipage}
\end{center}
Main Case I is the heart of our argument and given directly below in \rf{ss:both leS}.  Once this is established, Main Case II follows easily: If $\bp\le \XL$ along $\Gam^h$, we appeal to the first Easy Case with $C=\XL$.  If $\bp\ge \XL$ at some point of $\Gam^h$, we redo Main Case I replacing the constants $\SSS,\MM,\LL$ with $\MM,\LL,\XL$.  Main Case III also follows from earlier results, but it requires some additional effort as described below in \rf{ss:one geM}.

\medskip

First, we describe the key ideas in our argument for Main Case I.  We write each of the quasi-geodesics $\Gam^g$ as a concatenation of ``good'' subarcs $\gam^g=\gam^g_i$ and ``bad'' subarcs $\beta^g=\beta^g_i$.  On the ``good'' subarcs $\bp$ is not large, so the hyperbolic and quasihyperbolic metrics are \BL\ equivalent there, and thus for both $g=h$ and $g=k$, $\ell_h(\gam^g)\eqx\ell_k(\gam^g)$.

On the ``bad'' subarcs, $\bp$ is not small.  These subarcs lie in the middle core of LFS (large fat separating) annuli in $\Om$.  It is important to know that \emph{both} quasi-geodesics cross the LFS annuli; this is due to the fact (see \rf{L:cores separate}) that these middle cores separate $\{a,b\}$ which in turn is a consequence of the ABC property.

Geometrically, these LFS annuli are long tubes that are, roughly, long Euclidean cylinders in $(\Om,k)$ and long ``pinched'' cylinders in $(\Om,h)$.  Because of this geometry, we can show that the hyperbolic and quasihyperbolic lengths of the ``bad'' subarcs are comparable.  That is, for both $g=h$ and $g=k$ we have $\ell_g(\beta^h)\eqx\ell_g(\beta^k)$.  (The reader should pay attention to the difference between this comparability and that for the ``good'' subarcs!)  In fact, these estimates are a consequence of \rf{TT:bad??}.

The final step is to show that we also have a similar comparability for the ``good'' subarcs; i.e.,  for both $g=h$ and $g=k$ we have $\ell_g(\gam^h)\eqx\ell_g(\gam^k)$.

The crucial act is selecting the ``bad'' subarcs, which are chosen to be the subarcs that cross the middle cores of certain LFS annuli, and then the ``good'' subarcs are what remains.

\subsection{Main Case I} 	\label{ss:both leS} %
Here we assume that both endpoints satisfy $\bp(a)\le\SSS$ and $\bp(b)\le\SSS$.  We also assume that there is a point on $\Gam^h$ where $\bp$ is at least $\LL$.  Immediately below we describe an algorithm that produces an integer $n\ge1$ and points
\begin{gather*}
    a=:a^h_0<b^h_0<z_1<a^h_1<  \dots <  b^h_{i-1} < z_i < a^h_i < \dots <b^h_{n-1} < z_n < a^h_n < b^h_n:=b
  \intertext{along the quasi-geodesic $\Gam^h$, and similar points}
    a=:a^k_0<b^k_0<\dots <a^k_i<b^k_i<\dots<a^k_n<b^k_n:=b
  \intertext{along the quasi-geodesic $\Gam^k$.  Then the ``good'' and ``bad'' subarcs are precisely}
  \gam^g_i:=\Gam^g[a^g_i,b^g_i] \;\;(0\le i\le n) \quad\text{and}\quad \beta^g_i:=\Gam^g[b^g_{i-1},a^g_i] \;\;(1\le i\le n)\,.
\end{gather*}
Everywhere above and below we take both $g=h$ and $g=k$.

\medskip

The algorithm starts by putting $a^g_0:=a$ and $i:=1$.  Then, we repeat the following,

\smallskip\noindent
while there are points $z\in\Gam^h[a^h_{i-1},b]$ with $\bp(z)\ge\LL$:
\renewcommand{\labelitemi}{$\vcenter{\hbox{\tiny$\bullet$}}$}
\begin{itemize}
  \item  Let $z_i$ be the first point $z\in\Gam^h[a^h_{i-1},b]$ with $\bp(z)=\LL$.
  \item  Determine $\BP(z_i)$ and especially $A_i:=\A(z_i)$.
  \smallskip\item  Define $b^g_{i-1}$ and $a^g_i$ to be, respectively, the first and last points of $\Gam^g$ in $\overline\core_{2\MM}(A_i)$.
  \smallskip\item  Increment $i$ by 1.
\end{itemize}
Eventually, $\bp<\LL$ on $\Gam^h[a^h_{i-1},b]$, the process stops, and we put $n:=i-1$ and $b^g_n:=b$.

\medskip

Now we breakdown the additional details for our proof of Main Case I.

\subsubsection{Things are well defined} 	\label{ss:WD} %

Notice that as soon as $b^g_{i-1}$ and $a^g_i$ are defined, so are $\gam^g_{i-1}$ and $\beta^g_i$.
Observe that, by \rf{L:not punxd}, each $A_i$ is a non-degenerate annulus.  This is important because it allows us to utilize the \BPEP.  Also, by \rf{L:cores separate}, each of the cores $\core_{2\SSS}(A_i)$ separates $\{a,b\}$.  This means that both quasi-geodesics $\Gam^h$ and $\Gam^k$ cross each of these cores, so in particular, the points $b^g_{i-1}$ and $a^g_i$ (and hence also the subarcs $\gam^g_{i-1}$ and $\beta^g_i$) are indeed all well-defined.

\smallskip

Next we explain why $\core_\MM(A_i)\cap\core_\MM(A_j)=\emptyset$ for all $i\ne j$.  This is crucial information for our proof as it tells us that none of the ``bad'' subarcs of $\Gam^h$ or $\Gam^k$ overlap, and moreover it reveals that the points $a^g_i, b^g_i$ on $\Gam^g$ do lie in increasing order as asserted above.  In fact, for any $q\in(16\log2,\LL)$,  $\core_q(A_i)\cap\core_q(A_j)=\emptyset$ for all $i\ne j$, an observation that we require later.  This follows from  \rf{P:dsjt cores} once we check a few details.
\flag{A pix here of $A_i$ and its many cores---of `sizes' $2\MM,\MM,2\SSS,\half\SSS,\log16$---would be illuminating!}

Evidently, there is nothing to prove if $n=1$, so we may and do assume that $n>1$.  Fix $1\le i<n$.  Then, since $i<n$, $\Gam^h[a^h_i,b]$ has points where $\bp\ge\LL$.  Also, by our choice of the point $a^h_i$, this hyperbolic quasi-geodesic has empty intersection with $\core_{2\MM}(A_i)$.  In particular, for any point $z\in\Gam^h[a^h_i,b]\cap\core_{\log16}(A_i)$, $z\not\in\core_{2\MM}(A_i)$, and hence by \rf{R:ABC}(c) $\bp(z)\le4\MM<\LL$.

It follows that for all $1\le j\le n$ with $j\ne i$, $z_{j}\not\in \core_{\log16}(A_i)$.  Therefore $\Gam^h[a^h_i,z_{i+1}]$ must cross the ``outer fringe'' $\core_{\log16}(A_i)\sm\core_{\half\SSS}(A_i)$ of $A_i$ (this is a collar of $\core_{\half\SSS}(A_i)$ relative to $\core_{\log16}(A_i)$).  Let $p_i$ be any point in $\Gam^h[a^h_i,z_{i+1}]\cap[\core_{\log16}(A_i)\sm\core_{\half\SSS}(A_i)]$.  Then by \rf{R:bp ests}(c) we have $\bp(p_i)\le 2\,\half\SSS=\SSS$.  Thus we have points $p_i\in\Gam^h$ with $\bp(p_i)\le\SSS$ and
\[
  a <  b^h_{i-1} < z_i < a^h_i < p_i < b^h_i <z_{i+1}  < a^h_{i+1} < b\,,
\]
so by \rf{P:dsjt cores} (with $\tau=\LL, \sig=\SSS, q\in(16\log2,\LL)$), $\core_q(A_i)\cap\core_q(A_j)=\emptyset$ for all $i\ne j$.

\smallskip

Finally, each $\gam_i$ crosses one of the collars of $\core_{2\MM}(A_i)$ relative to $\core_\MM(A_i)$.  (To see this, note that the endpoints of $\gam_i$ satisfy $a_i\in\bd\core_{2\MM}(A_i)$ and $b_i\in\bd\core_{2\MM}(A_{i+1})$ while $\core_\MM(A_i)\cap\core_\MM(A_{i+1})=\emptyset$.)  According to \rf{C:h ge ln3/2}, it follows that $\ell_h(\gam_i)\ge\log(3/2)$.  As $\ell_h(\Gam^h)$ is finite, the above process does indeed stop.

\subsubsection{Initial estimates} 	\label{ss:IE} %

According to \rf{P:A:ell-bad-arcs}, we have $\ell_g(\beta^h)\eqx\ell_g(\beta^k)$; that is, for both $g=h$ and $g=k$ and all $1\le i\le n$,
\begin{equation}\label{E:ell-bad}
  C_1^{-1} \le \frac{\ell_g(\beta^h_i)}{\ell_g(\beta^k_i)} \le C_1:=C_0 e^{6\mu}\Lam
\end{equation}
where $C_0\eqx11\cdot2^8 e^4$.


Evidently, the $\gam^h_i$ are ``good'' subarcs, since by construction we have $\bp\le\LL$ along each $\gam^h_i$.  By using \rf{L:large bp}, it is not difficult to show that $\bp<4\LL$ along each $\gam^k_i$.  With a tad more effort, we see that $\bp<2\LL$ along each $\gam^k_i$. Thus the $\gam^k_i$ are are also ``good'' subarcs.

To corroborate this assertion, we first observe that for any $0\le i\le n$ and $1\le j\le n$, $|\gam^g_i|\cap\core_{2\MM}(A_j)=\emptyset$.  This is because $\Gam^g$ crosses $\core_{2\MM}(A_j)$ exactly one time, $\beta^g_j$ is the largest subarc of $\Gam^g$ that does this, $|\gam^g_i|\subset|\Gam^g|\sm|\beta^g_j|$, and $(|\Gam^g|\sm|\beta^g_j|)\cap\core_{2\MM}(A_j)=\emptyset$.

Next, note that in $\core_{\log16}(A_i)\sm\core_{2\MM}(A_i)$, $\bp\le4\MM$.  Therefore on the ``tails'' of each ``good'' subarc $\gam^g_i$ we have $\bp\le4\MM<\LL$, but once $\gam^g_i$ leaves $\core_{\log16}(A_i)$ or $\core_{\log16}(A_{i+1})$, we lose such control on $\bp$.  Nonetheless, this does reveal that 
$\bp<L$ on $|\gam^h_i|$.

Now suppose that for some $0\le i\le n$ there exists a point $w\in|\gam^k_i|$ with $\bp(w)\ge2\LL$.  Evidently, for all $1\le j\le n$, $w\not\in\overline\core_{\log16}(A_j)$ (see the discussion immediately above in paragraph 3 of \rf{ss:IE}).  Now as $\bp(a)\le\SSS$, $\bp(b)\le\SSS$, and $\bp(w)\ge2\LL$, an appeal to \rf{L:cores separate} reminds us that $\core_{2\SSS}\A(w)$ separates $\{a,b\}$, and hence $\mfS^1\bigl(\A(w)\bigr)$ separates $\{a,b\}$.  Therefore, there are points $\tilde{w}\in\Gam^k\cap\mfS^1\bigl(\A(w)\bigr)$ and $z\in\Gam^h\cap\mfS^1\bigl(\A(w)\bigr)$

Recalling \eqref{E:z in core bdy} we see that $w\in\bd\core_{\bp(w)}\A(w)\subset\overline\core_{2\LL}\A(w)$.  The ABC property for quasihyperbolic geodesics now gives $[w,\tilde{w}]_k\subset\core_{2\LL-2\mu}\A(w)$ (see \rf{L:ABC}(b)), and so along $[w,\tilde{w}]_k$ we have $\bp\ge\half( 2\LL-2\mu)=\LL-\mu$.

As $w\not\in\overline\core_{\log16}(A_j)$ (for all $1\le j\le n$), and on $\bd\core_{\log16}(A_j)$ we have $\bp\le8\log2$ (see \rf{R:bp ests}(b)), it now follows that $[w,\tilde{w}]_k$ never enters any $\overline\core_{\log16}(A_j)$.

Clearly, along $\mfS^1\bigl( \A(w) \bigr)$ we have $\bp\ge\LL$, so $\bp(z)\ge\LL$ and also $\mfS^1\bigl( \A(w) \bigr)$ never enters any $\overline\core_{\log16}(A_j)$ (for all $1\le j\le n$).  This last statement means there must exist $0\le j\le n$ with $z\in|\gam^h_j|$.  But this then implies that $\LL\le\bp(z)<\LL$.

This contradiction means that along each $\gam^k_i$ we do indeed have $\bp<2\LL$.  Thus from (\BPt) we obtain
\begin{gather}
  \frac1{3\LL} \le \lam\,\del \le 2 \quad\text{along each $\gam^g_i$}\,.  \notag
  \intertext{Therefore, for both $g=h$ and $g=k$, and for all $0\le i\le n$,}
  \half \le \frac{\ell_k(\gam^g_i)}{\ell_h(\gam^g_i)} \le 3\LL\,.  \label{E:ell-good1}
\end{gather}
The reader should pay careful attention to the differences between \eqref{E:ell-good1} and \eqref{E:ell-bad}

\subsubsection{Final estimates} 	\label{ss:FE} %
It remains to demonstrate that we also have $\ell_g(\gam^h)\eqx\ell_g(\gam^k)$ for each ``good'' subarc $\gam$.  We show that for both $g=h$ and $g=k$ and all $0\le i\le n$,
\begin{subequations}\label{E:ell-good2}
  \begin{align}
    \ell_k(\gam^k_i) &\le K_1 \ell_k(\gam^h_i)     \label{E:ell-g2a}
    \intertext{and}
    \ell_k(\gam^h_i) &\le K_2 \ell_k(\gam^k_i)     \label{E:ell-g2b}
  \end{align}
\end{subequations}
where $K_1:=(1+8\pi/\log(3/2))\Lam$ and $K_2:=6\LL K_1$.  Evidently, these inequalities combine to yield $\ell_k(\gam^h_i)\eqx\ell_k(\gam^k_i)$.  But, thanks to \eqref{E:ell-good1} we also get
\begin{gather*}
  \ell_h(\gam^k_i) \le 2 \ell_k(\gam^k_i) \le 2K_1 \ell_k(\gam^h_i) \le 6K_1\LL \, \ell_h(\gam^h_i)
  \intertext{and similarly}
  \ell_h(\gam^h_i) \le 2 \ell_k(\gam^h_i) \le 2K_2 \ell_k(\gam^k_i) \le 6K_2\LL \, \ell_h(\gam^k_i)
\end{gather*}
and the above in turn imply that $\ell_h(\gam^h_i)\eqx\ell_h(\gam^k_i)$.

Now we verify the inequalities in \eqref{E:ell-good2}.  Recall that $\gam^g_i:=\Gam^g[a^g_i,b^g_i]$.  From \eqref{E:k* ests} and \rf{R:k* ests} we know that
\begin{gather*}
  h(a^h_i,a^k_i)\le 2\,k(a^h_i,a^k_i) \le 4\pi \quad\text{and}\quad  h(b^h_i,b^k_i)\le2\,k(b^h_i,b^k_i) \le 4\pi
  \intertext{and from the last paragraph of \rf{ss:WD} we have $2\ell_k(\gam^g_i)\ge \ell_h(\gam^g_i) \ge \veps_o:=\log(3/2)$.  Thus}
  \ell_k(\gam^k_i) \le \Lam k(a^k_i,b^k_i) \le \Lam \bigl( k(a^h_i,a^k_i) + k(a^h_i,b^h_i) +  k(b^h_i,b^k_i) \bigr)
     \le \Lam \bigl(\ell_k(\gam^h_i) + 4\pi \bigr) \le K_1 \ell_k(\gam^h_i)
  \intertext{which establishes \eqref{E:ell-g2a}.  Also,}
  h(a^h_i,b^h_i)\le h(a^h_i,a^k_i)+h(a^k_i,b^k_i)+h(b^k_i,b^h_i)\le4\pi+\ell_h(\gam^k_i)+4\pi \le \Bigl( 1+\frac{8\pi}{\veps_o} \Bigr) \ell_h(\gam^k_i)
  \intertext{and therefore}
  \ell_k(\gam^h_i) \le 3\LL \ell_h(\gam^h_i) \le 3\LL\Lam h(a^h_i,h^k_i) \le 3\LL\Lam \Bigl( 1+\frac{8\pi}{\veps_o} \Bigr)\ell_h(\gam^k_i) \le K_2 \ell_k(\gam^k_i)
\end{gather*}
which establishes \eqref{E:ell-g2b}.

\subsection{Main Case III} 	\label{ss:one geM} %
Here we assume that $\bp\ge\MM$ at one (or both) endpoint(s), and there are two subcases.  We constantly make use of the fact that if the endpoints $a$ and $b$ both lie ``deep inside'' some LFS annulus to which 
one of \rfs{P:D*:ell-bad-arcs} or \ref{P:Dc:ell-bad-arcs} or \ref{P:A:ell-bad-arcs} applies, then we are done.

To begin, suppose  $\bp(a)\ge\MM\ge\bp(b)$; this is Main Case III(a).  Below (see \rf{ss:Chks}) we explain how to find points $c^g\in\Gam^g$ (for $g=h$ and $g=k$) such that
\begin{subequations}\label{E:MC3:ells qex}
  \begin{align}
    \ell_g(\Gam^h[a,c^h])&\eqx\ell_g(\Gam^k[a,c^k]) \quad\text{for $g=h,k$}        \label{E:MC3:ells eqx a}
    \intertext{and}
    \ell_g(\Gam^h[c^h,b])&\eqx\ell_g(\Gam^k[c^k,b]) \quad\text{for $g=h,k$}\,.        \label{E:MC3:ells eqx b}
  \end{align}
\end{subequations}
Our assertion that $\ell_g(\Gam^h)\eqx\ell_g(\Gam^k)$ then follows by adding these inequalities.

Since $\bp(a)\ge\MM$, we know that $a$ lies ``deep inside'' the LFS annulus $\A(a)$.  We select the points $c^g$ so that they also lie ``deep inside'' $\A(a)$, so then \rf{TT:bad??} applies to the points $a$ and $c^g$, and \eqref{E:MC3:ells eqx a} readily follows; establishing \eqref{E:MC3:ells eqx b} is the more difficult part.

We also pick the points $c^g$ so that $\bp(c^g)\le\MM$.  As $\bp(b)\le\MM$, Main Case II tells us that both
\begin{subequations}\label{E:MC3:**}
  \begin{align}
    \ell_g(\Gam^h[c^h,b])&\eqx\ell_g([c^h,b]_k) \quad\text{for $g=h,k$}        \label{E:MC3:2*ch}
    \intertext{and}
    \ell_g(\Gam^k[c^k,b])&\eqx\ell_g([c^k,b]_h) \quad\text{for $g=h,k$}\,.     \label{E:MC3:2*ck}
  \end{align}
\end{subequations}
Moreover, as shown below (see \rf{ss:***}), we easily have
\begin{subequations}\label{E:MC3:***}
  \begin{align}
    \ell_h(\Gam^h[c^h,b]) \eqx h(c^h,b) \eqx h(c^k,b) = \ell_h([c^k,b]_h)     \label{E:MC3:3*h}
    \intertext{and}
    \ell_k(\Gam^k[c^k,b]) \eqx k(c^k,b) \eqx k(c^h,b) = \ell_k([c^h,b]_k)\,.  \label{E:MC3:3*k}
  \end{align}
\end{subequations}

Using \eqref{E:MC3:3*h}, and then \eqref{E:MC3:2*ck} with $g=h$, we obtain
\begin{gather*}
  \ell_h(\Gam^h[c^h,b])\eqx\ell_h([c^k,b]_h) \eqx\ell_h(\Gam^k[c^k,b])  \quad\text{which is \eqref{E:MC3:ells eqx b} with $g=h$}
  \intertext{and using \eqref{E:MC3:3*k}, and then \eqref{E:MC3:2*ch} with $g=k$, we obtain}
  \ell_k(\Gam^k[c^k,b])\eqx\ell_k([c^h,b]_k) \eqx\ell_k(\Gam^h[c^h,b])  \quad\text{which is \eqref{E:MC3:ells eqx b} with $g=k$}.
\end{gather*}

\smallskip

Thus we must elucidate how to pick points $c^g\in\Gam^g$ that are ``deep inside'' $\A(a)$ (so that \rf{TT:bad??} applies to $a$ and $c^g$) with $\bp(c^g)\le\MM$ (so the inequalities \eqref{E:MC3:**} hold) and also so that the inequalities \eqref{E:MC3:***} hold.

\subsubsection{Picking $c^h$ \& $c^k$} 	\label{ss:Chks} %
We assume $\A(a)$ is centered at the origin, and then define
\[
  \B_a:=\core_{5\mu}\A(a)=\begin{cases}
    \D_*(0;e^{-5\mu}R)            & \text{when $\A(a)=\D_*(0;R)$}\,      \\
    \mfC\sm\D[0;e^{5\mu}R]        & \text{when $\A(a)=\mfC\sm\D[0;R]$}\,  \\
    \core_{5\mu}\A(a)             & \text{when $\A(a)$ is a non-degenerate annulus}\,.
  \end{cases}
\]
Evidently, $a\in \B_a$ and---since $\B_a$ is a LFS annulus---we may assume that $b\not\in \B_a$.\footnote{It is certainly possible that $b$ lies in $\A(a)$ or even in $\BP(a)$.}  Therefore, each quasi-geodesic $\Gam^g$ leaves the annulus $\B_a$, and so in each of the three cases there is a unique collar $\mfC_a$ of $\core_{5\mu}(\B_a)=\core_\SSS\A(a)$ relative to $\B_a$ that each $\Gam^g$ crosses.  (In the first two cases, when $\A(a)$ is a punctured disk, $\mfC_a$ is just $\B_a\sm\core_\SSS\A(a)$; when $\A(a)$ is a non-degenerate annulus, $\mfC_a$ is a component of $\B_a\sm\core_\SSS\A(a)$.)

We define $c^h$ and $c^k$ to be the first points of $\Gam^h$ and $\Gam^k$, respectively, that lie in $\overline\mfC_a$.  Thus $c^g\in\bd\mfC_a\cap\B_a$; in particular, $c^g\in\bd\core_\SSS\A(a)\subset\core_{\log16}\A(a)\sm\core_\SSS\A(a)$.  It now follows from \rf{R:ABC}(c) that $\bp(c^g)\le2\SSS<\MM$.  Moreover, from \eqref{E:z in core bdy} we have $a\in\overline\core_\MM\A(a)$ and thus we see that there is an annulus of modulus $\MM-\SSS=9\SSS$ that separates $\{a,c^g\}$.  Therefore, \rf{TT:bad??} certainly applies to the two pairs of points $a, c^h$ and $a, c^k$.

\subsubsection{Establishing \eqref{E:MC3:***}} 	\label{ss:***} %
Since $c^h$ and $c^k$ lie on the same component of $\bd\mfC_a$, they  are concentric \wrt to center of $\A(a)$ and therefore (see \rf{R:k* ests} and \eqref{E:k* ests})
\begin{gather*}
  k(c^h,c^k)\le2\pi \quad\text{and so also}\quad h(c^h,c^k)\le4\pi\,.
  \intertext{Since $c^g$ and $b$ are separated by $\mfC_a$,}
  k(c^g,b)\ge\md(\mfC_a)=5\mu\,.
  \intertext{As $\mfC_a$ is a collar of $\core_{5\mu}(\B_a)=\core_{10\mu}\A(a)$ relative to $\B_a=\core_{5\mu}\A(a)$, \rf{C:h ge ln3/2} tells us that}
  h(c^g,b)\ge\veps_o:=\log(3/2)\,.
\end{gather*}

Evidently,
\begin{gather*}
  h(c^h,b) \le h(c^h,c^k) + h(c^k,b) \le 4\pi + h(c^k,b) \le (1+4\pi/\veps_o) h(c^k,b)
  \intertext{and similarly, $h(c^k,b)\le(1+4\pi/\veps_o)h(c^h,b)$.  Thus}
  \ell_h(\Gam^h[c^h,b]) \le \Lam\,h(c^h,b) \le (1+4\pi/\veps_o)\Lam\, h(c^k,b)
  \intertext{and}
  h(c^k,b) \le(1+4\pi/\veps_o)h(c^h,b) \le (1+4\pi/\veps_o)\ell(\Gam^h[c^h,b])
\end{gather*}
which gives \eqref{E:MC3:3*h}.

Similarly,
\begin{gather*}
  k(c^h,b) \le k(c^h,c^k) + k(c^k,b) \le 2\pi + h(c^k,b) \le (1+2\pi/5\mu) k(c^k,b)
  \intertext{and $k(c^k,b)\le(1+2\pi/5\mu)k(c^h,b)$.  Thus}
  \ell_k(\Gam^k[c^k,b]) \le \Lam\,k(c^k,b) \le (1+2\pi/5\mu)\Lam\, k(c^h,b)
  \intertext{and}
  k(c^h,b) \le(1+2\pi/5\mu)k(c^k,b) \le (1+2\pi/5\mu)\ell(\Gam^k[c^k,b])
\end{gather*}
which gives \eqref{E:MC3:3*k}.

\subsubsection{Repeating the argument} 	\label{ss:MC3b} %
Finally, suppose $\bp(a)\ge\bp(b)\ge\MM$; this is Main Case III(b).  Roughly speaking, we now repeat the argument for Main Case III(a) using the same construction at each end of both quasi-geodesics.  As above, we get LFS annuli
\[
  \B_a:=\core_{5\mu}\A(a) \quad\text{and}\quad \B_b:=\core_{5\mu}\A(b) \quad\text{with $a\in\B_a$ and $b\in\B_b$}\,.
\]
We remind the reader that if the endpoints $a$ and $b$ both lie ``deep inside'' some LFS annulus to which 
one of \rfs{P:D*:ell-bad-arcs} or \ref{P:Dc:ell-bad-arcs} or \ref{P:A:ell-bad-arcs} applies, then we are done.  In particular, we may assume that $b\not\in \B_a$ and $a\not\in \B_b$.  In fact, below in \rf{ss:Ba cap Bb empty} we explain why we can (and do) assume that $\B_a\cap \B_b=\emptyset$.

Since each quasi-geodesic $\Gam^g$ leaves both annuli $\B_a$ and $\B_b$, there are unique disjoint collars $\mfC_a$ and $\mfC_b$ (of $\core_{5\mu}(\B_a)$ relative to $\B_a$ and of $\core_{5\mu}(\B_b)$ relative to $\B_b$, respectively) that each $\Gam^g$ crosses.

We define $c^g$ and $d^g$ to be, respectively, the first and last points of $\Gam^g$ that lie in $\overline\mfC_a$ and $\overline\mfC_b$.  In particular, $\bp\le\MM$ at all four points $c^h, c^k, d^h, d^k$.  Also, \rf{TT:bad??} applies to the four pairs of points $a,c^g$ and $b,d^g$ ($g=h,k$).  Thus
\begin{subequations}\label{E:MC3b:ells qex}
  \begin{align}
    \ell_g(\Gam^h[a,c^h])&\eqx\ell_g(\Gam^k[a,c^k]) \quad\text{for $g=h,k$}        \label{E:MC3b:ells eqx a}
    \intertext{and}
    \ell_g(\Gam^h[d^h,b])&\eqx\ell_g(\Gam^k[d^k,b]) \quad\text{for $g=h,k$}\,.        \label{E:MC3b:ells eqx b}
  \end{align}
\end{subequations}

We claim that
\begin{equation}\label{E:MC3b:**}
  \ell_g(\Gam^h[c^h,d^h]_g)\eqx\ell_g(\Gam^k[c^k,d^k]) \quad\text{for $g=h,k$}\,.
\end{equation}
Then $\ell_g(\Gam^h)\eqx\ell_g(\Gam^k)$ follows by adding up the appropriate inequalities in \eqref{E:MC3b:ells qex} and \eqref{E:MC3b:**}.

To verify \eqref{E:MC3b:**}, we first note that the assumption $\B_a\cap\B_b=\emptyset$---along with our choices of $c^g$ and $d^g$---ensure that for all choices of $c\in\{c^h,c^k\}$ and $d\in\{d^h,d^k\}$, $h(c,d)\ge2\log(3/2)$ (by \rf{C:h ge ln3/2}) and $k(c,d)\ge10\mu$ (because $c$ and $d$ are separated by $\mfC_a\cup\mfC_b$); also, $k(c^h,c^k)\le2\pi$ and $k(d^c,d^k)\le2\pi$.  As above (in \rf{ss:***}) we readily deduce that
\[
  k(c^h,d^h) \eqx k(c^k,d^k) \quad\text{and}\quad  h(c^h,d^h) \eqx h(c^k,d^k)\,.
\]
Since $\bp\le\MM$ at all four points $c,d$, it follows (from Main Case II) that
\begin{gather*}
  \ell_h(\Gam^k[c^k,d^k]) \eqx \ell_h([c^k,d^k]_h) = h(c^k,d^k) \eqx h(c^h,d^h) \eqx \ell_h(\Gam^h[c^h,d^h])
  \intertext{which is \eqref{E:MC3b:**} with $g=h$.  Similarly,}
  \ell_k(\Gam^h[c^h,d^h]) \eqx \ell_k([c^h,d^h]_k) = k(c^h,d^h) \eqx k(c^k,d^k) \eqx \ell_k(\Gam^k[c^k,d^k])
\end{gather*}
which is \eqref{E:MC3b:**} with $g=k$.

\subsubsection{Assuming $\B_a\cap \B_b=\emptyset$} 	\label{ss:Ba cap Bb empty}%
Finally, we explain why it is ok to assume that $\B_a\cap \B_b=\emptyset$.  In fact, just from the definitions of the cores $\B_a$ and $\B_b$ we (eventually) see that $\B_a\cap \B_b\neq\emptyset$ can occur \ifff either $\A(a)$ and $\A(b)$ are both punctured disks or both complements of closed disks; in each of these situations we show that both points $a$ and $b$ lie in the same LFS annulus to which 
one of \rfs{P:D*:ell-bad-arcs} or \ref{P:Dc:ell-bad-arcs} or \ref{P:A:ell-bad-arcs} applies.

One more time, we must examine the (six) cases that arise from the various possibilities for $\A(a)$ and $\A(b)$.  We leave it to the industrious reader to check that $\B_a\cap \B_b\neq\emptyset$ can occur \ifff either $\A(a)$ and $\A(b)$ are both punctured disks or both complements of closed disks, and we examine these two cases. \flag{Should we prove this?  I cannot imagine that anybody will actually read all of this!  Are you reading this Steve?}

First, suppose $\A(a)=\D_*(\zeta;R)$ and $\A(b)=\D_*(\vth;S)$ are both punctured disks.  Assume $\B_a\cap \B_b\ne\emptyset$.  We claim that either $b\in\D_*(\zeta;e^{-3\mu}R)$ or $a\in\D_*(\vth;e^{-3\mu}S)$, and in both cases  \rf{P:D*:ell-bad-arcs} applies.  Now $\B_a\cap \B_b\ne\emptyset$ implies that
\begin{gather*}
  |\zeta-\vth| \le e^{-5\mu}(R+S) \le 2e^{-5\mu}(R\vee S)\,.
  \intertext{Also, $\del(b)=e^{-\bp(b)}S\le e^{-M}S$, so}
  |b-\zeta| \le |b-\vth|+|\vth-\zeta| = \del(b)+|\vth-\zeta| \le e^{-M}S+|\zeta-\vth|
  \intertext{and similarly, $|a-\vth|\le e^{-M}R+|\zeta-\vth|$.  Therefore,}
  |a-\vth|\vee|b-\zeta| \le e^{-M}(R\vee S) + 2e^{-5\mu}(R\vee S) \le e^{-3\mu}(R\vee S)
\end{gather*}
as claimed.

Next, suppose $\A(a)=\mfC\sm\D[\zeta;R]$ and $\A(b)=\mfC\sm\D[\vth;S]$ are both complements of closed disks.  Assume $\B_a\cap \B_a\ne\emptyset$.  We claim that $a\in\mfC\sm\D[\vth;e^{3\mu}S]$, so \rf{P:Dc:ell-bad-arcs} applies.  Note that $\Om^c\subset\D[\zeta;R]\cap\D[\vth;S]$.   Recall that, by construction of the annuli $\A(a)$ and $\A(b)$, there are $\xi,\eta\in\bOm$ such that $R=|\zeta-\xi|$ and $S=|\vth-\eta|$.  Thus $R\le2S$ and $S\le2R$.  Now $|a-\zeta|=\del(a)=e^{\bp(a)}R\ge\half e^{\bp(a)} S$ and $|\zeta-\vth|\le R\wedge S\le S$, so
\[
  |a-\vth| \ge |a-\zeta|-|\zeta-\vth| \ge \bigl(\half\,e^{\bp(a)}-1\bigr)S \ge \bigl(\half\,e^{M}-1\bigr)S \ge e^{3\mu}S
\]
as claimed.

\subsection{Other Metrics}  \label{s:F&KPT} %
The so-called Ferrand metric and the Kulkarni-Pinkall-Thurston metric are \mob\ invariant conformal metrics defined on proper subdomains of the Riemann sphere.  (See \cite{DAH-univ-cvxty} for their definitions and other references.) So, it is worth recording the following consequence of \rf{TT:geos}.

\begin{cor} \label{C:F&KPY} %
In any hyperbolic domain $\Om$ in the Riemann sphere $\RS$, the following classes of curves are exactly the same, with quantitative estimates between their lengths.
\begin{itemize}
  \item  The hyperbolic quasi-geodesics in $\Om$.
  \item  The Ferrand quasi-geodesics in $\Om$.
  \item  The Kulkarni-Pinkall-Thurston quasi-geodesics in $\Om$.
\end{itemize}
\end{cor} 
\begin{proof}%
Since the three metrics in consideration are all \mob\ invariant, we may assume that $\Om$ is a plane domain.  (If $\Om\not\subset\mfC$, select any point in $\mfC\sm\Om$ and replace $\Om$ by its image under inversion at this point.)  In this setting, we know that both the Ferrand metric and the Kulkarni-Pinkall-Thurston metric are $2$-\BL\ equivalent to the quasihyperbolic metric, and so the assertions follow immediately from \rf{TT:geos}.
\end{proof}%
%
\section{Gromov Hyperbolicity}      \label{S:Gromov bs} 
Here we recall the definition of a Gromov hyperbolic space and the notion of geodesic stability, and present a chordarc surgery lemma.  Then we establish \rf{TT:GrHyp}.

A geodesic metric space $X$ is \emph{Gromov hyperbolic} if there exists a constant $\Th\ge0$ such that every geodesic triangle is \emph{$\Th$-thin}, meaning that each point on any edge of the triangle is at distance at most $\Th$ from the other two edges.

An important property of Gromov hyperbolicity is that each quasi-geodesic is not far from a geodesic.  Thus every quasi-geodesic triangle in a $\Th$-hyperbolic space, is $C\Th$-thin where $C$ depends only on the quasi-geodesic constant.  This even holds for rough-quasi-geodesics; see \cite[Theorem~III.H.1.7, p.401]{Brid-Haef} or \cite[Theorem~8.4.20, p.290]{BB-gromov}.

Bonk proved the following splendid theorem which asserts that the above so-called \emph{geodesic stability} property actually characterizes Gromov hyperbolicity; see \cite{Bonk-quasigdscs}.

\begin{fact}  \label{F:MB} %
A geodesic metric space if Gromov hyperbolic \ifff it is geodesically stable.
\end{fact} %

The following describes a simple type of ``chordarc surgery'' that we use repeatedly and call \emph{pruning the terminal end of $\alf$}; one can also prune the initial end.  Bonk has a more general result (see \cite[Lemma~2.4]{Bonk-quasigdscs}), but we need better control on the precise location of the new quasi-geodesic.  Recall from \eqref{E:chi} our definition of the circular-arc-segment path $\chi_{ab}$ from $a$ to $b$, where $a$ and $b$ are points on separate boundary circles of some annulus.  Recall also \rf{R:k ests in A}. 

\begin{lma} \label{L:CA surgery} %
Let $\alf$ be a quasihyperbolic $\Lam$-chordarc path with initial point $o$.  Let $A$ be 
a concentric subannulus of the $\log2$-core of some annulus in $\mcA_\Om(\log4)$.  Suppose that $o\not\in\cA$ and that $\alf$ meets $A$.  Let $a\in\bA$ be the first point of $\alf$ in $\cA$ and let $b$ be any point in the other component of $\bA$.  Put $\chi:=\chi^{-1}_{ba}$.\footnote{So, $\chi=[a,c]\star\kap$ is the \emph{reverse} of $\chi_{ba}$, where $\kap$ is a circular arc from $c$, the radial projection of $a$, to $b$.}  Then $\beta:=\alf[o,a]\star\chi$ is a quasihyperbolic $M$-chordarc path with $M=3\Lam+(2\pi/\md(A))(\Lam+1)+2$; if $\md(A)\ge2\pi$, $M=M(\Lam)\le4\Lam+3$.
\end{lma} 
\begin{proof}%
Let $x,y\in|\beta|$.  We show that $\ell_k(\beta[x,y])\le M k(x,y)$.  According to \rf{L:k ests in A}, this holds with $M=4\vee\Lam$ if $x,y$ both lie in $|\alf[o,a]|$ or both in $|\chi|$.  Thus we assume that $x\in|\alf[o,a]|$ and $y\in|\chi|$.  We also assume $c(A)=0$, so e.g., $c=(|b|/|a|)a$.

Suppose $y\in[a,c]$.  From \eqref{E:k ge log} we have
\begin{gather*}
  k(x,y) \ge \bigl|\log\frac{|x|}{|y|}\bigr|  \ge \bigl|\log\frac{|a|}{|y|}\bigr|
  \intertext{and thus by \rf{R:k* ests} and \eqref{E:k* ests}}
  k(a,y) \le \ell_k([a,y]) \le 2 \bigl|\log\frac{|a|}{|y|}\bigr| \le 2k(x,y)\,.
  \intertext{Therefore,}
  \begin{align*}
    \ell_k(\beta[x,y]) &= \ell_k(\alf[x,a])+\ell_k([a,y]) \le \Lam k(x,a) + 2 \bigl|\log\frac{|a|}{|y|}\bigr| \\
                       &\le \Lam \bigl(k(x,y)+k(y,a)\bigr) + 2k(x,y) \le (3\Lam +2)(k(x,y)\,.
  \end{align*}
\end{gather*}

Now suppose $y\in|\kap|$.  Again, \eqref{E:k ge log} tells us that
\begin{gather*}
  m:=\md(A) = \bigl|\log\frac{|a|}{|b|}\bigr| = \bigl|\log\frac{|a|}{|c|}\bigr| \le k(x,y)\,.
  \intertext{Also, by \eqref{E:kap QG} and \eqref{E:[b,c] QG},}
  k(a,y) \le \ell_k(\chi) = \ell_k([a,c])+\ell_k(\kap) \le 2m+2\pi \le (2+2\pi/m)k(x,y)\,.
  \intertext{Thus}
  k(x,a) \le k(x,y)+k(y,a) \le (3+2\pi/m) k(x,y)
  \intertext{and therefore}
  \begin{align*}
    \ell_k(\beta[x,y]) &= \ell_k(\alf[x,a])+\ell_k(\chi_{ab}[a,y]) \le \Lam k(x,a) + \ell_k(\chi_{ab}) \\
                       &\le \Lam(3+2\pi/m) k(x,y) + (2+2\pi/m) k(x,y) \\
                       &\le \bigl(3\Lam+(\Lam+1)\frac{2\pi}m + 2\bigr) k(x,y)\,. \qedhere
  \end{align*}
\end{gather*}
\end{proof}%

The following simple observation is repeatedly applied to handle various cases and subcases in our argument for \rf{TT:GrHyp}.

\begin{lma} \label{L:obs} %
Let $\alf$ be a quasihyperbolic $\Lam$-chordarc path with endpoints $a,b$.  Let $\mu=\mu(\Lam)$ be the asssociated ABC parameter.  Let $\beta$ be a subarc of $\alf$ with endpoints $x,y$ such that $a\le x\le y\le b$; so, $\beta=\alf[x,y]$.  Suppose there is an annulus $A\in\mcA_\Om(6\mu)$ with $c(A)=0$ and such that both
\[
  x,y\in\core_{3\mu}(A) \quad\text{and}\quad  [a,b]_k \;\text{ crosses }\; C:=\{r\le|z|\le s\}
\]
where $r:=|x|\wedge|y|$ and $s:=|x|\vee|y|$.\footnote{We allow the possibility that $r=s$ in which case $C$ is a circle.}  Then for any $p\in|\beta|$ there exists a point $q\in[a,b]_k$ such that $k(p,q)\le 4\mu+2\pi$.
\end{lma} 
\begin{proof}%
Clearly, $C\csubann\core_{3\mu}(A)$ and $B:=\band_{2\mu}(C) 
\csubann\core_{\log2}(A)$; the last containment is a simple calculation.  As each collar of $C$ relative to $B$ has modulus $2\mu$, \rf{L:ABC}(b) asserts that $|\beta|\subset B$.  Let $p\in|\beta|$.

Suppose $p\in C$.  Pick $q\in[a,b]_k\cap\mfS^1(0;|p|)$.  Then by \eqref{E:k* ests} and \rf{R:k* ests}, $k(p,q)\le2\pi$.

Suppose $p\not\in C$.  Let $p'$ be the radial projection of $p$ onto $C$.  Note that $\bigr|\log(|p|/|p'|)\bigr|\le2\mu$.  Pick $q\in[a,b]_k\cap\mfS^1(0;|p'|)$.  Again by \eqref{E:k* ests} and \rf{R:k* ests}, $k(p,q)\le k(p,p')+k(p',q)\le4\mu+2\pi$.
\end{proof}%
\noindent
In fact we also use modifications of the above.  \flag{if $c=a$ or $d=b$, prolly get better result....}

\subsection{Start of Proof of Theorem~\ref{TT:GrHyp}}  \label{s:GrHyp pf} %
We show that whenever one of the metric spaces $(\Om,h)$ or $(\Om,k)$ is Gromov hyperbolic, so is the other space; in fact, the hyperbolicity constants depend only on each other.

\medskip

First, suppose $(\Om,k)$ is $\Th$-hyperbolic.  Then quasihyperbolic quasi-geodesic triangles in $\Om$ are $C\Th$-thin where the constant $C$ depends only on the quasi-geodesic data.  Let $\Del=\Del_h=[a,b]_h\cup[b,c]_h\cup[c,a]_h$ be a hyperbolic geodesic triangle in $\Om$.  Thanks to \rf{TT:geos} we know that $\Del$ is also a quasihyperbolic quasi-geodesic triangle where the quasi-geodesic constant is an absolute constant.  Since $h\le2k$, it now follows that $\Del$ is hyperbolically $2C\Th$-thin where $C$ is an absolute constant.  Thus $(\Om,h)$ is also Gromov hyperbolic.

\medskip


Next, suppose $(\Om,h)$ is $\Th$-hyperbolic.  Thanks to \rf{F:MB},  it suffices to demonstrate that $(\Om,k)$ is geodesically stable.
We must verify that each point of any given quasihyperbolic quasi-geodesic is quasihyperbolically roughly close to some fixed quasihyperbolic geodesic (with the same endpoints) with a uniform constant that depends only on the quasi-geodesic data and $\Th$.

Let $\Lam\ge1$.  We produce a constant $D=D(\Lam,\Th)$ such that for any quasihyperbolic $\Lam$-chordarc path $\alf$ with endpoints $a$ and $b$, there exists a fixed quasihyperbolic geodesic $[a,b]_k$ such that for each $p\in|\alf|$ there is a $q\in[a,b]_k$ with $k(p,q)\le D$.  So, let $\alf$ be such a path, and let $[a,b]_k$ be any fixed quasihyperbolic geodesic with the same endpoints as $\alf$.  For convenience, we say that $p\in|\alf|$ is \emph{$D$-close} to $[a,b]_k$ if $k(p,[a,b]_k)\le D$. 

Let $\mu=\mu(\Lam)=\pi\Lam$ be the ABC parameter for quasihyperbolic $\Lam$-chordarc paths, and put $\SSS:=10\mu, \MM:=10\SSS, \LL:=10\MM$.

\smallskip

As in the proof of \rf{TT:geos}, our argument involves several cases each having subcases and subsubcases.  There are two Easy Cases where we assume that $\bp$ is bounded (either above or below) on $\alf$, and then two remaining Main Cases.  Surprisingly, Easy Case II is the heart of our argument, at least in the sense that in Main Case I we reduce back to Easy Case II and in Main Case II we reduce back to Main Case I.  All of this relies heavily on the ideas and constructions in our proof of \rf{TT:geos}.

\subsubsection{Easy Case I} 	\label{ss:thmB ECI} %
Here we assume that along $\alf$, $\bp\ge P$ for some constant $P>6\mu$.  Thanks to \rf{L:bp ge tau} we know that $|\alf|\subset\overline\core_{P/2}\A(a)\csubann\core_{3\mu}\A(a)$.  Appealing to \rf{L:obs} (with $\beta=\alf$) we see that each point of $|\alf|$ is $5\mu$-close to $[a,b]_k$.

\subsubsection{Easy Case II} 	\label{ss:thmB ECII} %
Here we assume that along $\alf$, $\bp\le P$ for some constant $P>\pi$.  Since $[a,b]_k$ has the $(\pi,\log2)$-ABC property, \rf{L:large bp} tells us that $\bp<4 P$ on $[a,b]_k$.

Let $p\in|\alf|$.  Since $(\Om,h)$ is $\Th$-hyperbolic, it is geodesically stable. As $\alf$ and $[a,b]_k$ are both hyperbolic quasi-geodesics with the same endpoints, there is a point $q$ in $[a,b]_k$ with $h(p,q)\le D_h=D_h(\Lam,\Th)$.  We claim that along $[p,q]_h$, $\bp\le10P$.  Therefore, $k(p,q)\le(\kk+10P)h(p,q)\le (\kk+10P)D_h$, so $p$ is $D$-close to $[a,b]_k$ with $D=D(\Lam,\Th)=(\kk+10P)D_h$.

If the claim were false, then we would get a LFS annulus $A$ that separates $\{p,q\}$, but then as $\alf\star[b,a]_k$ crosses $A$ we would have $\bp$ large somewhere on $\alf\star[b,a]_k$.  We supply the details.

Suppose there is a point $z_o\in[p,q]_h$ with $\bp(z_o)>10P$.  Then $\bp<4P$ at each of $p,q$ and $\bp(z_o)>10P$, so by \rfs{L:not punxd} and \ref{L:cores separate} (applied to $[p,q]_h$), $A_o:=\A(z_o)$ is a non-degenerate annulus and $\core_{8P}(A_o)$ separates $\{p,q\}$.  Evidently, $\alf[p,a]\cup[a,q]_k$ joins the points $p, q$ so
\[
  \big(|\alf|\cup[a,b]_k\bigr)\cap\core_{8P}(A_o)\ne\emptyset\,.
\]
However, according to \rf{R:bp ests}(c), in $\core_{8P}(A_o)$ we have $\bp>\half 8P=4P$ which then contradicts $\bp<4P$ on $|\alf|\cup[a,b]_k$.


\medskip

Now we examine the Main Cases where we assume that $\bp$ is both smaller than $\SSS$ at some points of $|\alf|$ and larger than $\LL$ at some points of $|\alf|$.

\subsection{Main Case I in Proof of Theorem~\ref{TT:GrHyp}} 	\label{ss:thmB MCI} %
Here we assume that $\bp(a)\le\SSS$ and $\bp(b)\le\SSS$.  This permits us to utilize the constructions from Main Case I in the proof of \rf{TT:geos}; see \rf{ss:both leS}.  In particular, we get an integer $n\ge1$ and points
\begin{gather*}
    a=:a^h_0<b^h_0<z_1<a^h_1<  \dots <  b^h_{i-1} < z_i < a^h_i < \dots <b^h_{n-1} < z_n < a^h_n < b^h_n:=b
  \intertext{along the quasihyperbolic quasi-geodesic $\alf$ such that}
    \alf:=\gam_0 \star \beta_1 \star \gam_1 \star \dots \star \beta_n \star \gam_n
  \intertext{where the ``good'' and ``bad'' subarcs of $\alf$ are}
  \gam_i:=\alf[a_i,b_i] \;\;(0\le i\le n) \quad\text{and}\quad \beta_i:=\alf[b_{i-1},a_i] \;\;(1\le i\le n)\,.
\end{gather*}
Recall too that $b_{i-1}$ and $a_i$ are, respectively, the first and last points of $\alf$ in $\overline\core_{2\MM}(A_i)$ where $A_i:=\A(z_i)$ are non-degenerate annuli that each separate $\{a,b\}$ and that satisfy $\core_q(A_i)\cap\core_q(A_j)=\emptyset$ for any $q\in(\log16,\LL)$.

In particular, $[a,b]_k$ crosses $\overline\core_{2\MM}(A_i)$ and $|\beta_i|\subset\core_\MM(A_i)\csubann\core_{3\mu}(A_i)$.  It therefore follows, from \rf{L:obs}, that each point of each $\beta_i$ is $5\mu$-close to $[a,b]_k$.  Thus it remains to reveal why the points of each $\gam_i$ are quasihyperbolically roughly close to the quasihyperbolic geodesic $[a,b]_k$.

Fix $0<i<n$ and consider the ``good'' subarc $\gam_i$; the two subcases where $\gam_i$ is a ``tail'' of $\alf$ (i.e., $i=0$ or $i=n$) are left for the diligent reader.  Let $c$ and $d$ be, respectively, the last and first points of $\gam_i$ in $\overline\core_{\MM}(A_i)$ and $\overline\core_{\MM}(A_{i+1})$.  Then let $\tilde{c}$ and $\tilde{d}$ be the radial projections of $c$ and $d$, respectively, onto $\overline\core_{2\MM}(A_i)$ and $\overline\core_{2\MM}(A_{i+1})$ (so, $\tilde{c}\in\bd\core_{2\MM}(A_i)$ and $\tilde{d}\in\bd\core_{2\MM}(A_{i+1})$); here by radial projection we mean \wrt the centers of the two annuli $A_i$ and $A_{i+1}$, respectively.

Next, pick any point $\tilde{a}\in[a,b]_k\cap\bd\core_{2\MM}(A_i)$ so that $a_i, \tilde{a}, \tilde{c}$ all lie in the same boundary circle, and pick $\tilde{b}\in[a,b]_k\cap\bd\core_{2\MM}(A_{i+1})$ so that $b_i, \tilde{b}, \tilde{d}$ all lie in the same boundary circle.  We are now positioned to perform (twice) the ``chordarc surgery'' as described in \rf{L:CA surgery}; that is, we prune both ends of $\alf$.  Doing this we obtain a quasihyperbolic quasi-geodesic
\[
  \tilde\alf:=\chi_{\tilde{a} c} \star \tilde\gam \star \chi_{\tilde{b}d}^{-1} \quad\text{where}\quad
  \tilde\gam:=\gam_i[c,d],\;
  \chi_{\tilde{a}c}=\kap_{\tilde{a}\tilde{c}}\star[\tilde{c},c],\;
  \chi_{\tilde{b}d}=\kap_{\tilde{b}\tilde{d}}\star[\tilde{d},d]\,.
\]
Here the circular-arc-segment paths $\chi_{\tilde{a}c}$ and $\chi_{\tilde{b}d}$ are constructed as described in \eqref{E:chi}.

It is straightforward to check that $\tilde\alf$ is a quasihyperbolic $\tilde\Lam$-chordarc path where $\tilde\Lam=4\Lam+3$; see the proof of \rf{L:CA surgery}, but note too that the quasihyperbolic distance between the arcs $\chi_{\tilde{a} c}$ and $\chi_{\tilde{b}d}$ is at least $2(\MM-\SSS)=(9/5)\MM$.

Now by construction, $\bp\le\LL$ along $\gam_i$, so also along $\tilde\gam$.  Next, as $|\chi_{\tilde{a} c}|\subset\core_{\log16}(A_i)\sm\core_{2\MM}(A_i)$ and $|\chi_{\tilde{b}d}|\subset\core_{\log16}(A_{i+1})\sm\core_{2\MM}(A_{i+1})$, we know that $\bp\le 4\MM$ on each of these arcs (e.g., by \rf{R:bp ests}(b)).  Thus $\bp\le\LL$ on all of $|\tilde\alf|$.   Therefore, we can apply Easy Case~II, with $P:=\LL=10^3\mu=10^3\pi\Lam$, to the quasihyperbolic $\tilde\Lam$-chordarc path $\tilde\alf$ (and $[\tilde{a},\tilde{b}]_k\subset[a,b]_k$) to assert that each point of $|\tilde\alf|$ is $D$-close to $[\tilde{a},\tilde{b}]_k$ with $D=D(\Lam,\Th)=(\kk+10\LL)D_h$.  In particular, each point of $|\tilde\gam|$ is $D$-close to $[a,b]_k$.

It remains to examine points in $|\gam_i[a_i,c]|\cup|\gam[d,b_i]|$.  Our choices of $a_i, b_i, c, d$ give us that $|\gam_i[a_i,c]|=|\alf[a_i,c]|\subset\overline\core_{\MM-2\mu}(A_i)\sm\core_{2\MM}(A_i)$ and $|\gam_i[d,b_i]|=|\alf[d,b_i]|\subset\overline\core_{\MM-2\mu}(A_{i+1})\sm\core_{2\MM}(A_{i+1})$.  Using \eqref{E:k* ests} and \rf{R:k* ests} we deduce that for each $p\in|\gam_i[a_i,c]|$,
\[
  k(p,\tilde{a}) \le 2\pi+2(\MM+2\mu) \le 205\mu\,,
\]
and similarly each point of $|\gam[d,b_i]|$ is also $205\mu$-close to $\tilde{b}$.

\smallskip

Thus, when $\bp\le\SSS$ at both endpoints of $\alf$, each point in $|\alf|$ is $D$-close to $[a,b]_k$ with $D=D(\Lam,\Th)=(\kk+10\LL)D_h$.\footnote{Here we assume that $D_h\ge1$.}  Here $\SSS=\SSS(\Lam)=10\mu=10\pi\Lam$ and $\LL=\LL(\Lam)=100\SSS$.

\subsection{Main Case II in Proof of Theorem~\ref{TT:GrHyp}} 	\label{ss:thmB MCII} %
As above, we must demonstrate that there is a constant $D=D(\Lam,\Th)$ such that each $p\in|\alf|$ is $D$-close to $[a,b]_k$.  To this end, eventually, we use ``chordarc surgery'' (see \rf{L:CA surgery}) to prune the ends of $\alf$ and replace it by another quasihyperbolic chordarc path $\alf':=\chi_{a'c}\star\alf[c,d]\star\chi^{-1}_{b'd}$ where: $c,d\in|\alf|$; the endpoints $a',b'$ of $\alf'$ lie in $[a,b]_k$ with $\bp\le\SSS':=4\SSS$ at each of $a',b'$; and such that all points of $|\alf[a,c]|\cup|\alf[d,b]|$ are quasihyperbolically roughly close to $[a,b]_k$.  Once this is accomplished, we then apply Main Case I to $\alf'$ and $[a',b']_k\subset[a,b]_k$ to assert that all points of $|\alf'|$ are quasihyperbolically roughly close to $[a',b']_k\subset[a,b]_k$.  Since $|\alf'|\supset\alf[c,d]$, we then see that all points of $|\alf|$ are quasihyperbolically roughly close to $[a,b]_k$.


In a number of special cases we can immediately assert that all points of $|\alf|$ are quasihyperbolically roughly close to $[a,b]_k$; in these cases we are done with the proof.  For example, according to \rf{L:obs}, if there is an annulus $A\in\mcA_\Om(6\mu)$ with $a,b\in\core_{3\mu}(A)$, then each $p\in|\alf|$ is $5\mu$-close to $[a,b]_k$.  In particular, we may, and do, assume that for any such $A$, $\{a,b\}\not\subset\core_{3\mu}(A)$.

\medskip

Our first goal is to prune the $a$-end of $\alf$.  We seek points $a'\in[a,b]_k$ and $c\in|\alf|$ such that $\bp(a')\le\SSS'=4\SSS$, each $p\in|\alf[a,c]|$ is quasihyperbolically roughly close to $[a,b]_k$, and so that $a',c$ lie on the boundary of some thin (not too fat) annulus.  Once obtained, we replace $\alf$ by the quasihyperbolic chordarc path $\beta:=\chi_{a'c}\star\alf[c,d]$.

\subsubsection{First Step} 	\label{ss:thmB-MCII-FS} %
If already $\bp(a)\le\SSS'$, then we set $c:=a':=a$, $\beta:=\alf$, and go to \rf{ss:thmB-MCII-PS} (the Penultimate Step).

\smallskip  

Assume $\bp(a)>\LL$.  Put $A:=\A(a)$ and assume $c(A)=0$; $A$ could be a degenerate annulus.  If $b\in\core_{3\mu}(A)$, then we are done (by \rf{L:obs}); assume $b\not\in\core_{3\mu}(A)$.  Recalling that $a\in\bd\core_{\bp(a)}(A)\subset\core_\LL(A)$ (see \eqref{E:z in core bdy}), we see that there is a unique collar $C$ of $\core_{2\SSS}(A)$ relative to $\core_\SSS(A)$ with the property that both $\alf$ and $[a,b]_k$ cross $C$.  (If $A$ is degenerate, there is only one such collar, while if $A$ is non-degenerate there are two such collars but only one that gets crossed.)

Let $c$ be the last point of $\alf$ in $\overline\core_\SSS(A)$ and let $c'$ be the radial projection of $c$ onto $\overline\core_{2\SSS}(A)$; so $c\in\bd C\cap\bd\core_\SSS(A)$, $c'\in\bd C\cap\bd\core_{2\SSS}(A)$, and $[c,c']$ is a radial segment of $C$ that joins the boundary circles of $C$.
It is not difficult to check that $a,c\in\core_{3\mu}(A)$, and thus by \rf{L:obs} all points of $|\alf[a,c]|$ are $5\mu$-close to $[a,b]_k$.

Now pick any point $a'\in[a,b]_k\cap\bd\core_{2\SSS}(A)$.  Then $a'$ and $c'$ lie in the same component of $\bd C$ (i.e., in the same boundary circle).  Thus there is a circular arc $\kap':=\kap_{a'c'}$ and a quasihyperbolic chordarc path $\chi_{a'c}:=\kap'\star[c',c]$ from $a'$ through $c'$ to $c$ (cf.\ \eqref{E:chi}).  Therefore, we can do the ``chordarc surgery'' as described in \rf{L:CA surgery} to obtain the quasihyperbolic $M$-chordarc path $\beta:=\chi_{a'c}\star\alf[c,b]$ from $a'$  to $b$ with $M=4\Lam+3$.

Finally, note that $a'\in\core_{\log16}(A)\sm\core_{2\SSS}(A)$, so by \rf{R:bp ests}(c) $\bp(a')\le4\SSS=\SSS'$.  Now we go to \rf{ss:thmB-MCII-PS} (the Penultimate Step).

\smallskip  
Assume $\SSS'<\bp(a)\le\LL$.  Let $z_o$ be the first point $z$ of $\alf$ with $\bp(z)=\LL$.  Put $A_o:=\A(z_o)$ (which could be a degenerate annulus) and assume $c(A_o)=0$.  Eventually---see \rf{ss:thmB-MCII-NS}---we will deduce that we can assume $b\not\in\core_\SSS(A_o)$ and so we can mimic the above argument to prune the $a$-end of $\alf$ via ``chordarc surgery''.

Again, we may assume $\{a,b\}\not\subset\core_{3\mu}(A_o)$; but, note that---as $z_o\in|\alf|$---the quasi-geodesic $\alf$ enters deep into $A_o$.  Let $S_o$ be the boundary circle of $\core_\LL(A_o)$ that contains $z_o$ (if $A_o$ is degenerate, then $S_o=\bd\core_\LL(A_o)$) and for $j=1,2$ let $B_j:=\band_{2j\mu}(S_o)$\footnote{So, $B_1$ and $B_2$ are open annuli with $\mfS^1(B_j)=S_o$ and $\md(B_j)=4j\mu$.}.  Then
\[
  S_o\subset B_1\csubann B_2\csubann \core_{\LL-4\mu}(A_o) \csubann \core_{3\mu}(A_o)\,,
\]
so $\{a,b\}\not\subset B_2$, and if $\{a,b\}\cap B_1=\emptyset$, then $B_1$ separates $\{a,b\}$.
It is not difficult, albeit tedious, to check (e.g., using \rf{R:bp ests}(c)) that for both $j=1$ and $j=2$,
\[
  \half\bigl( \LL-2j\mu\bigr) \le \bp \le 2\bigl(\LL+2j\mu\bigr) \quad\text{in $B_j$}\,.
\]

Suppose $b\in B_1$.  Then $|\alf[z_o,b]|\subset B_2$ (by the ABC property) and therefore $\bp\le2(\LL+4\mu)$ on $|\alf[z_o,b]|$.  Since $\bp\le\LL$ on $\alf[a,z_o]$ (by construction), we can appeal to Easy Case II (with $P:=2(\LL+4\mu)$) to conclude that each $p\in|\alf|$ is $D$-close to $[a,b]_k$ with $D:=[\kk+20(\LL+4\mu)]D_h$.  Here we are done with the proof.

Suppose $b\not\in B_1$.  We claim that each point of $|\alf[a,z_o]|$ is quasihyperbolically roughly close to $[a,b]_k$.  First, suppose $a\in B_1$ (the easy case).  Then $a,z_o\in B_1\csubann\core_{3\mu}(A_o)$, and we would like to appeal to \rf{L:obs} but cannot.  We do know that $|\alf[a,z_o]|\subset B_2$ (by the ABC property).  Let $p\in|\alf[a,z_o]|$.  Put $p':=\bigl(|a|/|p|\bigr)p$.  Since $a,b\in B_2\csubann\core_{\log2}(A_o)$, \eqref{E:k* ests} and \rf{R:k* ests} tell us that
\[
  k(p,a) \le k(p,p') + k(p',a) \le 2\bigl|\log\frac{|p|}{|p'|}\bigr|+2\pi \le 8\mu+2\pi\,.
\]

Next, suppose $a\not\in B_1$.  This subcase requires a wee bit of work.  Here $\{a,b\}\cap B_1=\emptyset$, so $B_1$ separates $\{a,b\}$ and we can write $\bB_1=S_a\cup S_b$ where $S_a$ and $S_b$ are, respectively, the first and last components of $\bB_1$ that $\alf$ meets.  Now we do ``chordarc surgery'' replacing $\alf[a,z_o]$ with another quasihyperbolic quasi-geodesic $\tilde\alf$ which has endpoints $a,\tilde{b}\in[a,b]_k$ and with $\bp\lex1$ along all of $\tilde\alf$.  Then Easy Case II tells us that all points in $|\tilde\alf|$ are quasihyperbolically roughly close to $[a,b]_k$, and it is not hard to check that the same is true for points in $|\alf[a,z_o]|\sm|\tilde\alf|$.

With this goal in mind, let $z_a$ be the first point of $\alf$ in $S_o$ and pick any point $\tilde{b}\in[a,b]_k\cap S_b$.  Now do ``chordarc surgery'' via \rf{L:CA surgery} using the quasi-geodesic $\alf$, the annulus with boundary circles $S_o$ and $S_b$, and the points $z_a$ and $\tilde{b}$ to get $\tilde\alf:=\alf[a,z_a]\star\chi_{z_a\tilde{b}}$ which is a quasihyperbolic $M$-chordarc path with endpoints $a,\tilde{b}\in[a,b]_k$ and $M=4\Lam+3$.

Clearly, $|\tilde\alf[a,z_a]|=|\alf[a,z_a]|\subset|\alf[a,z_o]|$, so $\bp\le\LL$ on $|\tilde\alf[a,z_a]|$.  Also, $|\tilde\alf[z_a,\tilde{b}]|=|\chi_{z_a\tilde{b}}|\subset\cB_1$, so $\bp\le2(\LL+2\mu)$ on $|\tilde\alf[z_a,\tilde{b}]|$.  Appealing to Easy Case II we deduce that each $p\in|\tilde\alf|$ is $D$-close to $[a,\tilde{b}]_k\subset[a,b]_k$ with $D:=[\kk+20(\LL+2\mu)]D_h$.

Since $\alf[a,z_a]=\tilde\alf[a,z_a]$, it remains to check points in $|\alf[z_a,z_o]|\subset B_1$.  As $z_a$ and $z_o$ both lie in the circle $S_o$, we can proceed as in the easy case (where $a\in B_1$---see four paragraphs above) to see that these points are $5\mu$-close to $[a,b]_k$.

\subsubsection{Next Step} 	\label{ss:thmB-MCII-NS} %
Here we continue from the First Step; we have $\SSS'<\bp(a)\le\LL$, $b\not\in B_1$, and we know that all points $p\in|\alf[a,z_o]|$ are $D_1$-close to $[a,b]_k$ with $D_1:=[\kk+20(\LL+4\mu)]D_h$.  To start, we explain why we may assume that $b\not\in\core_\SSS(A_o)$.

Suppose $b\in\core_\SSS(A_o)$.  Then $a\not\in\core_\SSS(A_o)$, since $\core_\SSS(A_o)\csubann\core_{3\mu}(A_o)$.  In particular, $\{a,b\}\cap B_1=\emptyset$ and $B_1$ separates $\{a,b\}$.  Therefore, $[a,b]_k$ crosses the annulus with boundary circles $S_o$ and $\mfS^1(0;|b|)$.  Since $z_o,b\in\core_{3\mu}(A_o)$, we can appeal to \rf{L:obs} to assert that each $p\in|\alf[z_o,b]|$ is $5\mu$-close to $[a,b]_k$.  Here we are done with the proof.

Assume $b\not\in\core_\SSS(A_o)$.  Our goal is to prune the $a$-end of $\alf$ via ``chordarc surgery'' as described in \rf{L:CA surgery}.  We seek the appropriate points $a'\in[a,b]_k$ and $c\in|\alf|$ as described in the paragraph preceding \rf{ss:thmB-MCII-FS}.

We do not know exactly where the point $a$ lies, but as in the First Step (see the first paragraph), there is a unique collar $C$ of $\core_{2\SSS}(A_o)$ relative to $\core_\SSS(A_o)$ with the property that $\alf[z_o,b]$ crosses $C$.  (Again, if $A_o$ is degenerate, there is only one such collar, while if $A$ is non-degenerate there are two such collars but only one that gets crossed:-)

Since $\alf$ crosses $C$ and $\{a,b\}\cap C=\emptyset$, we know from \rf{L:ABC}(d) that $C$ separates $\{a,b\}$.  Therefore $[a,b]_k$ also crosses $C$.  Thus we can write $\bd C=S_a\cup S_b$ where $S_a$ and $S_b$ are, respectively, the first and last components of $\bd C$ that $\alf$ meets.

Let $c$ be the last point of $\alf$ in $S_b$, let $c'$ be the radial projection of $c$ onto $\overline\core_{2\SSS}(A_o)$ (so $c'\in S_a$), and pick any point $a'\in[a,b]_k\cap S_a$.  According to \rf{L:CA surgery}, $\beta:=\chi_{a'c}\star\alf_{cb}$ is a quasihyperbolic $M$-chordarc path with endpoints $a',b$ and $M=4\Lam+3$.  Note that as $a'\in S_a\subset\bd\core_{2\SSS}(A_o)\subset\core_{\log16}(A_o)\sm\core_{2\SSS}(A_o)$, we have $\bp(a')\le4\SSS=\SSS'$ (thanks to \rf{R:bp ests}(c)).

We claim that all points in $|\alf[a,c]|$ are quasihyperbolically roughly close to $[a,b]_k$.  We already know this for points in $|\alf[a,z_o]|$ (with constant $D_1$), so we need only check points in $|\alf[z_o,c]|$.

Let $A_1$ be the annulus with boundary circles $S_o$ and $S_b$.  Suppose $A_1$ separates $\{a,b\}$.  Then $[a,b]_k$ crosses $A_1$ and $z_o,c\in\cA_1\subset\overline\core_{2\SSS}(A_o)\csubann\core_{3\mu}(A_o)$, so \rf{L:obs} tells us that points in $|\alf[z_o,c]|$ are $5\mu$-close to $[a,b]_k$.

Suppose $A_1$ does not separate $\{a,b\}$.  Then $a\in A_1\cap B_1$.  Since $[a,b]_k$ crosses the annulus with boundary circles $\mfS^1(0;|a|)$ and $\mfS^1(0;|c|)=S_b$, \rf{L:obs} tells us that all $p\in|\alf[z_o,c]|$ with $|p|\ge|a|$ are $5\mu$-close to $[a,b]_k$.  Thus we are left to examine points in $|\alf[z_o,c]|\cap B_1$.

Take $p\in|\alf[z_o,c]|\cap B_1$ and let $p'$ be the radial projection of $p$ onto the circle $\mfS^1(0;|a|)$.  From \eqref{E:k* ests} and \rf{R:k* ests} we obtain $k(p,a)\le k(p,p')+k(p',a) \le 2\bigl|\log\bigl(|p|/|p'|\bigr)\bigr|+2\pi\le8\mu+2\pi$.

We have demonstrated that all points $p\in|\alf[a,c]|$, are $D_1$-close to $[a,b]_k$.  Note too that points in $|\chi_{a'c}|$ are $2\pi$-close to $[a,b]_k$.

\subsubsection{Penultimate Step} 	\label{ss:thmB-MCII-PS} %
We arrive here with $\alf$ replaced by the quasihyperbolic $M$-chordarc path $\beta:=\chi_{a'c}\star\alf[c,b]$ where $M=4\Lam+3$, $a'\in[a,b]_k$ with $\bp(a')\le\SSS'$, and $c\in|\alf|$ with all $p\in\alf[a,c]$ $D_1$-close to $[a,b]_k$.  If $\bp(b)\le\SSS'$, we put $d:=b':=b$, $\alf':=\beta$, and go to \rf{ss:thmB-MCII-LS} (the Last Step).

Suppose $\bp(b)>\SSS'$.  We repeat \emph{all} of above, as necessary, for the reverse quasihyperbolic quasi-geodesic $\beta^{-1}$ and $[b,a']_k$.  If any of Easy Case~I or Easy Case~II or Main Case~I applies, then we are done with the proof.  So, we may assume that we get to Main Case~II, and that we have again completed the First and Next Steps, as described in \rf{ss:thmB-MCII-FS} and \rf{ss:thmB-MCII-NS}, for $\beta^{-1}$ and $[b,a']_k$.  (At some time during these two steps, we could have been done with the proof, but we assume otherwise.)

We should exercise some caution here, so let's be explicit.  The quasihyperbolic $M$-chordarc path $\beta$ enjoys the $(\nu,\log2)$-ABC property where $\nu=\nu(M)=\pi M=\pi(4\Lam+3)$; see \rf{R:ABC}.  By performing the First and Next Steps we prune the $b$-end of $\beta$ via ``chordarc surgery'' to obtain a quasihyperbolic $N$-chordarc path $\alf':=\beta[a',d]\star\chi^{-1}_{b'd}$ where $N=4M+3$, $b'\in[a',b]_k$ with $\bp(b')\le\SSS'(M)$, and $d\in|\beta|$ with all $p\in|\beta[d,b]|$ $D_2$-close to $[a',b]_k\subset[a,b]_k$ for $D_2:=[\kk+20(\LL(M)+4\nu)]D_h$.

Here $\SSS'(M):=4\SSS(M)=40\nu=40\pi M$ and $\LL(M)=10^3\nu$.  Note too that we can assume that $d$ is some point in $\alf[c,b]$, so $\beta[a',d]=\chi_{a'c}*\alf[c,d]$ and thus $\alf'=\chi_{a'c}*\alf[c,d]\star\chi^{-1}_{b'd}$.

\subsubsection{Last Step} 	\label{ss:thmB-MCII-LS} %
We arrive here with $\alf$ replaced by the quasihyperbolic $N$-chordarc path $\alf'=\chi_{a'c}*\alf[c,d]\star\chi^{-1}_{b'd}$ with endpoints $a',b'\in[a,b]_k$ where $N=4M+3$, $\bp(a')\le\SSS'(\Lam)=4\SSS(\Lam)$, $\bp(b')\le\SSS'(M)=4\SSS(M)$, and all points in $|\alf[a,c]|\cup|\alf[d,b]|$ are $D$-close to $[a,b]_k$ with $D:=D_1\vee D_2$.  It remains to establish that points in $|\alf[c,d]|$ are quasihyperbolically roughly close to $[a,b]_k$.

If either Easy Case~I or Easy Case~II applies to $\alf'$, we are done with the proof.  We claim that we can appeal to Main Case I to assert that all points in $|\alf'|$ are $D'$-close to $[a',b']_k$ with $D':=\bigl(\kk+10\LL(N)\bigr) D_h$.  In order to use Main Case~I with the quasihyperbolic $N$-chordarc path $\alf'$, we need to know that $\bp\le\SSS(N)=10\,\om$ at both endpoints of $\alf'$, where $\om=\om(N)$ is the ABC property parameter for quasihyperbolic $N$-chordarc paths.  Now, $\om=\pi N=\pi (4M+3)>4\pi M=4\nu$, so
\[
  \SSS(N)=10\,\om > 4\cdot10\nu = 4\SSS(M) = \SSS'(M) \ge \bp(b') > \SSS'(\Lam) \ge \bp(a')\,. \quad\text{\HappyFace}
\]

\bibliographystyle{amsalpha} 
\bibliography{mrabbrev,bib}  

\def\cprime{$'$}
\providecommand{\bysame}{\leavevmode\hbox to3em{\hrulefill}\thinspace}
\providecommand{\MR}{\relax\ifhmode\unskip\space\fi MR }
\providecommand{\MRhref}[2]{%
  \href{http://www.ams.org/mathscinet-getitem?mr=#1}{#2}
}
\providecommand{\href}[2]{#2}
\begin{thebibliography}{HPRT10}

\bibitem[Aga68]{Agard-distortion}
S.\ Agard, \emph{Distortion theorems for quasiconformal mappings}, Ann. Acad.
  Sci. Fenn. Ser. A I No. \textbf{413} (1968), 12. \MR{0222288}

\bibitem[BB03]{BB-gromov}
Z.M. Balogh and S.M. Buckley, \emph{Geometric characterizations of {G}romov
  hyperbolicty}, Invent. Math. \textbf{153} (2003), 261--301.

\bibitem[BP78]{BP-beta}
A.F. Beardon and Ch. Pommerenke, \emph{The {P}oincar\'e metric of plane
  domains}, J. London Math. Soc. \textbf{18} (1978), no.~2, 475--483.

\bibitem[BHK01]{BHK-unif}
M.~Bonk, J.~Heinonen, and P.~Koskela, \emph{Uniformizing {G}romov hyperbolic
  spaces}, Ast\'erisque. \textbf{270} (2001), 1--99.

\bibitem[Bon96]{Bonk-quasigdscs}
M.~Bonk, \emph{Quasi-geodesic segments and {G}romov hyperbolic spaces},
  Geometriae Dedicata \textbf{62} (1996), 281--298.

\bibitem[BH99]{Brid-Haef}
M.R. Bridson and A.~Haefliger, \emph{Metric spaces of non-positive curvature},
  Springer-Verlag, Berlin, 1999.

\bibitem[GO79]{GO-unif}
F.W. Gehring and B.G. Osgood, \emph{Uniform domains and the quasi-hyperbolic
  metric}, J. Analyse Math. \textbf{36} (1979), 50--74.

\bibitem[GP76]{GP-qch}
F.W. Gehring and B.P. Palka, \emph{Quasiconformally homogeneous domains}, J.
  Analyse Math. \textbf{30} (1976), 172--199.

\bibitem[Hem79]{Hempel-twice}
J.A. Hempel, \emph{The {P}oincar{\'e} metric on the twice punctured plane and
  the theorems of {L}andau and {S}chottky}, J. London Math. Soc. \textbf{20}
  (1979), 435--445.

\bibitem[Her16]{DAH-univ-cvxty}
D.A.\ Herron, \emph{Universal convexity for quasihyperbolic type metrics},
  Conform. Geom. Dyn. \textbf{20} (2016), 1--24. \MR{3463280}

\bibitem[HMM08]{HMM-mob3}
D.A. Herron, W.~Ma, and D.~Minda, \emph{M\"obius invariant metrics bilipschitz
  equivalent to the hyperbolic metric}, Conform. Geom. Dyn. \textbf{12} (2008),
  67--96.

\bibitem[HPRT10]{HPRT-GrHypDenjoy}
P.~H\"ast\"o, A.~Portilla, J.~Rodr{\'\i}guez, and E.~Tour{\'\i}s, \emph{Gromov
  hyperbolic equivalence of the hyperbolic and quasihyperbolic metrics in
  {D}enjoy domains}, Bull. Lond. Math. Soc. \textbf{42} (2010), no.~2,
  282--294.

\bibitem[Jen81]{Jenkins-LandauII}
J.A. Jenkins, \emph{On explicit bounds in {L}andau's theorem {II}}, Canad. J.
  Math. \textbf{33} (1981), 559--562.

\bibitem[LVV59]{LVV-distortion}
O.~Lehto, K.I.\ Virtanen, and J.~V{\"a}is{\"a}l{\"a}, \emph{Contributions to
  the distortion theory of quasiconformal mappings}, Ann. Acad. Sci. Fenn. Ser.
  A I No. \textbf{273} (1959), 14. \MR{0122990}

\bibitem[MO86]{MO-qhyp}
G.J. Martin and B.G. Osgood, \emph{The quasihyperbolic metric and associated
  estimates on the hyperbolic metric}, J. Analyse Math. \textbf{47} (1986),
  37--53.

\bibitem[Min87]{Minda-LNM}
D.~Minda, \emph{Inequalities for the hyperbolic metric and applications to
  geometric function theory}, Complex Analysis, I (College Park, MD, 1985-1986)
  (Berlin), Lecture Notes in Math., no. 1275, Springer-Verlag, 1987,
  pp.~235--252.

\bibitem[Pom79]{Pomm-unifly-perfect1}
Ch. Pommerenke, \emph{Uniformly perfect sets and the {P}oincar\'e metric},
  Arch. Math. (Basel) \textbf{32} (1979), no.~2, 192--199.

\bibitem[Pom84]{Pomm-unifly-perfect2}
\bysame, \emph{On uniformly perfect sets and fuchsian groups}, Analysis
  \textbf{4} (1984), no.~3--4, 299--321.

\bibitem[SV01]{SV-estimates}
A.Yu. Solynin and M.~Vuorinen, \emph{Estimates for the hyperbolic metric of the
  punctured plane and applications}, Israel J. Math. \textbf{124} (2001),
  29--60.

\bibitem[SV05]{SV-ineqs}
T.~Sugawa and M.~Vuorinen, \emph{Some inequalities for the {P}oincar{\'e}
  metric of plane domains}, Mathematische Zeitschrift \textbf{250} (2005),
  no.~4, 885--906.

\bibitem[V{\"a}i05]{V-hyp-unif}
J.~V{\"a}is{\"a}l{\"a}, \emph{Hyperbolic and uniform domains in {B}anach
  spaces}, Ann. Acad. Sci. Fenn. Math. \textbf{30} (2005), 261--302.

\end{thebibliography}

\end{document} 

paper outline  %
Intro          %
Prelims        %
  basic info   %
  std notation %
Examples       %
Theorems       %
Questions      %


\section{Introduction}  \label{S:Intro} 
\subsection{bla}  \label{s:bla} 
\begin{Thm}  \label{TT:what} 
\end{Thm}                    

\begin{Cor}  \label{CC:..} 
\end{Cor}                  

\begin{thm} \label{T:ttt} %
\end{thm} 
\begin{proof}%
\end{proof}%

\begin{prop} \label{P:ppp} %
\end{prop} 
\begin{proof}%
\end{proof}%

\begin{lma} \label{L:lma} %
\end{lma} 

\begin{fact}  \label{F:ff} %
\end{fact} %
also use for:  remarks(s), examples, defns

\begin{rmks}  \label{R:} 
blah blah blah
\smallskip\noindent(a)
\smallskip\noindent(a)

OR can use below but it gets seriously indented :-((
\begin{itemize}
  \item[ ]
  \item[(a)]
  \item[(b)]
\end{itemize}
\end{rmks} %

\begin{qstn} \label{Q:...} %

figures, pictures, diagrams etc %





This document is organized as follows: Section~\ref{S:Prelims} contains preliminary information including basic definitions and terminology as well as elementary and/or well-known facts.  We exhibit examples and examine ...  In Section~\ref{S:misc} we verify ....